\documentclass[12pt]{amsart}

\IfFileExists{srcltx.sty}{\usepackage{srcltx}}

\numberwithin{equation}{section}

\usepackage[latin1]{inputenc}
\usepackage{xspace,amssymb,amsfonts,euscript}
\usepackage{amsthm,amsmath}
\usepackage{palatino}
\usepackage{euscript}
\usepackage[dvips]{epsfig}
\usepackage{pinlabel}

\usepackage[hmargin=3cm,vmargin=3.5cm]{geometry}

\usepackage{psfrag}

\input xy \xyoption {all}

\usepackage{tikz}
\usetikzlibrary{calc}

\RequirePackage{color}
\definecolor{myred}{rgb}{0.75,0,0}
\definecolor{mygreen}{rgb}{0,0.5,0}
\definecolor{myblue}{rgb}{0,0,0.65}

\RequirePackage{ifpdf}
\ifpdf
  \IfFileExists{pdfsync.sty}{\RequirePackage{pdfsync}}{}
  \RequirePackage[pdftex,
   colorlinks = true,
   urlcolor = myblue, 
   citecolor = mygreen, 
   linkcolor = myred, 
   pagebackref,
   bookmarksopen=true]{hyperref}
\else
  \RequirePackage[hypertex]{hyperref}
\fi

    \def\CM{{\mathbb{C}}}

    \def\FM{{\mathbb{F}}}
  \def\gg{{\mathfrak g}}  
  \def\hg{{\mathfrak h}}  \def\HM{{\mathbb{H}}}

    \def\RM{{\mathbb{R}}}
    
  \def\tg{{\mathfrak t}}

    \def\ZM{{\mathbb{Z}}}

    \def\BC{{\mathcal{B}}}
    \def\CC{{\mathcal{C}}}
    \def\DC{{\mathcal{D}}}
  \def\eb{{\mathbf e}}  
  \def\fb{{\mathbf f}}  \def\FC{{\mathcal{F}}}
  \def\gb{{\mathbf g}}  \def\GC{{\mathcal{G}}}
\def\HB{{\mathbf H}}  \def\hb{{\mathbf h}}  \def\HC{{\mathcal{H}}}

    \def\OC{{\mathcal{O}}}

    \def\SC{{\mathcal{S}}}
    
  \def\ub{{\mathbf u}}  
  \def\vb{{\mathbf v}}  
    
  \def\xb{{\mathbf x}}  \def\XC{{\mathcal{X}}}
  \def\yb{{\mathbf y}}  
  \def\zb{{\mathbf z}}  

\def\FS{{\EuScript F}}
\def\GS{{\EuScript G}}

\def\a{\alpha}
\def\b{\beta}
\def\g{\gamma}
\def\G{\Gamma}
\newcommand{\oG}{\overline{\Gamma}}
\def\d{\delta}

\def\e{\varepsilon}

\def\l{\lambda}

\def\o{\omega}

\def\w{\omega}

\newcommand{\nc}{\newcommand} \newcommand{\renc}{\renewcommand}

\newcommand{\rdots}{\mathinner{ \mkern1mu\raise1pt\hbox{.}
    \mkern2mu\raise4pt\hbox{.}
    \mkern2mu\raise7pt\vbox{\kern7pt\hbox{.}}\mkern1mu}}

\def\un{\underline}

\def\to{\rightarrow}

\def\longto{\longrightarrow}

\nc{\triright}{\stackrel{[1]}{\to}}
\nc{\longtriright}{\stackrel{[1]}{\longto}}

\nc{\Br}{\mathcal{B}}
\nc{\HotRR}{{}_R\mathcal{K}_R}
\nc{\HotR}{\mathcal{K}_R}
\nc{\excise}[1]{}
\nc{\defect}{\text{df}}
\nc{\h}[1]{\underline{H}_{#1}}

\nc{\Ga}{\mathbb{G}_a} 
\nc{\Gm}{\mathbb{G}_m} 

\nc{\Perv}{{\mathbf{P}}}

\nc{\IH}{{\mathrm{IH}}}

\nc{\ic}{\mathbf{IC}}

\nc{\gl}{{\mathfrak{gl}}}
\renc{\sl}{{\mathfrak{sl}}}
\renc{\sp}{{\mathfrak{sp}}}

\nc{\HBM}{H^{BM}}

 \DeclareMathOperator{\Hom}{Hom}
 \DeclareMathOperator{\ch}{ch}
\DeclareMathOperator{\End}{End} 

\DeclareMathOperator{\grk}{grk}

\newtheorem{thm}{Theorem}[section]
\newtheorem{lem}[thm]{Lemma}
\newtheorem{lemma}[thm]{Lemma}

\newtheorem{prop}[thm]{Proposition}
\newtheorem{cor}[thm]{Corollary}
\newtheorem{claim}[thm]{Claim}  
\newtheorem{conj}[thm]{Conjecture}
\newtheorem{assumption}[thm]{Assumption}

\theoremstyle{definition}
\newtheorem{defn}[thm]{Definition}
\newtheorem{constr}[thm]{Construction}
\newtheorem{notation}[thm]{Notation}
\newtheorem{ex}[thm]{Example}

\theoremstyle{remark}
\newtheorem{remark}[thm]{Remark}

\newcommand{\into}{\hookrightarrow}

\def\pt{{\mathrm{pt}}}
\def\op{{\mathrm{op}}}

\def\Mod{\mathrm{Mod}}
\def\Bim{\mathrm{Bim}}

\usepackage{bbm}

\def\R{{\mathbbm R}}
\def\Z{{\mathbbm Z}}

\def\1{\mathbbm{1}}

\newcommand{\ubr}[2]{\underbrace{#1}_{#2}}
\newcommand{\Zvv}{\ZM[v^{\pm 1}]}
\newcommand{\Nvv}{\ZM_{\ge 0}[v,v^{-1}]}
\newcommand{\ot}{\otimes}
\newcommand{\LL}{LL}
\newcommand{\oLL}{\overline{LL}}
\newcommand{\LLL}{\mathbb{LL}}

\newcommand{\LLspangen}{<\LL>^{\textrm{any}}}
\newcommand{\BSBim}{\mathbb{BS}\textrm{Bim}}
\newcommand{\SBim}{\mathbb{S}\textrm{Bim}}
\newcommand{\StdBim}{Std\textrm{Bim}}
\newcommand{\pa}{\partial}
\newcommand{\co}{\colon}
\newcommand{\define}{\stackrel{\mbox{\scriptsize{def}}}{=}}
\newcommand{\sumset}{\stackrel{\scriptstyle{\oplus}}{\scriptstyle{\subset}}}
\newcommand{\std}{\textrm{std}}

\newcommand{\DStd}{\DC^\std}
\newcommand{\FStd}{\FC^\std}
\newcommand{\Kar}{\textrm{Kar}}
\newcommand{\StdF}{\mathcal{S}td}

\newcommand{\ig}[2]
{\vcenter{\xy (0,0)*{\includegraphics[scale=#1]{fig/#2}} \endxy}}
\newcommand{\igc}[2]
{\begin{center} \includegraphics[scale=#1]{fig/#2} \end{center}}
\newcommand{\ige}[2]
{\includegraphics[scale=#1]{fig/#2}}

\newcommand{\simto}{\stackrel{\sim}{\longrightarrow}}

\psfrag{f}[cc][cc]{\small{$f$}}
\psfrag{soff}[cc][cc]{\small{$s(f)$}}
\psfrag{as}[cc][cc]{\small{$\a_s$}}
\psfrag{at}[cc][cc]{\small{$\a_t$}}
\psfrag{parab}[cc][cc]{\small{$\pa_r(\a_b)$}}
\psfrag{JWmm1}[cc][cc]{\small{$JW_{m-1}$}}
\psfrag{pasf}[cc][cc]{\small{$\pa_s(f)$}}
\psfrag{i1}[cc][cc]{\small{$i_1$}}
\psfrag{i2}[cc][cc]{\small{$i_2$}}
\psfrag{p1}[cc][cc]{\small{$p_1$}}
\psfrag{p2}[cc][cc]{\small{$p_2$}}
\psfrag{half}[cc][cc]{\small{$\frac{1}{2}$}}
\psfrag{a}[cc][cc]{\tiny{$JW$}}
\psfrag{JW}[cc][cc]{\small{$JW$}}
\psfrag{b}[cc][cc]{\small{$b$}}
\psfrag{1overa}[cc][cc]{\small{$\frac{1}{\a_s}$}}
\psfrag{1overar}[cc][cc]{\small{$\frac{1}{\a_r}$}}
\psfrag{1overag}[cc][cc]{\small{$\frac{1}{\a_g}$}}
\psfrag{1overab}[cc][cc]{\small{$\frac{1}{\a_b}$}}
\psfrag{1overa2}[cc][cc]{\small{$\frac{1}{\a_s^2}$}}
\psfrag{LLk}[cc][cc]{\small{$\LL_k$}}
\psfrag{LLk-1}[cc][cc]{\small{$\LL_{k-1}$}}
\psfrag{wk-1}[cc][cc]{\small{$w_{k-1}$}}
\psfrag{wk}[cc][cc]{\small{$w_k$}}
\psfrag{phik}[cc][cc]{\small{$\phi_k$}}
\psfrag{s}[cc][cc]{\small{$s$}}
\psfrag{LLkpn}[cc][cc]{\small{$\LL_{k+n}$}}
\psfrag{LLnplwk}[cc][cc]{\small{$\LL_{n+l(w_k)}$}}
\psfrag{ab2ar}[cc][cc]{\small{$\a_b^2\a_r$}}
\psfrag{abararb}[cc][cc]{\small{$\a_b\a_r(\a_b+\a_r)$}}
\psfrag{alpha}[cc][cc]{\small{$\a$}}
\psfrag{beta}[cc][cc]{\small{$\b$}}
\psfrag{gamma}[cc][cc]{\small{$\gamma$}}
\psfrag{LLspan}[cc][cc]{\small{$\LLspangen$}}
\psfrag{lenm1}[cc][cc]{\small{$\le n-1$}}
\psfrag{isin}[cc][cc]{\small{$\in$}}
\psfrag{len}[cc][cc]{\small{$\le n$}}

\psfrag{1over2}{\Large{$\frac{1}{[2]}$}}
\psfrag{1over3}{\Large{$\frac{1}{[3]}$}}
\psfrag{2over3}{\Large{$\frac{[2]}{[3]}$}}
\psfrag{1overx}{\Large{$\frac{1}{x}$}}
\psfrag{1overxym1}{\Large{$\frac{1}{xy-1}$}}
\psfrag{xoverxym1}{\Large{$\frac{x}{xy-1}$}}
\psfrag{yoverxym1}{\Large{$\frac{y}{xy-1}$}}
\psfrag{JW1}{$JW_1$}
\psfrag{JW2}{$JW_2$}
\psfrag{JW2sm}{\tiny{$JW_2$}}
\psfrag{JW3}{$JW_3$}

\psfrag{fgtofgeqfsg}[cc][cc]{\small{$f \ot g \mapsto f \ot g = fs(g) \ot 1$}}
\psfrag{1tomess}[cc][cc]{\small{$1\ot 1 \mapsto \d \ot 1 - 1 \ot \d$}}
\psfrag{m1m}[cc][cc]{\small{$(-1)^m$}}
\psfrag{1overrho}[cc][cc]{\small{$\frac{1}{\rho}$}}

\psfrag{gl}[cc][cc]{\small{$g_l$}}
\psfrag{psi}[cc][cc]{\small{$\psi$}}
\psfrag{phiM}[cc][cc]{\small{$\phi_M$}}
\psfrag{eta}[cc][cc]{\small{$\eta$}}
\psfrag{xb}[cc][cc]{\small{$\xb$}}
\psfrag{yb}[cc][cc]{\small{$\yb$}}
\psfrag{zb}[cc][cc]{\small{$\zb$}}
\psfrag{Mm1}[cc][cc]{\small{$M-1$}}
\psfrag{lk}[cc][cc]{\small{$< k$}}
\psfrag{k}[cc][cc]{\small{$k$}}

\psfrag{lll}[cc][cc]{\small{$\lambda$}}
\psfrag{lli}[cc][cc]{\small{$\lambda^{-1}$}}

\begin{document}

\title{Soergel calculus}

\author{Ben Elias} \address{Massachusetts Institute of Technology,
  Boston, USA}
\email{belias@mit.edu}
\author{Geordie Williamson} \address{Max-Planck-Institut f\"ur
  Mathematik, Bonn, Germany}
\email{geordie@mpim-bonn.mpg.de}

\maketitle

\begin{center}
  \emph{\small{ To Mikhail Khovanov and Rapha\"el Rouquier, who taught us generators and relations.}}
\end{center}

\begin{abstract}
The monoidal category of Soergel bimodules is an
  incarnation of the Hecke category, a fundamental object in
  representation theory. We present
  this category by generators and relations, using the language of
  planar diagrammatics. We show that Libedinsky's light leaves give a
  basis for morphism spaces and give
  a new proof of Soergel's classification of the indecomposable
  Soergel bimodules.
\end{abstract}


\section{Introduction}
\label{sec-intro}

Let us recall the history of the Hecke algebra from the perspectives of algebraization and categorification.

\subsection{The Hecke algebra by generators and relations}

Let $G$ be a split finite reductive group over a finite field $\mathbb{F}_q$, and let $B \subset G$ be a Borel subgroup.  A fundamental object in representation theory is the Hecke algebra
\[
\mathrm{Fun}_{B \times B}(G, \mathbb{C})
\]
of $B$-biinvariant complex valued functions on $G$, with multiplication given by convolution. This algebra first emerged when studying the irreducible complex characters of $G$, but has gone on to play an essential role in many (at times unexpected) branches of representation theory.

Iwahori \cite{Iwahori} made the crucial observation that the Hecke algebra admits a description which is ``independent" of the size $q$ of the base field and only depends on the Weyl group.
Fix a maximal split torus $T \subset B$ and let $(W,S)$ denote the Weyl group and its simple reflections. Using the Bruhat decomposition
\[
G = \bigsqcup_{w \in W} BwB
\]
it follows that the Hecke algebra has a basis given by indicator functions of the subsets $BwB \subset G$. Now let $\HB$ be the free $\Zvv$-module with basis $\{ T_w \; | \; w \in W \}$. There is a unique algebra structure on $\HB$ determined by
\[
T_wT_s = \begin{cases} T_{ws} & \text{if $ws > w$,} \\ (v^{-2}-1)T_w + v^{-2}T_{ws} & \text{if $ws < w$}. \end{cases}
\]
Writing $\HB_{\FM_q} \define \HB \otimes \CM$ for the specialization of $\HM$ at $v^{-1} \mapsto \sqrt{|\FM_q|} \in \CM$, we have an isomorphism of algebras
\[
\HB_{\FM_q} \stackrel{\sim}{\longto} \mathrm{Fun}_{B \times B}(G, \mathbb{C})
\]
sending $T_w$ to the indicator function of $BwB \subset G$. (One could define $\HB$ over $\ZM[q^{\pm 1}]$. The introduction of a square root of $q$ is a notational convenience which becomes important later.)
Furthermore, the algebra $\HB$ is generated by the elements $\{T_s \; | \; s \in S \}$, modulo the quadratic relation \[ (T_s + 1)(T_s - v^{-2}) = 0\] and the braid relations
\[ \ubr{T_sT_t \dots }{m_{st}} = \ubr{T_tT_s \dots}{m_{st} }
\] where $m_{st}$ denotes the order of $st$ in $W$.

This presentation of the Hecke algebra has paved the way for an algebraic study of its representation theory. An immediate consequence is that the Hecke algebra may be defined for any
Coxeter system, whether or not it arises as the Weyl group of a reductive algebraic group (or suitable generalizations such as an affine or Kac-Moody groups). Thus was the Hecke algebra
freed from its concrete realization as a convolution algebra.

\subsection{The Hecke category}

Beginning with the seminal work \cite{KaL1} of Kazhdan and Lusztig it was realised that the Hecke algebra admits a categorification, which has come to be known as the Hecke category. 
 According to Grothendieck's function-sheaf dictionary, the algebra of $B$-biinvariant functions on $G$ should be categorified by some version of $B$-biequivariant sheaves on $G$. For concreteness, we now suppose that $G$ is a complex reductive group with Borel subgroup $B$ and maximal torus $T \subset G$. The Hecke category (in its simplest geometric incarnation) is the additive subcategory of semi-simple complexes
\[
\HC \subset D^b_{B \times B}(G, \CM)
\]
in the equivariant derived category of $B$-biequivariant sheaves on $G$. In other words, the objects of $\HC$ are direct sums of shifts of various $\ic_w \define \ic(\overline{BwB})$, the equivariant intersection cohomology complexes of $B \times B$-orbits. There is a monoidal structure $*$ on $D^b_{B \times B}(G,\CM)$ given by convolution, and it preserves $\HC$, thanks to the Decomposition Theorem and the compactness of $G/B$. Therefore the split Grothendieck group $[\HC]$ of $\HC$ has a $\Zvv$-algebra structure (the $\ZM[v^{\pm 1}]$ structure is given by $v[\FC] = [\FC[1]]$). The key result is an isomorphism of $\Zvv$-algebras
\[
\HB \stackrel{\sim}{\longto} [\HC]
\]
which sends the Kazhdan-Lusztig basis element $\underline{H}_w \in \HB$ to $[\ic_w]$. The Kazhdan-Lusztig basis also has a purely algebraic definition, and it was quickly realised that this definition mimics the defining properties of an intersection cohomology complex.

\subsection{Soergel bimodules}

Just as Iwahori gave an intrinsic construction of the Hecke algebra so too would Soergel give one of the Hecke category.

Let $\hg$ denote the Lie algebra of $T$ and $R$ the regular functions on $\hg$, graded so that $\hg^*$ is in degree two. We have a canonical identification of equivariant cohomology groups
(the ``Borel isomorphism'') \[ R = H^{\bullet}_T(\pt) = H^{\bullet}_B(\pt). \] In particular, the hypercohomology of any object in $ D^b_{B \times B}(G, \CM)$ is naturally a graded module
over $H^{\bullet}_{B \times B}(\pt) = R \otimes_{\CM} R$. Because $R$ is commutative, we may regard any $R \otimes R$-module as an $R$-bimodule. Hence hypercohomology can be see as a functor to 
$R-\Bim$, the category of graded $R$-bimodules.

Soergel's first key observation is that hypercohomology \[ \HM^{\bullet}_{B \times B} : \HC \to R-\Bim \] is fully-faithful and monoidal (that is $\HM^{\bullet}_{B \times B}(\FS * \GS) \cong \HM^{\bullet}_{B \times B}(\FS) \otimes_R \HM^{\bullet}_{B \times B}(\GS)$ for $\FS, \GS \in \HC$). It follows that the Hecke category is
equivalent to its essential image. His second key observation is that the Decomposition Theorem gives an alternative description of the intersection cohomology complexes. That is,
$\ic_w$ is the unique summand of $\ic_s * \ic_t * \cdots * \ic_u$ (for a reduced expression $st\ldots u$ of $w$) which does not appear in the analogous convolution for any shorter reduced
expression. This description arises from the Bott-Samelson resolution of a Schubert variety, and so we call it the ``Bott-Samelson description" of an intersection cohomology complex.

Note that $W$ acts on $\hg$ and hence on $R$ via graded algebra automorphisms. It is easy to calculate that
\[
\HM^{\bullet}_{B \times B} (\ic_s) = B_s \define R \otimes_{R^s} R(1)
\]
where $R^s \subset R$ denotes the subalgebra of $s$ invariants in $R$, and $(1)$ denotes the grading shift. From this, Soergel obtained the following elementary description of $\HC$: it is equivalent to the smallest full additive monoidal Karoubian graded subcategory of $R-\Bim$ containing $B_s$ for all $s \in S$. This category is by definition the category $\SBim$ of \emph{Soergel bimodules}. By the above discussion, hypercohomology yields an equivalence of graded monoidal categories:
\[
\HC \stackrel{\sim}{\longrightarrow} \SBim.
\]


In this setting, $\SBim$ is just another incarnation of the Hecke category. However, as Soergel pointed out, this algebraic description allows one to define the Hecke category for arbitrary Coxeter systems, for which there is no suitable geometric context. In \cite{Soe3} Soergel imitates the above definition of $\SBim$ starting with an appropriate (``reflection faithful'') representation $\hg$ of $W$, which plays the role of the representation of $W$ on the Lie algebra $\hg$ of $T$. With $\SBim$ defined as above, Soergel then constructs an isomorphism of $\Zvv$-algebras
\[
\HB \stackrel{\sim}{\longto} [\SBim].
\]
In analogy to the Bott-Samelson description of intersection cohomology complexes in $\HC$, Soergel proves that the indecomposable bimodules $\{B_w\}$ in $\SBim$ are in bijection with $W$, and that $B_w$ is the unique summand of $B_s \ot B_t \ot \cdots \ot B_u$ (for a reduced expression of $w$) which does not appear for a shorter expression. These results are known as Soergel's Categorification Theorem.

For a Weyl group, one can prove Soergel's Categorification Theorem easily by transferring known facts about $\HC$ to $\SBim$ using hypercohomology. Soergel's proof for the general case is much trickier, but relies only on commutative algebra. Soergel proves his results for reflection faithful representations of a Coxeter system over an infinite field of characteristic $\ne 2$.

Soergel's theory (or the Bott-Samelson description of $\HC$) states that the objects of $\SBim$ are generated by the objects $B_s$. Moreover, there are isomorphisms between objects in $\SBim$ which lift the quadratic and braid relations of the Hecke algebra. (One should not categorify the Iwahori presentation given above, but a presentation using the Kazhdan-Lusztig generators which we will describe below.)  Heuristically speaking, this is the categorical analogue of Iwahori's algebraization of $\HB$, on the level of objects. However, in $\HC$ or $\SBim$ there is a whole new layer of structure, with no analogue in the Hecke algebra: the composition of morphisms.

\subsection{Soergel bimodules by generators and relations}

A \emph{Bott-Samelson bimodule} is a bimodule of the form $B_{\un{w}} \define B_s \ot_R B_t \ot_R \cdots \ot_R B_u$ for an expression $\un{w}=st\dots u$. They form a full
monoidal subcategory of $R-\Bim$, which we denote $\BSBim$. By definition, any Soergel bimodule is a direct sum of shifts of summands of bimodules in $\BSBim$. Said another way, the
category of Soergel bimodules $\SBim$ is the Karoubi envelope of (the additive, graded envelope of) $\BSBim$. The upshot is that in order to describe the category of Soergel bimodules it
is enough to describe the monoidal category of Bott-Samelson bimodules. This is an easier problem because (in contrast to Soergel bimodules) one has concrete combinatorial realizations of
the objects in $\BSBim$.

In this paper we describe the monoidal category of Bott-Samelson bimodules by
generators and relations. (There is one caveat, involving standard parabolic subgroups of type $H_3$ which shall be discussed in section \ref{threecolorintro} below.) Such a description has already been given by Libedinsky
\cite{LibRA} in the right-angled case (i.e. when $m_{st} \in \{2,\infty\}$), in type $A$ by the first author and Khovanov \cite{EKh}, and in dihedral type by the first author in
\cite{EDihedral}. This is the next step in the algebraization of $\HC$, freeing the category of Soergel bimodules from its realization as a full subcategory of a bimodule category. Said
another way, in this paper we give a $2$-presentation of $\BSBim$, in analogy to the $1$-presentation of $\HB$ given by Iwahori. Our presentation will use the technology of planar
diagrammatics.

We assign a color to each element of $S$, which allows us to encode a Bott-Samelson bimodule as a sequence $\un{w}$ of colored dots ordered on
a line. A morphism between Bott-Samelson bimodules will be given by a linear combination of isotopy classes of decorated graphs embedded in the planar strip $\RM \times [0,1]$. The edges
in these graphs will be colored, and may run into the bottom boundary $\R \times \{0\}$ or the top boundary $\R \times \{1\}$, yielding a sequence of colored dots on each boundary. A
morphism from $B_{\un{w}}$ to $B_{\un{y}}$ will have bottom boundary $\un{w}$ and top boundary $\un{y}$. For example, the following planar diagram
\igc{1}{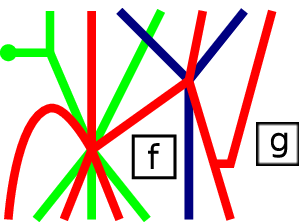}
represents a map from $rgrgrgbr$ to $grbgrbr$.

For the moment, let us ignore the notion of isotopy classes of graphs, and consider instead diagrams which can be constructed from horizontal and vertical concatenation of the following generators. Here is a list of generating morphisms, their degrees and the maps of Soergel bimodules which they represent:
\[
\begin{array}{cccc}
\begin{array}{c}\tikz[scale=0.5]{\draw[dashed] (3,0) circle (1cm);
\draw[color=blue] (3,-1) to (3,0);
\node[circle,fill,draw,inner sep=0mm,minimum size=1mm,color=blue] at (3,0) {};}\end{array}& \text{deg 1} & B_{\color{blue}s} \to R  & f \otimes g \mapsto fg 
\\
\begin{array}{c}\tikz[scale=0.5]{\draw[dashed] (3,0) circle (1cm);
\draw[color=blue] (3,1) to (3,0);
\node[circle,fill,draw,inner sep=0mm,minimum size=1mm,color=blue] at (3,0) {};}\end{array}
& \text{deg 1}&
R \to B_{\color{blue}s}  & 1 \mapsto \frac{1}{2}(\alpha_s \otimes 1 + 1
  \otimes \alpha_s)
\\
\begin{array}{c}\tikz[scale=0.5]{\draw[dashed] (0,0) circle (1cm);
\draw[color=blue] (-30:1cm) -- (0,0) -- (90:1cm);
\draw[color=blue] (-150:1cm) -- (0,0);}\end{array}
&  \text{deg -1} & 
B_{\color{blue}s} B_{\color{blue}s} \to B_{\color{blue}s}  & 1 \otimes g \otimes 1 \mapsto \partial_s g
  \otimes 1
\\

\begin{array}{c}\tikz[scale=0.5]{\draw[dashed] (0,0) circle (1cm);
\draw[color=blue] (30:1cm) -- (0,0) -- (-90:1cm);
\draw[color=blue] (150:1cm) -- (0,0);}\end{array}
&  \text{deg -1} & B_{\color{blue}s} \to B_{\color{blue}s} B_{\color{blue}s}  & 1 \otimes 1 \mapsto 1 \otimes 1
  \otimes 1 \\
\begin{array}{c}\tikz[scale=0.5]{\draw[dashed] (0,0) circle (1cm);
\node at (0,0) {$f$};}\end{array}
&  \text{deg f} &
R \to R  & 1 \mapsto f
\\
\begin{array}{c}\tikz[scale=0.5]{\draw[dashed] (0,0) circle (1cm);
\draw[color=blue] (0,0) -- (22.5:1cm);
\draw[color=blue] (0,0) -- (67.5:1cm);
\draw[color=blue] (0,0) -- (112.5:1cm);
\draw[color=blue] (0,0) -- (157.5:1cm);
\draw[color=blue] (0,0) -- (-22.5:1cm);
\draw[color=blue] (0,0) -- (-67.5:1cm);
\draw[color=blue] (0,0) -- (-112.5:1cm);
\draw[color=blue] (0,0) -- (-157.5:1cm);
\draw[color=red] (0,0) -- (0:1cm);
\draw[color=red] (0,0) -- (45:1cm);
\draw[color=red] (0,0) -- (90:1cm);
\draw[color=red] (0,0) -- (135:1cm);
\draw[color=red] (0,0) -- (180:1cm);
\draw[color=red] (0,0) -- (-45:1cm);
\draw[color=red] (0,0) -- (-90:1cm);
\draw[color=red] (0,0) -- (-135:1cm);
}\end{array}
&  \text{deg 0}&  \ubr{B_{\color{blue}s}B_{\color{red}t} \dots }{m_{st}}  \to \ubr{B_{\color{red}t} B_{\color{blue}s} \dots }{m_{st}}    & 
\end{array}
\]
In the above, $\alpha_s$ denotes a fixed choice of equation for the hyperplane fixed by $s$, and $\pa_s$ denotes the Demazure operator $\pa_s(f) \define (f - sf)/\a_s$. We refer to the first two morphisms as \emph{dots}, the second two morphisms as \emph{trivalent vertices}, and the final morphism as the \emph{$2m_{st}$-valent vertex}.

We have not given a formula for the $2m_{st}$-valent vertex, as it is both difficult and unenlightening to write down explicitly in general. It can be described conceptually as follows. Let $B_{s,t}$ denote the indecomposable Soergel bimodule indexed by the longest element of the (finite) rank two parabolic subgroup generated by $s$ and $t$.  The bimodules $B_{\color{blue}s}B_{\color{red}t} \dots $ and $B_{\color{red}t} B_{\color{blue}s} \dots $ each contain $B_{s,t}$ as a summand with multiplicity one. The $2m_{st}$-valent vertex is the projection and inclusion of this common summand. (This determines the morphism up to a scalar, and there is a simple way to fix this choice of scalar.)

It is a result due to Libedinsky \cite{Lib} that these morphisms generate all morphisms between Bott-Samelson bimodules. In this paper we determine the relations which they satisfy. Let us
call a subset $J \subset S$ \emph{finitary} if the corresponding parabolic subgroup is finite. In the Iwahori presentation of $\HB$ for a Coxeter system, there is a generator
$T_s$ for each $s \in S$, i.e. for each finitary subset of rank $1$; there is a quadratic relation for each finitary subset of rank $1$, and a braid relation for each finitary subset of
rank $2$. In our $2$-presentation of $\BSBim$, the generating objects $B_s$ are associated to $s \in S$ (i.e. finitary subsets of rank $1$), the generating morphisms are associated to
finitary subsets of size $\le 2$, and the relations are associated to finitary subsets of of size $\le 3$.

Because the simple reflections are encoded by colors, we refer to a relation as
a \emph{one}, \emph{two} or \emph{three color relation}, depending on the size of the subset of $S$ involved. As one might guess, the relations become more complicated as the number of
colors increases. Here is a description of the relations:

\subsubsection{One color relations:} 

It was pointed out in \cite{EKh} that most of the one color relations can be concisely encoded in the statement that $B_s$ is a Frobenius object in the
category of $R$-bimodules. The trivalent vertices give the multiplication and comultiplication, whilst the dots provide the unit and counit. The Frobenius biadjunction of $B_s$ with itself arises from certain cups and caps, which are constructed from trivalent vertices and dots. The axioms governing Frobenius objects guarantees that any diagram involving one color is isotopy invariant. The one-color relations not involving polynomials are then the following:
\begin{gather*} \ig{1}{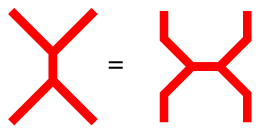} \qquad \qquad \ig{1}{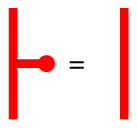} \qquad \qquad \ig{1}{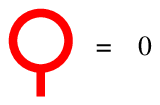} \end{gather*}
In addition to the Frobenius relations, one has the following relations governing the interaction of colors and polynomials:
\begin{gather*}
\begin{array}{c}
\tikz[scale=0.9]{
\draw[dashed,gray] (0,0) circle (1cm);
\draw[color=red] (0,0.5) -- (0,-0.5);
\node[draw,fill,circle,inner sep=0mm,minimum size=1mm,color=red] at (0,0.5) {};
\node[draw,fill,circle,inner sep=0mm,minimum size=1mm,color=red] at (0,-0.5) {};
} \end{array}
=
\begin{array}{c}
\tikz[scale=0.9]{
\draw[dashed,gray] (0,0) circle (1cm);
\node at (0,0) {$\alpha_s$};
} \end{array}, \\
\begin{array}{c}
\tikz[scale=0.9]{
\draw[dashed,gray] (0,0) circle (1cm);
\draw[color=red] (0,-1) to (0,1);
\node at (-0.5,0) {$f$};
} \end{array}
=
\begin{array}{c}
\tikz[scale=0.9]{
\draw[dashed,gray] (0,0) circle (1cm);
\draw[color=red] (0,-1) to (0,1);
\node at (0.5,0) {${\color{red}s}(f)$};
} \end{array}
+
\begin{array}{c}
\tikz[scale=0.9]{
\draw[dashed,gray] (0,0) circle (1cm);
\draw[color=red] (0,-1) to (0,-0.5);
\draw[color=red] (0,1) to (0,0.5);
\node[color=red,draw,fill,circle,inner sep=0mm,minimum size=1mm] at (0,-0.5) {};
\node[color=red,draw,fill,circle,inner sep=0mm,minimum size=1mm] at
(0,0.5) {};
\node at (0,0) {$\partial_{\color{red}s} f$};
} \end{array}.
\end{gather*}

\subsubsection{Two color relations:} Two color graphs and Soergel bimodules for the dihedral group are explored in detail in \cite{EDihedral}. Essentially, morphisms between Bott-Samelson bimodules in rank 2 are governed by the Temperley-Lieb algebra at a root of unity, a fact related to the (quantum) geometric Satake equivalence for $\mathfrak{sl}_2$.

The first important two color relation is the cyclicity of the $2m_{st}$-valent vertex. (There is a subtlety here if the Cartan matrix is not symmetric, which we ignore in the introduction.) This allows us to consider all morphisms as isotopy classes of diagrams. The second relation is the so-called \emph{two color associativity} (here $m = m_{{\color{red}s}{\color{blue}t}}$):
\[
\includegraphics[width=10cm]{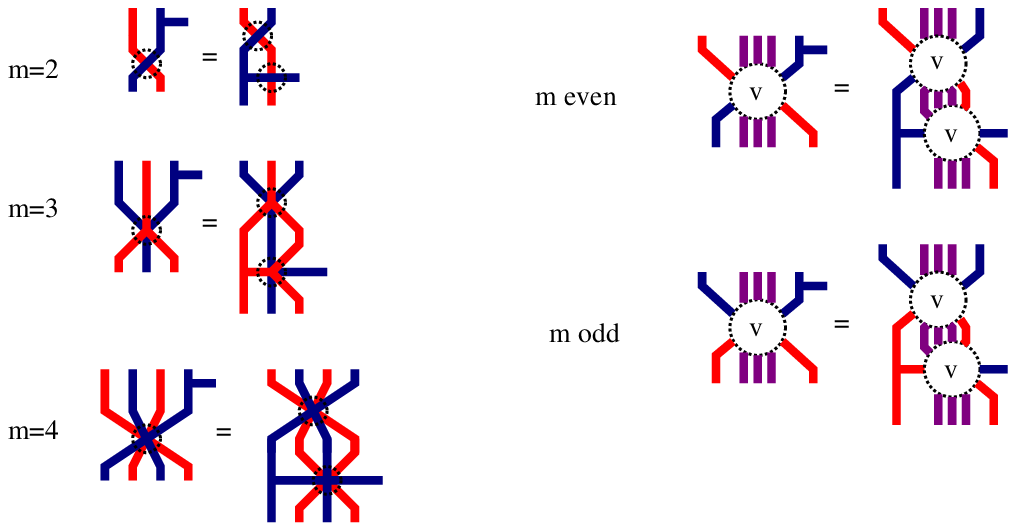}
\]
It allows one to ``pull'' a trivalent vertex through a $2m_{st}$-valent vertex.

The third relation allows one to expand the composition of a dot and a $2m_{st}$-valent vertex into a linear combination of diagrams in which the $2m_{st}$-valent vertex does not occur. This relation is best understood using Jones-Wenzl projectors, as explained in \cite{EDihedral}, and is difficult to state without developing this machinery. (For example, for a Weyl group of type $G_2$, 42 terms occur). Here we give examples for finite parabolic subgroups of types $A_1 \times A_1$, $A_2$ and $B_2$, i.e. the cases $m_{st} = 2,3,4$:
\[
\includegraphics[width=10cm]{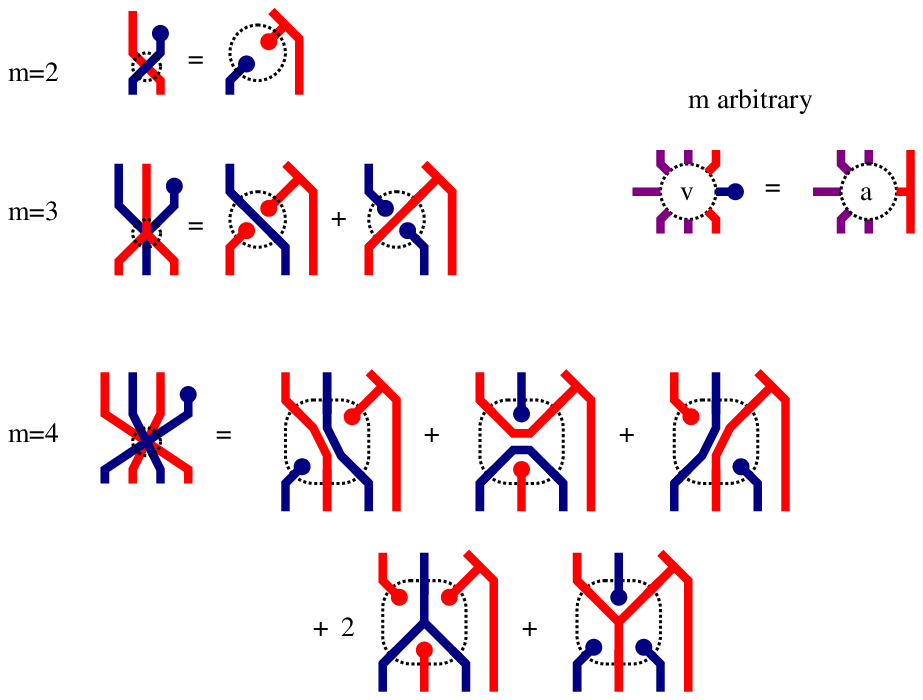}
\]
We hope that the reader has not missed the appearance of a 2 in the relation for $B_2$. In general, these coefficients are polynomials in the entries of the Cartan matrix of the corresponding root system.

Together, the two-color relations imply that the composition of two $2m_{st}$-valent vertices is an idempotent endomorphism (corresponding to the projection to $B_{s,t}$ inside the Bott-Samelson bimodule):
\[
\includegraphics[width=10cm]{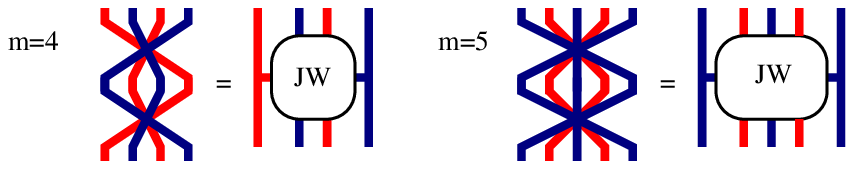}
\]

\subsubsection{Three color or ``Zamolodchikov'' relations:} \label{threecolorintro} There is one relation for each finite parabolic subgroup of rank 3, generalizing the so-called
Zamolodchikov tetrahedron equation. We feel this interesting topic deserves some introduction of its own. In braided monoidal categories a
fundamental role is played by the Yang-Baxter equation or braid relation, which guarantees that one obtains an action of the braid group on the tensor powers of any object. In the setting
of braided monoidal 2-categories, the role of the Yang-Baxter equation is played by the Yang-Baxter isomorphism, and the consistency relation between these isomorphisms is known as the
\emph{Zamolodchikov tetrahedron equation.}\footnote{There are also higher Zamolodchikov relations governing braided monoidal $n$-categories. These will not be considered in this paper.} Instead of describing this theory in its original context, we give a description using the combinatorics of Coxeter groups.

Consider a Coxeter system $W$ of type $A_3$ with simple reflections $s,t,u$ such that $s$ and $u$ commute. The vertices of the following graph encode the reduced expressions for the
longest element $w_0$ of $W$, and the edges indicate the application of a braid relation (the dashed lines correspond to the ``boring'' braid relations $su \leftrightarrow us$):
\igc{1}{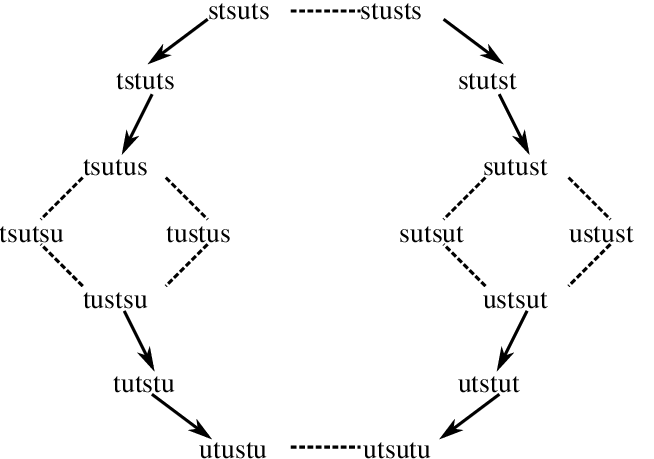} In the setting of braided monoidal 2-categories each vertex encodes a 1-morphism, and each edge gives a 2-morphism between these 1-morphisms (or rather, each edge
gives a pair of inverse 2-isomorphisms), so that a path gives a 2-morphism by composition. It is an easy consequence of the axioms for braided monoidal
2-categories that the two dashed squares commute. The Zamolodchikov tetrahedron equation is the requirement that the 2-morphism obtained by following a non-trivial loop around this graph
is the identity 2-morphism.

For any element $w \in W$ of type $A_n$, one can draw a ``reduced expression graph" as above, and obtain a 2-morphism for any path in this graph. The Zamolodchikov relation is already
sufficient to imply that any non-trivial loop yields the identity morphism, which explains its great importance. In fact, for any element of any Coxeter group, the loops in its reduced
expression graph are (in a suitable sense) generated by the loops in the reduced expression graph of the longest element in any finite parabolic subgroup of rank 3. Thus, in addition to the $A_3$ Zamolodchikov relation discussed above, there are Zamolodchikov-style relations in type $B_3$, $H_3$, and $A_1 \times I_2(m)$ for $m < \infty$ (though in type $A_1 \times I_2(m)$ things are not very exciting, which is why we ignored $A_1 \times A_2$ and $A_1 \times A_1 \times A_1$ in the type $A_n$ discussion earlier).

Let us explain how analogous relations arise for morphisms between Bott-Samelson bimodules. A vertex $\un{w}$ of a reduced expression graph is associated with a
Bott-Samelson bimodule $B_{\un{w}}$, and edges give morphisms ($2m_{st}$-valent vertices) between these bimodules. Unlike the Yang-Baxter situation, the edges are not isomorphisms (unless
$m_{st}=2$), but are only projections to a common summand, so that one should not expect a loop to be equal to the identity. For general reasons (which will be discussed later in this
paper), two paths with the same start and endpoint will be equal ``modulo lower terms," i.e. modulo morphisms which factor through $B_{\un{y}}$ for a sequence $\un{y}$ strictly shorter than $\un{w}$.

What is miraculous (and currently lacking a satisfying explanation) is that, for $A_3$ and $B_3$ (and more trivially, for $A_1 \times I_2(m)$), one can choose an ``orientation" on the
reduced expression graph of the longest element, such that the two paths from source to sink yield morphisms in $\BSBim$ which are equal on the nose! For example, we have placed the orientation on the non-dashed edges in the $A_3$ graph above. Tracing out the two morphisms from the sink to source in the above graph yields the $A_3$
Zamolodchikov relation: 

\igc{1}{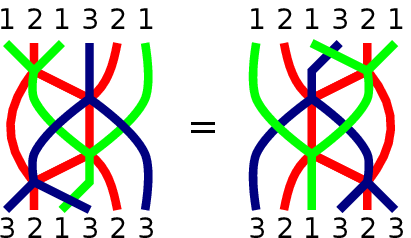}

Entirely analogously one has the following Zamolodchikov relation for $B3$:
\begin{gather*} \ig{1.5}{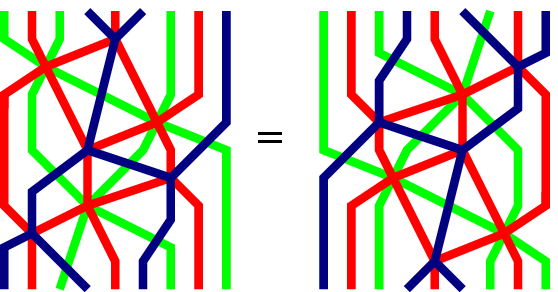} \end{gather*}

For the $A_3$ graph, there are many orientations which have a unique source (up to dashed edges) and a unique sink, but the orientation chosen above (resp. the reverse orientation) is
special. For any other choice of source and sink, the morphisms attached to the two oriented paths from source to sink need not be equal in $\BSBim$, their difference being a nontrivial
sum of lower terms. That there is a ``canonical" (and an ``anti-canonical") choice of orientation on a reduced expression graph for any element of any type $A$ Coxeter group is an old
result of Manin-Schectman \cite{MS}, and the implications of this for morphisms between Bott-Samelson bimodules have been explored in \cite{EInd}. However, the relationship between
Manin-Schectman theory and Soergel bimodules is not understood.

Let us quickly mention the $A_1 \times I_2(m)$ Zamolodchikov relations. The reduced expression graph of the longest element only has one choice of orientation (with its reverse), and it
yields the following equality of morphisms in $\BSBim$:
\[
 \ig{1}{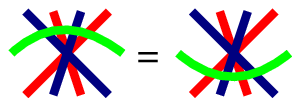}
\]
For the $H_3$ graph, on the other hand, computer calculations have verified that there is no suitable choice of source and sink for the reduced expression graph of the longest element. In other words, two distinct paths will always differ by a nontrivial sum of lower terms. There is some relation of the form
\[
\ig{1}{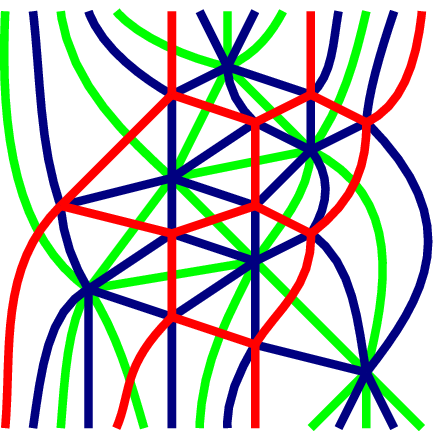}
- 
\ig{1}{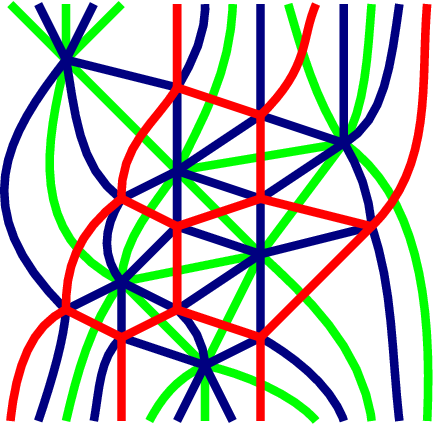}
= 
\text{lower terms}
\]
however, despite considerable effort, we have not been able to compute the lower terms which appear. The question of what these lower terms are could in principle be decided by computer, however the computation is impossible with our current algorithms and technology. This is the caveat mentioned earlier: we do not have a completely explicit presentation of the category $\BSBim$ when $W$ contains a parabolic subgroup of type $H_3$, knowing this Zamolodchikov relation only in the rough form above.


It is surprising that the analogues of the Zamolodchikov relation hold in all finite rank 3 groups except type $H_3$. We do not know a good reason why this is the case. We also do not have
a good conceptual understanding of why certain paths in reduced expression graphs lead to relations which hold in Soergel bimodules, and others do not.

\subsection{Consequences for the structure of Soergel bimodules}

Let $\DC$ denote the diagrammatic category defined by generators and relations in the previous section. In order to prove that this category is equivalent to $\BSBim$, we construct a basis
for morphisms in $\DC$, which is sent to a basis for morphisms in $\BSBim$. In \cite{Lib}, Libedinsky constructed a combinatorial basis for the morphism space between two Bott-Samelson
bimodules, which he called the \emph{light leaves} basis. His construction involves the technique of localization, which we will also explore diagrammatically in this paper.

Let $\un{x}$ be an expression and $\eb$ a subexpression expressing $w$. That is if $\un{x} = s_1 \dots s_m$ then $\eb$ is a sequence $\eb_1 \dots \eb_m$ of 0's and 1's such that $w = s_1^{\eb_1} \dots s_m^{\eb_m}$.
To this pair Libedinsky assigns a morphism $\LL_{\un{x},\eb} : B_{\un{x}} \to B_{\un{w}}$ where $\un{w}$ is a reduced expression for $w$. Libedinsky's definition of $\LL_{\un{x}, \eb}$ is inductive and at each step may involve a choice of reduced expression and as well as a sequence of braid relations to reach such an expression. So, though beautiful, it is highly non-canonical. We  translate Libedinsky's construction into our diagrammatic language, in which case a light leaf morphism can be thought of as a kind of rhombus (see section \ref{sec-LLL})
\[
\begin{array}{c}
\begin{tikzpicture}[scale=0.7]
\draw (-0.3,-2) -- (-2,-2) -- (-1,0) -- (-0.3,0);
\draw (0.3,-2) -- (2,-2) -- (1,0) -- (0.3,0);
\node (top) at (0,0) {$\un{w}$};
\node (mid) at (0,-1) {$\LL_{\un{x}, \eb}$};
\node (bot) at (0,-2) {$\un{x}$};
\end{tikzpicture}
\end{array}
\]
where the upper expression is always reduced.

Now let $\un{x}$ and $\un{y}$ be fixed expressions. 
Following Libedinsky, we introduce the \emph{double leaves} morphism
\[
\LLL_{w,\eb,\fb} \define 
\begin{array}{c}
\begin{tikzpicture}[scale=0.7]
\draw (-0.3,-2) -- (-2,-2) -- (-1,0) -- (-0.3,0);
\draw (0.3,-2) -- (2,-2) -- (1,0) -- (0.3,0);
\draw (-0.3,2) -- (-2,2) -- (-1,0) -- (-0.3,0);
\draw (0.3,2) -- (2,2) -- (1,0) -- (0.3,0);
\node (top) at (0,2) {$\un{y}$};
\node (mid) at (0,1) {$\oLL_{\un{y}, \fb}$};
\node (top) at (0,0) {$\un{w}$};
\node (mid) at (0,-1) {$\LL_{\un{x}, \eb}$};
\node (bot) at (0,-2) {$\un{x}$};
\end{tikzpicture}
\end{array}
\]
where $\oLL_{\un{y}, \fb}$ denotes the vertical flip of a light leaf morphism, and $\eb$ (resp. $\fb$) is a subexpression of $\un{x}$ (resp. $\un{y}$) which expresses $w$. Hence $\LLL_{w,\eb,\fb}$ is a morphism $B_{\un{x}} \to B_{\un{y}}$. Consider the set
\[
\LLL_{\un{x}, \un{y}} \define \bigcup_{w \in W} \{ \LLL_{w,\eb,\fb} \; | \; \eb \in M(\un{x}, w), \fb \in M(\un{y}, w)\}
\]
where, $M(\un{x}, w)$ denotes the set of subexpressions of $\un{x}$ expressing $w$. Our main theorem is then the following:

\begin{thm} $\LLL_{\un{x},\un{y}}$ is an $R$-basis for $\Hom_{\DC}(B_{\un{x}}, B_{\un{y}})$. \end{thm}

The proof of this theorem is pure diagrammatic algebra. Unfortunately it is quite convoluted. It has the following consequences:
\begin{enumerate}
\item Under very general assumptions one has a functor $\DC \to \SBim$ (we defined this functor earlier under the assumption that $2$ was invertible).  If the category of Soergel bimodules is well-behaved (i.e. if the Soergel Categorification Theorem holds) then this functor is an equivalence of monoidal categories.
\item Soergel bimodules play an important role in modular representation theory \cite{Soergel, FiebigJAMS}.
Here it is desirable to have versions over fields of positive characteristic and local rings. One can define the diagrammatic category over very general rings. For example, if $W$ is crystallographic then one can define the diagrammatic category over $\ZM$. Over a complete local ring the indecomposable objects in the diagrammatic category are parametrized (up to shift) by $W$, and the split Grothendieck group always categorifies the Hecke algebra. This gives a new proof of Soergel's Categorification Theorem whenever Soergel bimodules are well-behaved, and suggests that the diagrammatic category is the correct replacement when they are not.
\item The category of Bott-Samelson bimodules is a cellular category, with cells parametrized by $w \in W$, and with cellular basis given by a fixed choice of light leaves morphisms. In particular, the endomorphism ring of any Bott-Samelson bimodule is a cellular algebra.
\end{enumerate}



\subsection{Organization of the paper}

This paper contains three parts. 

\begin{description}
\item[Part \ref{pt:background}] The first two sections give background on the Hecke algebra and Soergel bimodules.
\begin{description}
  \item[Section \ref{sec-hecke}] We recall the Hecke algebra, Kazhdan-Lusztig basis and Deodhar's defect formula.
  \item[Section \ref{sec-sbim}] We define standard and Soergel bimodules, state Soergel's categorification theorem and discuss realizations and localization.
\end{description}
\item[Part \ref{pt:diagrammatics}] In the next two sections, we define the diagrammatic categories.
 \begin{description}
 \item[Section \ref{sec-diagstd}] We define the diagrammatic presentation of standard modules.
 \item[Section \ref{sec-diagsbim}] We recall results of \cite{EDihedral} and define the diagrammatic category of Soergel bimodules.
  \end{description}
\item[Part \ref{pt:proofs}] In the last two sections we study the diagrammatic categories and prove the equivalence to Soergel bimodules.
\begin{description}
  \item[Section \ref{sec-LLL}] We define Libedinsky's light leaves and double leaves morphisms in the diagrammatic setting. We state the theorem that double leaves span, and deduce the main theorems of the paper.
\item[Section \ref{sec-LLLproof}] We prove that double leaves span.
  \end{description}

\end{description}

\subsection{Acknowledgements} Both authors would like to thank Mikhail Khovanov for encouraging and supporting their collaboration. The second author would like to thank Nicolas Libedinsky and Rapha\"el Rouquier for emphasising the importance of generators and relations, and Jean Michel for help speeding up his programs.

\part{Background on the Hecke algebra and Soergel bimodules}

\label{pt:background}

\section{The Hecke Algebra}
\label{sec-hecke}

Background on this section can be found in \cite{Humphreys}.

\subsection{Basic definitions}
\label{subsec-basicdefinition}

Let $(W,S)$ be a Coxeter system and let $e \in W$ denote the identity. That is, $W$ is the group generated by $S$ subject to the relations:
\begin{equation} s^2=e   
\qquad \text{for all $s \in S$,}\end{equation}
\begin{equation} \ubr{sts\ldots}{m_{st}} = \ubr{tst\ldots}{m_{st}}
\label{braidreln} 
\qquad \text{for all $s \ne t \in S$.}
\end{equation}
The numbers $m_{st} = m_{ts}$ associated to each pair of simple reflections determine the group $W$, and must satisfy $m_{st} \in \{2,3,4,\ldots\}$ or $m_{st}=\infty$. When $m_{st}=\infty$, the
so-called \emph{braid relation} \eqref{braidreln} is omitted. The group $W$ is equipped with a Bruhat order $\le$ and a length function $\ell : W \to \ZM_{\ge 0}$.

For any subset $J \subset S$ the corresponding \emph{parabolic subgroup} $W_J$ is the subgroup generated by $s \in J$. Then $(W_J, J)$ is a Coxeter system with presentation induced
from that of $(W,S)$. The \emph{rank} of $W_J$ is the size of $J$. In particular, the parabolic subgroup of a single vertex is isomorphic to $S_2$, and the parabolic subgroup of a pair of
vertices is a finite or infinite dihedral group. We call $J$ \emph{finitary} if $W_J$ is finite, in which case it has a longest element $w_J$.

The Hecke algebra $\HB$ of $W$ is the free $\Zvv$-algebra generated by symbols $T_s$ for $s \in S$, modulo the following relations:

\begin{equation} T_s^2= (v^{-2}-1) T_s  + v^{-2} \label{rank1heckestd} \qquad \text{for all $s \in S$,}\end{equation}
\begin{equation} \ubr{T_sT_tT_s\ldots}{m_{st}} = \ubr{T_tT_sT_t\ldots}{m_{st}} \label{rank2heckestd} 
\qquad \text{for all $s \ne t \in S$.}
\end{equation}
It will be useful to work with a different normalization. If we set $H_s \define vT_s$ then these relations become:
\begin{equation} H_s^2= (v^{-1}-v)H_s + 1 \label{rank1heckestd'} \qquad \text{for all $s \in S$,}\end{equation}
\begin{equation} \ubr{H_sH_tH_s\ldots}{m_{st}} = \ubr{H_tH_sH_t\ldots}{m_{st}} \label{rank2heckestd'} 
\qquad \text{for all $s \ne t \in S$.}
\end{equation}

\begin{notation} 
We will use an underlined roman letter $\un{w} = (s_1,s_2, \dots, s_m)$ to denote a finite sequence of elements of $S$. Omission of the underline will denote the product $w=s_1s_2 \dots s_m$ in $W$. The \emph{length} of $\un{w}$ is $\ell(\un{w}) = m$. Note that $\ell(\un{w}) \ge \ell(w)$ with equality holding if and only if $\un{w}$ is a reduced expression for $w$. We will often abuse notation and write $\un{w} = s_1 s_2  \dots s_m$. The underline reminds us that the sequence of simple reflections, and not just their product in $W$, is important. Given $\un{w} = s_1 s_2  \dots s_m$ we write $H_{\un{w}} = H_{s_1}H_{s_2} \dots H_{s_d}$. Certainly it is possible that $H_{\un{w}} \ne H_{\un{x}}$ even though $x = w$. Later in this paper, similar notation will apply to other iterated products or tensor products.
\end{notation}

\begin{notation} The phrase ``reduced expression" dominates this paper to such an extent we have decided to shorten it to \emph{rex}. The plural of rex is rexes. \end{notation}

Given any two rexes $\un{w}$ and $\un{w}'$ for $w \in W$, it is possible to pass from $\un{w}$ to $\un{w}'$ using only braid relations. It follows from \eqref{rank2heckestd'} that the elements $H_{\un{w}}$ and $H_{\un{w}'}$ are equal, and are denoted $H_w$. We write $H_e \define 1$ for the identity of $\HB$. These elements $\{H_w\}_{w
\in W}$ form the \emph{standard basis} of $\HB$ as a $\Zvv$-module.

\subsection{The Kazhdan-Lusztig basis}
\label{subsec-KLbasis}

The Hecke algebra is equipped with a $\Z$-linear \emph{bar involution}, denoted $h \mapsto \overline{h}$ and
uniquely specified as an algebra homomorphism by $v \mapsto v^{-1}$ and $H_s \mapsto H_s^{-1}$. A simple
calculation shows that $H_s + v = H_s^{-1}+v^{-1}$, so that the element $\un{H}_s = H_s+v$ is bar-invariant.

\begin{thm}(Kazhdan-Lusztig \cite{KaL1}) There exists a unique basis $\{\un{H}_w\}_{w \in W}$ of $\HB$ as a $\Zvv$-module, called the \emph{Kazhdan-Lusztig basis}, which satisfies:
\begin{itemize}
\item $\overline{\un{H}_w}=\un{H}_w$;
\item $\un{H}_w = H_w + \sum_{x < w} h_{x,w} H_x$ where $h_{x,w} \in v\ZM[v]$.\end{itemize}
\end{thm}

The polynomials $h_{x,w} \in \ZM[v]$ are \emph{Kazhdan-Lusztig polynomials}. 

\begin{remark} In this paper we follow the normalization of Soergel \cite{Soe4} rather than the original normalization of \cite{KaL1}. In particular, the Kazhdan-Lusztig polynomials $P_{x,w}$ in \cite{KaL1} are related to the Kazhdan-Lusztig polynomials above by the formula
\[
h_{x,w} = v^{2(\ell(w) - \ell(x))}P_{x,w}(v^{-2}).
\]
\end{remark}

Given $\un{w} = s_1s_2 \dots s_m$ we set $\un{H}_{\un{w}} \define \un{H}_{s_1} \dots \un{H}_{s_m}$. Note that $\un{H}_{\un{w}} \ne \un{H}_w$ in general, but equality does hold when $\ell(\un{w}) \le 2$ and $s_1 \ne s_2$.

An element $w \in W$ is called \emph{(rationally) smooth} if $\un{H}_w = \sum_{y \le w} v^{\ell(w) -\ell(y)}H_y$. The longest element of any finite parabolic subgroup is smooth. Any
element of a rank 2 parabolic subgroup is smooth.

We pause to present three conjectures of Kazhdan and Lusztig:

\begin{conj} The polynomials $h_{x,y}$ have positive coefficients, i.e. they lie in $\ZM_{\ge 0}[v]$. \end{conj}

\begin{conj} The structure coefficients of the Kazhdan-Lusztig basis $\HB$ are positive. In other words, $\un{H}_u \un{H}_v = \sum c^w_{u,v} \un{H}_w$ for some $c^w_{u,v} \in \Nvv$. \end{conj}

These first two conjectures are commonly referred to as the \emph{Kazhdan-Lusztig positivity conjectures}. The following conjecture is the \emph{Kazhdan-Lusztig conjecture}. It and its many generalizations account to a large extent for the interest in Kazhdan-Lusztig polynomials. (For a precise statement of the Kazhdan-Lusztig conjecture see \cite{KaL1}.)

\begin{conj} (The Kazhdan-Lusztig Conjecture) If $W$ is the Weyl group of a complex semi-simple Lie algebra, then the Kazhdan-Lusztig polynomials give the multiplicities of simple modules in Verma modules in the principal block of category $\OC$.\end{conj}

These three conjectures are themselves implications of the Soergel conjecture. We will discuss this in section \ref{subsec-SoergelCatfn}.

\subsection{The presentation in the Kazhdan-Lusztig generators}
\label{subsec-KLpresentation}

As the elements $H_s$ generate $\HB$, so too do the elements $\un{H}_s$. The corresponding relations are slightly more complicated:
\begin{equation} \un{H}_s^2 = \un{H}_s (v+v^{-1}), \label{rank1hecke} \end{equation}
\begin{equation} \ubr{\un{H}_s\un{H}_t\un{H}_s\ldots}{k+1} = \sum_d  c_{k,d} \un{H}_{\ubr{sts\ldots}{d+1}} \ \mathrm{ for }\ 1 \le k+1 \le m_{st}.  \label{rank2hecke} \end{equation}
The first relation corresponds to \eqref{rank1heckestd}.
The second relation expresses a product of generators in terms of the Kazhdan-Lusztig basis within a given dihedral group. The coefficients $c_{k,d}$ appearing are decomposition
numbers for $\mathfrak{sl}_2$ tensor products. If we let $V_i$ denote the $i+1$-dimensional irreducible of $\mathfrak{sl}_2$, then $V_1^{\ot k} \bigoplus \cong V_d^{\oplus c_{k,d}}$. Roughly
speaking, this is because the Temperley-Lieb algebra at a root of unity ``categorifies" the alternating product $\un{H}_s \un{H}_t \un{H}_s \ldots$ in the dihedral Hecke algebra. More details can be
found in \cite{EDihedral}.

Relation \eqref{rank2hecke} holds even when $m_{st}=\infty$, and for any $m$ can be viewed as a definition or explicit construction of each $\un{H}_w$. When $m=\infty$, this relation does not
impose any new algebraic relations on products of $\un{H}_s$ and $\un{H}_t$. When $m < \infty$, there is one new relation on products of $\un{H}_s$ and $\un{H}_t$ coming from the fact that $\un{H}_{\ubr{sts\ldots}{m}} = \un{H}_{\ubr{tst\ldots}{m}}$.

\begin{ex} \label{rank2heckeEx} The first few examples are: \[\begin{array}{lc} m_{st}=2: & \un{H}_s\un{H}_t=\un{H}_t\un{H}_s \\ m_{st}=3: & \un{H}_s\un{H}_t\un{H}_s-\un{H}_s=\un{H}_t\un{H}_s\un{H}_t-\un{H}_t\\ m_{st}=4: & \un{H}_s\un{H}_t\un{H}_s\un{H}_t-2\un{H}_s\un{H}_t = \un{H}_t\un{H}_s\un{H}_t\un{H}_s - 2 \un{H}_t \un{H}_s \\ m_{st}=5: & \un{H}_s\un{H}_t\un{H}_s\un{H}_t\un{H}_s - 3
\un{H}_s\un{H}_t\un{H}_s + \un{H}_s = \un{H}_t\un{H}_s\un{H}_t\un{H}_s\un{H}_t - 3 \un{H}_t \un{H}_s\un{H}_t + \un{H}_t. \end{array} \] \end{ex}

\subsection{The standard trace and the defect formula}
\label{subsec-defectformula}

A \emph{trace} on $\HB$ is a $\Zvv$-linear map $\e \co \HB \to \Zvv$ satisfying $\e(hh') = \e(h'h)$ for all $h,h' \in \HB$. A straightforward calculation shows that the map $\e(\sum c_w
H_w) = c_e$ is a trace, called the \emph{standard trace}. There is a nice combinatorial formula for the standard trace of a product $\un{H}_{\un{w}}$, known
as the \emph{defect formula}, which we now discuss.

A \emph{subsequence} of $\un{w} = s_1s_2\dots s_m$ is a sequence $\pi_1 \pi_2 \dots \pi_m$ such that $\pi_i \in \{ e, s_i \}$ for all $1 \le i \le m$. Instead of working with subsequences,
we work with the equivalent datum of a sequence $\eb = \eb_1 \eb_2 \dots \eb_m$ of 1's and 0's giving the indicator function of a subsequence, which we refer to as a \emph{01-sequence}.

We can also think of $\eb$ as a roadmap for a gentle stroll through the Bruhat graph (with much pausing to admire the scenery). This \emph{Bruhat stroll} is the sequence $x_0=e, x_1,
\dots, x_m$ defined by \[ x_i \define s_1^{\eb_1} s_2^{\eb_2} \dots s_i^{\eb_i} \] for $0 \le i \le m$. We call $x_m$ the \emph{end-point} of the Bruhat stroll, and denote it by $\un{w}^{\eb}$. Alternatively, we will say that a subsequence $\eb$ of $\un{w}$ \emph{expresses} the end-point $\un{w}^{\eb}$.
The Bruhat stroll allows us to decorate each index of $\eb$ with an additional token, either U(p) or D(own). If $\eb_i=1$ so that $x_i = x_{i-1}s_i$, then we assign U to $i$ if $x_i >
x_{i-1}$ (so that we moved up in the Bruhat order at time $i$) or D if $x_i < x_{i-1}$ (so that we moved down at time $i$). If $\eb_i=0$ so that $x_i = x_{i-1}$, we glance longingly in the
direction of $x_{i-1}s_i$ but remain unmoved: we assign U or D to the index $i$ according to whether $x_{i-1}s_i > x_{i-1}$ or $x_{i-1}s_i < x_{i-1}$. The \emph{defect} of a 01-sequence
$\eb$, denoted $d(\eb)$, is defined to be the number of U0's minus the number of D0's. It measures the defect between where we longed to go and where we actually went.

\begin{ex} \label{ex:defects}
Here are some examples of subexpressions, end-points and defects:
  \begin{itemize}
  \item Suppose that $\un{w}=sss$. There are four subsequences with end-point $e$: $(U1,D1,U0)$ and $(U0,U1,D1)$ with defect $1$, $(U1,D0,D1)$ with defect $-1$, and $(U0,U0,U0)$
with defect $3$. There are four subsequences with end-point $s$: $(U1,D1,U1)$ and $(U0,U1,D0)$ with defect $1$, $(U1,D0,D0)$ with defect $-2$, and $(U0,U0,U1)$ with defect $2$.
  \item Suppose that $\un{w} = sts$ and that $m_{st}  = 3$. There are unique subexpressions with endpoints $sts, ts, st$ and $t$ with defects $0$, $1$, $1$ and $2$ respectively. There are two subexpressions $(U1, U0, D0)$ and $(U0, U0, U1)$ with end-point $s$ of defects $0$ and $2$ respectively, and two subexpressions $(U1, U0, D1)$ and $(U0, U0, U0)$ with end-point $e$ with defects $1$ and $3$ respectively.
  \end{itemize}
\end{ex}

The defect is useful because of the following lemma of Deodhar \cite{DD}:

\begin{lem} \label{lem:defects}
For any expression $\un{w}$ we have
\[
\un{H}_{\un{w}} = \sum v^{d(\eb)}H_{\un{w}^{\eb}}
\]
where the sum runs over all 01-sequences of length $\ell(\un{w})$.
\end{lem}

\begin{proof}(Sketch) It is an straightforward consequence of \eqref{rank1heckestd'} that in $\HB$ we have the relation
\[
H_x\un{H}_s = \begin{cases} H_{xs} + vH_{x} & \text{if $xs > x$,} \\
H_{xs} + v^{-1}H_x & \text{if $xs < x$.} \end{cases}
\]
We conclude that if the lemma is true for $\un{w} = \un{x}$ it is also true for $\un{w} = \un{x}s$. The result now follows by induction.
\end{proof}

We now come to the defect formula for the trace:

\begin{cor}\label{cor:defects}
For any expression $\un{w}$ we have
\[ \e(\un{H}_{\un{w}}) = \sum v^{d(\eb)} \]
where the sum is over all 01-sequences expressing the identity element.\end{cor}

\begin{ex} We continue Example \ref{ex:defects} and check Lemma \ref{lem:defects}:
  \begin{itemize}
  \item Using relation \eqref{rank1hecke} we see that
\[
\un{H}_s^3=(v+v^{-1})^2\un{H}_s = (v^{-2} + 2 + v^2)H_s + (v^{-1}+2v+v^3)H_e.
\]
  \item Using direct calculation, or smoothness and \eqref{rank2hecke}, or Lemma \ref{lem:defects} we obtain
\[
\un{H}_s\un{H}_t\un{H}_s = H_{sts} + vH_{ts} + vH_{st} + v^2H_t + (1+v^2)H_s + (v+v^3)H_e.
\]
  \end{itemize}
\end{ex}

It will be important later that the set of subsequences $\eb$ of a fixed expression $\un{w}$ is equipped with a partial order, the \emph{path dominance order}. Let $\eb$ and $\fb$ be two 01-sequences and let their corresponding Bruhat strolls be $x_0, x_1, \dots, x_m$ and $y_0, y_1, \dots, y_m$. We say that $\eb \ge \fb$ if $x_i \ge y_i$ for all $0 \le i \le m$. Clearly if $\eb \ge \fb$ then the end-point of $\eb$ is greater than the end-point of $\fb$. The path dominance order restricts to a partial order on the set of subsequences with fixed end-point.

Let $\w$ be the $\Z$-linear antiinvolution for which $\w(\un{H}_s)=\un{H}_s$ and $\w(v)=v^{-1}$. The standard trace gives rise to the \emph{standard pairing} $\HC \times \HC \to \Zvv$,
defined by $(a,b) = \e(b\w(a))$. This pairing is \emph{semilinear} over $\Zvv$; that is $(v^{-1}a,b)=(a,vb)=v(a,b)$ for all $a, b \in \HB$. Under this pairing, $b_s$ is self-biadjoint, i.e. \[(b_s x,y) = (x,b_s y),\qquad \qquad (x b_s,y) = (x,y b_s). \]

\begin{remark} \label{rmk:traces pairings} The formula $(a,b) = \e(b\w(a))$ can be used both ways, to define a pairing from a trace or vice versa. One can see that the $\Zvv$-module of all
semilinear pairings with self-biadjoint $b_s$ is isomorphic to the module of all $\Zvv$-linear traces. Any such pairing is determined by the values $\e(b_\xb)=(1,b_\xb)$ over all sequences
$\xb$. \end{remark}

\section{Soergel Bimodules}
\label{sec-sbim}

\subsection{Realizations of Coxeter systems} \label{subsec-realization}

For both Soergel's construction of Soergel bimodules, and for our construction of a diagrammatic category by generators and relations, the starting point will be the data of a realization
of a Coxeter system.

\begin{defn} \label{defn:realization} Let $\Bbbk$ be a commutative ring. A \emph{realization} of $(W,S)$ over $\Bbbk$ is a free, finite rank $\Bbbk$-module $\hg$, together with subsets $\{ \a_s^\vee \; | \; s \in S\} \subset \hg$ and $\{ \a_s \; | \; s \in S\} \subset \hg^* = \Hom_{\Bbbk}(\hg,\Bbbk)$, satisfying:
\begin{enumerate}
\item $\langle \alpha_s^\vee, \alpha_s \rangle = 2$ for all $s \in S$;
\item the assignment $s(v) \define v - \langle v, \alpha_s\rangle \alpha_s^{\vee}$ for all $v \in \hg$ yields a representation of $W$;
\item the technical condition in \eqref{eq:technical-condition} is satisfied. (Its description requires some background.)
\end{enumerate}
We will often refer to $\hg$ as a realization, however the choice of $\{ \alpha_s^\vee \}$ and $\{\alpha_s \}$ is always implicit. \end{defn}

Given a realization over $\Bbbk$ and a homomorphism $\Bbbk \to \Bbbk'$ we obtain a realization $\Bbbk' \otimes_{\Bbbk} \hg$ over $\Bbbk'$ by base change. We call a realization
\emph{faithful} if the action of $W$ on $\hg$ (and hence the contragredient action on $\hg^*$) is faithful. Base change does not preserve faithfulness in general. For us, the ability to
perform base change is the more important property, so we must allow realizations which are not faithful. For instance, any realization of the dihedral group with $m_{st}=m<\infty$ is also
a realization of the dihedral group with $m_{st}=2m, 3m, \ldots$, and is also a realization of the infinite dihedral group.

We call a realization \emph{symmetric} if $\langle \alpha_s^\vee, \alpha_t \rangle = \langle \alpha_t^\vee, \alpha_s \rangle$ for all $s, t \in S$.

\begin{ex} \label{ex:realizations} Some examples of realizations that we have in mind are the following:
  \begin{enumerate}
  \item Let $(W,S)$ be any Coxeter system of finite rank. Let $\Bbbk = \RM$ and $\hg = \bigoplus_{s \in S} \RM \alpha_s^{\vee}$. Define elements $\{ \alpha_s \} \subset \hg^*$ by
\begin{equation} \label{eq:justcos}
\langle \alpha_t^\vee, \alpha_s \rangle = -2\cos(\pi/m_{st})
\end{equation}
(by convention $m_{ss} =1$ and $\pi/\infty = 0$). Then $\hg$ is a symmetric realization of $(W,S)$, called the \emph{geometric representation} (see \cite[\S 5.3]{Humphreys}). Note that the subset $\{ \alpha_s \} \subset  \hg^*$ is linearly independent if and only if $W$ is finite.

\item
More generally, given a real vector space $\hg$ with subsets $\{ \alpha_s^{\vee} \} \subset \hg$ and $\{ \alpha_s \} \subset \hg^*$ satisfying \eqref{eq:justcos} then $\hg$ is a realization of $(W,S)$. In  \cite[\S 2]{Soe3} Soergel builds his theory of Soergel bimodules for arbitrary Coxeter systems on a realization for which both $\{ \alpha_s^{\vee} \} \subset \hg$ and $\{ \alpha_s \} \subset \hg^*$ are linearly independent, and such that $\hg$ has minimal dimension with this property. To construct such a representation, Soergel mimics the construction of the action of an affine Weyl group on the Cartan subalgebra of an affine Kac-Moody group.

\item Let $(X, R, X^\vee, R^\vee)$ be a (reduced) root datum (see \cite[\S 7.4]{SprLAG} for notation) and let $\Delta \subset R$ be a set of simple roots. Let $(W,S)$ be the corresponding Weyl group and simple reflections. Then the triple $\hg = X$, $\{\alpha \; | \; \alpha \in \Delta \} \subset \hg$ and $\{\alpha^\vee \; | \; \alpha \in \Delta \} \subset \hg^* = X^\vee$ gives a faithful realization of $(W,S)$ over $\ZM$. We obtain a (potentially non-faithful) realization of $(W,S)$ over any $\Bbbk$ by extension of scalars.

\item
More generally, if $A$ is a generalized Cartan matrix and $\tg$ denotes the Cartan subalgebra of the corresponding Kac-Moody Lie algebra $\gg(A)$ (see \cite[Chapters 1 and 3]{KacBook}) then any choice of $\ZM$-lattices $\hg \subset \tg$ such that $\hg$ contains the root lattice and its dual lattice $\hg^* \subset \tg^*$ contains the coroot lattice yields a realization of the Weyl group $(W,S)$ of $\gg(A)$. In this way one obtains realizations over $\ZM$ (and hence over any $\Bbbk$) of any Coxeter system for which $m_{st} \in \{ 2,3,4,6, \infty\}$ for all $s \ne t \in S$. (Such Coxeter systems are called \emph{crystallographic}.)

\item Let $(W,S)$ be a Coxeter system for  which $m_{st} \in \{ 2,3,5,\infty\}$ and let $\Bbbk = \ZM[\phi]$ where $\phi = (1 + \sqrt{5})/2$ denotes the golden ratio. Let $\hg = \bigoplus_{s \in S} \Bbbk \alpha_s^{\vee}$ and define $\alpha_s \in \hg^*$ via
\[
\langle \alpha_s^\vee, \alpha_t \rangle = \begin{cases} 2 & \text{if $s = t$,} \\ 0 & \text{if $m_{st} =2 ,$} \\ -1 & \text{if $m_{st} =3 ,$} \\ -\phi & \text{if $m_{st} =5,$} \\ -2 & \text{if $m_{st} =\infty.$} \end{cases}
\]
Using that $-2\cos(\pi/5) = -\phi$ it follows from the example of the geometric realization that $(\hg, \{ \alpha_s^{\vee} \}, \{ \alpha_s \})$ is a (symmetric) realization of $(W,S)$ over $\Bbbk$. In particular, the finite reflection groups of types $H_3$ and $H_4$ have symmetric realizations over (any extension of) $\Bbbk$.

\item Let $(W,S)$ be the affine Weyl group of type $\widetilde{A_n}$ for $n \ge 2$. Let $\Bbbk = \ZM[q,q^{-1}]$ and $\hg = \bigoplus_{s \in S} \ZM[q,q^{-1}] \alpha_s^{\vee}$, and let the values of
$\langle \alpha_s^\vee, \alpha_t \rangle$ be encoded (as will be described soon) in the following matrix: \[A = \left( \begin{array}{ccccc} 2 & -1 & 0 & 0 & -q^{-1} \\ -1 & 2 & -1 & 0 & 0
\\ 0 & -1 & 2 & -1 & 0 \\ 0 & 0 & -1 & 2 & -q \\ -q & 0 & 0 & -q^{-1} & 2 \end{array} \right).\] 
(More precisely, this is the example when $n = 4$.)
This gives a realization of $W$. Specializing $q$ to an element of $\CM \setminus \RM$, one
obtains a realization of $W$ over $\CM$ which can not be obtained by extension of scalars from a realization defined over $\RM$.
\end{enumerate}
\end{ex}

Given a realization $(\hg, \{\alpha_s^{\vee} \}, \{\alpha_s\})$ of $(W,S)$ over $\Bbbk$ we can consider its \emph{Cartan matrix} $(\langle \alpha_s^\vee, \alpha_t \rangle ) _{s,t \in S}$.
Clearly a realization is symmetric if and only if its Cartan matrix is. Conversely, given a matrix $(a_{st})_{s,t \in S}$ such that $a_{ss}=2$, one can construct the free $\Bbbk$-module
$\hg = \bigoplus_{s \in S} \Bbbk \alpha_s^{\vee}$, and define $\alpha_s \in \hg^*$ by $\langle \alpha_s^\vee, \alpha_t \rangle = a_{st}$. When this yields a realization of $(W,S)$ we call
the matrix $(a_{st})$ a \emph{Cartan matrix} for $(W,S)$ over $\Bbbk$. Any realization for which $\{\a_s^\vee\}$ is a basis for $\hg$ can be reconstructed from its Cartan matrix; we call
such realizations \emph{minimal}.

\begin{ex} In Example \ref{ex:realizations} the realizations discussed in (1), (5) and (6) are minimal. The example in (3) is minimal if and only if the root system is simply connected and
of adjoint type (so that root lattice coincides with $\hg$). Examples (2) and (4) are not minimal in general. \end{ex}

\begin{remark} We expect that there is a rich Koszul duality theory for categories obtained from Soergel bimodules for arbitrary Coxeter systems (generalizing Soergel's description
\cite{Soe1} of the algebra of category $\OC$ in the case of Weyl groups). Here one expects Koszul duality to exchange $\hg$ and $\hg^*$, roots and coroots. In this setting it seems natural
to require both $\{ \alpha_s^{\vee} \} \subset \hg$ and $\{ \alpha_s \} \subset \hg^*$ to be linearly independent. This explains in part why we do not assume that our realizations are
minimal. \end{remark}

It is natural to ask under which conditions a matrix $(a_{st})_{s,t \in S}$ with $a_{ss} = 2$ is a Cartan matrix of $(W,S)$. A thorough discussion of this can be found in the appendix to
\cite{EDihedral}. We provide a short discussion here.

\begin{defn} \label{def:2qnum} Define the \emph{2-colored quantum numbers} $[k]_x$ and $[k]_y$ inside the ring $\ZM[x,y]$ inductively. One has $[0]_x = [0]_y = 0$, $[1]_x = [1]_y = 1$, and $[2]_x = x$, $[2]_y=y$. The other 2-colored quantum numbers are defined by the rules

\begin{subequations}
\begin{equation} [2]_x [k]_y = [k+1]_x + [k-1]_x, \end{equation}
\begin{equation} [2]_y [k]_x = [k+1]_y + [k-1]_y. \end{equation}
\end{subequations}

When $k$ is odd, $[k]_x = [k]_y$ and we shorten the notation to $[k]$. \end{defn}

Fix a pair $s,t \in S$ and let $x=a_{st}$ and $y=a_{ts}$. The condition that $(st)$ has order exactly $k$ is an algebraic condition on $x$ and $y$. For instance, when $\a_s$ and $\a_t$ are
linearly independent, the action of $(st)$ on their span has order $k$ when $[2k+1]=1$ and $[2k]_x = [2k]_y = 0$. In fact, this implies further that $2[k]_x = 2[k]_y = [2]_x [k]_y = [2]_y
[k]_x = 0$. This suggests that one should set \begin{equation} \label{eq:technical-condition} [m_{st}]_x = [m_{st}]_y = 0. \end{equation} This is the technical condition mentioned in
Definition \ref{defn:realization}. While \eqref{eq:technical-condition} is sufficient to imply that $(st)$ has order dividing $k$ on the span of the roots, it is independent of the
condition that $W$ acts on $\hg$. The reason that \eqref{eq:technical-condition} is required is to ensure that 2-colored Jones-Wenzl projectors are rotation-invariant, as discussed in
section \ref{subsec-JW}.

If either $x$ or $y$ is a non-zero-divisor, then \eqref{eq:technical-condition} is equivalent to the statement that $xy$ satisfies the minimal polynomial of the algebraic integer $4
\cos^2(\frac{\pi}{m})$. If this algebraic integer does not exist in $\Bbbk$ then $(W,S)$ does not admit a realization over $\Bbbk$. Any Coxeter system (of finite rank) admits a realization
over some ring of integers. Finally, we introduce one other technical condition.

\begin{defn} We call a realization \emph{balanced} if for every $s,t \in S$ one has $[m_{st}-1]_x = [m_{st}-1]_y = 1$. We refine this notion by calling the realization \emph{even-balanced}
(resp. \emph{odd-balanced}) if this property holds when $m_{st}$ is even (resp. odd). The opposite of even-balanced is \emph{even-unbalanced}. \end{defn}

The familiar Cartan matrices of Weyl groups are balanced. However, the Cartan matrix of type $A_2$ is not balanced when viewed as a realization of $G_2$. The exotic Cartan matrices for
type $\tilde{A}_n$, $n \ge 3$ given in Example (6) above are not balanced, except when $q=1$. Being balanced is equivalent to the existence of an unambiguous notion of positive roots in
$\hg^*$; when the realization is symmetric, being balanced is similar to the condition that simple roots form an obtuse angle. Faithful realizations are almost always even-balanced; any
minimal even-unbalanced realization over a domain is not faithful. Once again, a thorough discussion of these technicalities can be found in \cite{EDihedral}.

Fix a realization $(\hg, \{ \alpha_s^\vee \}, \{ \alpha_s \})$ of $(W,S)$ and let \[ R = \bigoplus_{m \ge 0} S^m(\hg^*) \] denote the symmetric algebra on $\hg^*$, which we view as a
graded $\Bbbk$-algebra with $\deg \hg^* = 2$. Then $W$ acts on $\hg^*$ via the contragredient representation ($s(\gamma) = \gamma - \langle \alpha_s^\vee, \gamma \rangle \alpha$ for all
$\g \in \hg^*$) and this extends to an action of $W$ on the algebra $R$ by graded automorphisms. We think of $R$ as the polynomial functions on $\hg$.

We let $R-\Mod$ and $R-\Bim$ denote the category of graded $R$-modules and bimodules respectively. We view $R-\Mod$ and $R-\Bim$ as graded categories; that is, as categories enriched in
graded $\Bbbk$-modules. We denote the grading shift by $(1)$: if $M = \bigoplus M^i$ is a graded (bi)module then $M(1)^i = M^{i+1}$. Degree $0$ maps of (graded) $R$-(bi)modules will be
denoted by $\Hom_0(M,N)$.

\subsection{Technicalities}

It is important to remember the key dichotomy in this paper: we will discuss two separate categories ``at once." Fix a realization of $(W,S)$, and consider the ring $R$ defined above. In
this chapter we will introduce Soergel's monoidal category $\SBim$, which is a full subcategory of $R-\Bim$. In chapter \ref{sec-diagsbim} we will define a diagrammatic category $\DC$ by
generators and relations, whose morphism spaces will be enriched in $R-\Bim$. One will need to make some assumptions on the realization in order for either category to ``behave well" (i.e.
in order for the Soergel Categorification Theorem to hold, and in order for double leaves to form a basis for Hom spaces; see the introduction). Whenever $\SBim$ behaves well, one can
construct an equivalence from $\DC$ to $\SBim$. However there are certain situations (for example when the characteristic of $\Bbbk$ is small, or when working over a complete local ring)
where the diagrammatic theory continues behaving well, but the bimodule theory either breaks down or has not yet been developed. In these cases, the diagrammatic theory seems to provide a
natural replacement for Soergel bimodules.\footnote{Another natural replacement is Fiebig's theory of sheaves on moment graphs.} This is one of the advantages of
the diagrammatic approach.

In this section we will discuss the technical assumptions one must make on the realization in order for the diagrammatics to behave well, and the further assumptions needed for Soergel
bimodules to behave well. The novice reader should ignore this section, and should stick with the geometric realization defined in part (1) of Example \ref{ex:realizations}. This section
may be overly pedantic, however, in view of current and future applications we make an effort to state all results in a natural level of generality.

Note that the very existence of a realization is already an assumption on the base ring $\Bbbk$: namely, that it contains certain algebraic integers.

\begin{assumption}(Demazure Surjectivity) \label{ass:Demazure Surjectivity} The map $\a_s \co \hg \to \Bbbk$ is surjective, for all $s \in S$. Evaluation at $\a_s^\vee$ gives a surjective
map $\hg^* \to \Bbbk$, for all $s \in S$. \end{assumption}

Whenever Demazure Surjectivity holds, there is some $\d \in \hg^*$ for which $\langle \a_s^\vee,\d \rangle = 1$. Moreover, $\a_s \ne 0$, so that $s(\d) = \d - \a_s \ne \d$.

If $2$ is invertible in $\Bbbk$ then Demazure Surjectivity holds, because $a_{ss}=2$. If $m_{st}$ is odd then both $\a_s$ and $\a_t$ (and $\a_s^\vee$ and $\a_t^\vee$) are surjective,
because the algebraic integer $4 \cos^2(\frac{\pi}{m_{st}})$ is invertible in any ring which contains it (see \cite{EDihedral}). Even when the ideal in $\Bbbk$ generated by $\langle
\a_t^\vee, \a_s \rangle$ as $t$ varies (for fixed $s$) is not the unit ideal, it is still possible that $\a_s$ is surjective when the realization is not minimal. Finally, the Demazure Surjectivity property is preserved by base change.

We will assume Demazure Surjectivity henceforth (with the exception of some remarks). In addition to standard ring-theoretic assumptions, this will be the only special assumption we need
to make in order for $\DC$ to be well-behaved.

Our arguments in Section \ref{subsec-SThmDiag} classifying the indecomposable objects in $\Kar(\DC)$ require that $\Bbbk$ is a complete local ring. This assumption is needed for either
category to satisfy the Krull-Schmidt theorem, and for idempotent lifting arguments to work. If $\Bbbk$ is not a complete local ring (for example $\ZM$) we have no idea how many
indecomposable Soergel bimodules there are, nor whether the Krull-Schmidt theorem holds. This is a typical situation in representation theory: one has a generic category (for example
representations of a finite group) defined over (some finite extension of) $\ZM$, but it is only after completing at a prime that one obtains a category in which it makes sense to discuss
indecomposable objects, do homological algebra etc. Moreover, the behavior at different primes can be vastly different.

We assume that $\Bbbk$ is a domain. In particular, it has no non-trivial idempotents, so that any graded $R$-(bi)module with $\End_0(M) = \Bbbk$ is indecomposable.

\begin{defn} Following Soergel \cite{Soe3}, we say that a realization $\hg$ over a field $\Bbbk$ is \emph{reflection faithful} if $\hg$ is a faithful representation of $W$, and if there is
a bijection between the set of reflections (i.e. the conjugates in $W$ of $S$) and the codimension one subspaces of $\hg$ that are fixed by some element of $W$. \end{defn}

This is a fairly serious assumption on a realization. For instance, no infinite Coxeter group admits a faithful representation over $\overline{\FM}_p$. Soergel constructs a reflection
faithful representation of any Coxeter group over $\RM$, using the approach mentioned in part (2) of Example \ref{ex:realizations}.

Soergel's theory gives techniques to study $\SBim$ defined for a reflection faithful representation over an infinite field $\Bbbk$ of characteristic $\ne 2$. Libedinsky \cite{Lib} has
shown that his results extend to the geometric realization as well. It seems plausible that many of Soergel's techniques could be adapted to other complete local rings and faithful
realizations over them, but the true generality of his results is unknown. We say that a realization is a \emph{Soergel realization} if it is faithful and Soergel's techniques can be
applied (i.e. if we can quote the Soergel Categorification Theorem).

Finally, while it will not affect the truth or falsehood of the Soergel Categorification Theorem, the assumption that the realization is balanced will drastically simplify both the study
of bimodules and the study of diagrammatics. We do not take this assumption in general.

\subsection{Demazure operators}
\label{subsec-polys}

Fix $s \in S$. We will extend the map $\langle \a_s^\vee, \cdot \rangle \co \hg^* \to \Bbbk$ to the \emph{Demazure operator} $\pa_s \co R \to R^s(-2)$, by the formula 
\[
\pa_s(f) = \frac{f - s(f)}{\a_s}.
\]
The numerator and denominator are both $s$-antiinvariant, so that the fraction, assuming it is well-defined, should lie in the subring $R^s$ of $s$-invariants.
Clearly this map makes sense for $f \in \hg^*$, and agrees with $\langle \a_s^\vee, \cdot \rangle$. Let us demonstrate that it makes sense in general.

Suppose that $\d$ is an element of $\hg^*$ such that $\langle \a_s^\vee, \d \rangle = 1$, guaranteed to exist by the assumption of Demazure Surjectivity. 

\begin{claim} Any element $f \in R$ can be written uniquely as $f = g\d + h$ for $g,h \in R^s$. \label{claim:unique decomposition}\end{claim}
	
\begin{proof} If $f$ is of this form then $f - s(f) = g(\d - s(\d)) = g \a_s$. The uniqueness of such an expression is now clear: if $g
\d + h = g' \d + h' = f$ then $f-s(f) = g \a_s = g' \a_s$. Since $\Bbbk$ is a domain $R$ is also, and therefore $g=g'$, so that $h=h'$. 

Clearly $a_{st} \d - \a_t \in \hg^*$ is $s$-invariant, lying in the kernel of $\langle \a_s^\vee, \cdot \rangle$. In particular, this implies that any polynomial in $R$ can
be expressed as a polynomial in $\d$ with coefficients in $R^s$. Moreover, $\d^2 = \d(\d + s(\d)) - \d s(\d)$, where both $\d + s(\d)$ and $\d s(\d)$ are $s$-invariant. Therefore any
polynomial in $\d$ can be written as $\d g + h$ for $g,h \in R^s$. \end{proof}

Thus defining $\pa_s(f) \define g$ makes sense, and agrees with the above formula. A similar argument shows that this alternative definition of $\pa_s$ is independent of
the choice of $\d$. The Demazure operator $\pa_s$ is a map $R^s$-bimodules, whose kernel is precisely $R^s$. It is easy to show that $\pa_s$ satisfies the \emph{twisted Leibniz rule}
\[\pa_s(fg) = a\pa_s(g) + \pa_s(a)(sg) \qquad \text{for all $f, g \in R$.} \] Unsurprisingly, Demazure Surjectivity implies that $\pa_s$ is surjective, for all $s \in S$.

Claim \ref{claim:unique decomposition} implies that $R$ is free of rank 2 over $R^s$, generated by $1$ and $\d$. In fact, the sets $\{1,\d\}$ and $\{-s(\d),1\}$ give dual bases of $R$ as
an $R^s$-module under the pairing $(g,h)=\pa_s(gh)$. This gives $R$ the structure of a graded Frobenius algebra over $R^s$. The element $\Delta_s = \d \ot 1 - 1 \ot s(\d) \in R \ot_{R^s}
R$ is independent of the choice of $\d$ with $\pa_s(\d)=1$, and satisfies $g \Delta_s = \Delta_s g$ for any $g \in R$.

\begin{remark} There is a unique choice $\d = \frac{\a_s}{2}$ such that the basis $\{1,\d\}$ is self-dual. It only exists when $2$ is invertible in $\Bbbk$. This was the choice used in the introduction. \end{remark}

\begin{remark} We have taken the assumption of Demazure Surjectivity precisely in order that $R$ would be a Frobenius extension of $R^s$, with trace map $\pa_s$. Without this assumption
the situation is less tractable. When the image of $\pa_s \co \hg^* \to \Bbbk$ is a non-trivial principal ideal, it will be true that $R$ is a Frobenius extension of $R^s$, but with
a rescaled trace map. If this image is a non-principal ideal, then $R$ is not even free over $R^s$. Other scenarios which our assumption forbids are $\pa_s = 0$ or $\a_s = 0$, as either
would imply $R=R^s$ (this is only a possibility in characteristic 2). \end{remark}

\begin{remark} If the realization is odd-balanced, then Demazure operators associated to $s \in S$ satisfy the braid relations. Otherwise, they do not. See \cite{EDihedral} for more
details. \end{remark}

\subsection{Soergel bimodules and standard bimodules}
\label{subsec-soergelandstandard}
In this section we give an introduction to Soergel bimodules and standard bimodules, following Soergel's ``classical'' treatment of the subject \cite{Soe3}.

For $s \in S$, let $B_s$ denote the $R$-bimodule $R \ot_{R^s} R (1)$, given by restriction followed by induction and a grading shift. Henceforth, $\ot$ will denote the tensor product over $R$, while $\ot_s$ will denote the tensor product over $R^s$.

Given a sequence $\un{w} = s_1s_2 \dots s_d$ the corresponding \emph{Bott-Samelson bimodule} is the tensor product
\[
B_{\un{w}} \define B_{s_1} \ot B_{s_2} \ot \ldots \ot B_{s_d}
\]
viewed as an $R$-bimodule under left and right multiplication. The Bott-Samelson bimodule $B_{\un{w}}$ is isomorphic to $R \ot_{s_1} R \ot_{s_2} R \ot \cdots \ot_{s_d} R (d)$. We let $\BSBim$ denote the full monoidal subcategory of $R-\Bim$ whose objects are Bott-Samelson bimodules (where as before, morphism spaces are graded $\Bbbk$-modules). Finally, we let $\SBim$ denote the idempotent closure or Karoubi envelope of (the additive, graded closure of) $\BSBim$, known as the category of \emph{Soergel bimodules}. That is, the indecomposable Soergel bimodules are the indecomposable direct summands of shifts of Bott-Samelson bimodules. Note that $\SBim$ is additive but not abelian.

There are also a number of other bimodules which play an important role in the theory. They are \emph{not} Soergel bimodules in general, because they do not appear as summands in Bott-Samelson
bimodules, only as submodules and quotients. These are the \emph{standard bimodules}. For $w \in W$, let $R_w$ denote the $R$-bimodule which is isomorphic to $R$ as a $\Bbbk$-module, and
where the left action of $f \in R$ is multiplication by $f$, while the right multiplication is multiplication by $w(f)$. It is clear that $R_w \ot R_v \cong R_{wv}$. We refer to the
additive monoidal category consisting of all direct sums of grading shifts of $R_w$ as $\StdBim$. This monoidal category is generated by $R_s$ for $s \in S$. A prototypical
object is $R_{\un{w}} \define R_{s_1} \ot R_{s_2} \ot \cdots R_{s_d}$. Unlike for Bott-Samelson bimodules, one has $R_{\un{w}} \cong  R_{\un{w}'}$ if $w = w'$.

It is useful to picture tensor products of bimodules $B_s$ and $R_w$, for example $B_s \ot B_t \ot R_w \ot B_s$, as being separators or dividers between regions, with regions corresponding
to the $\ot$ signs as well as to the left and right sides. A standard tensor in such a bimodule consists of a polynomial in each region. The bimodule encodes certain rules about how
polynomials may slide across the dividers. For instance, an element of $B_s \ot B_t$ consists of (a linear combination of) a choice of 3 polynomials (left, middle, and right), such that an
$s$-invariant polynomial may slide across the first divider, and a $t$-invariant polynomial across the second. An element of $R_w \ot R_v$ consists of 3 polynomials, and any polynomial may
be slid across any divider, at the cost of applying the appropriate element of $W$ to it. When we write $R_w$ in this way, right multiplication is the usual \emph{untwisted} multiplication
on the right slot; it is the left slot which is identified as a $\Bbbk$-module with $R$ in the definition of the previous paragraph. We call the element $1 \ot 1 \ot 1 \ot \cdots \ot 1$ of
such a bimodule a \emph{1-tensor}. Clearly the 1-tensor is the unique element of minimal degree, up to a scalar.

Let us assume for the rest of this discussion that the realization is faithful. We have
\[
\Hom(R_w,R_v) = \begin{cases} R & \text{if $w = v$} \\ 0 & \text{otherwise.} \end{cases}
\]
In other words, $\StdBim$ is isomorphic to the graded 2-groupoid for $W$ over $R$ (see \cite{EWFenn} for terminology). Therefore, any map between standard bimodules is determined by the image of the 1-tensor.  In particular, a degree $0$ map in $\StdBim$ between $R_\xb$ and $R_\yb$ will send the 1-tensor to a scalar multiple of the 1-tensor.

We write $\StdBim_0$ for the ungraded monoidal category whose objects are standard bimodules and whose morphisms are degree $0$ maps. This is equivalent as a monoidal category to the
2-groupoid for $W$ over $\Bbbk$, and can be defined without any restrictions on $\Bbbk$. Presenting this 2-groupoid as a monoidal category by generators and relations is surprisingly
interesting (see \cite{EWFenn} and the next chapter).

We note that all tensor products of Soergel bimodules and standard bimodules are free as left (or right) $R$-modules, and thus the Hom spaces between them are $R$-torsion-free. We will soon show that, under certain assumptions on the realization, they are free as $R$-modules.

Now we discuss the maps between Soergel bimodules and standard bimodules. There is an injection of bimodules $R(-1) \to B_s$ arising from the Frobenius algebra structure, sending $1
\mapsto \Delta_s$ (defined in the previous section). The cokernel of this map is naturally isomorphic to $R_s(1)$, via the map $B_s \to R_s(1)$ sending $f \ot g \mapsto fs(g)$. Conversely,
there is a surjection $B_s \to R(1)$ sending $f \ot g \mapsto fg$. The kernel of this map is naturally isomorphic to $R_s(-1)$ via the map sending $1 \mapsto \d \ot 1 - 1 \ot \d$.

\begin{remark} Just as $\Delta_s$ has a canonical description, so too does this element $\d \ot 1 - 1 \ot \d$. The trace map $\pa_s \co R \to R^s$ induces an $R$-bilinear pairing $(f,g)
\mapsto \pa_s(fg)$, but it also induces a twisted-bilinear pairing $(f,g) \mapsto \pa_s(f s(g))$. The bases $\{1,\d\}$ and $\{-\d,1\}$ of $R$ over $R^s$ are dual for the twisted pairing,
and the element $\d \ot 1 - 1 \ot \d$ is independent of the choice of dual bases. The reader can ponder the notion of a twisted Frobenius extension. \end{remark}

All four of these maps have graded degree $1$. We encode them in two short exact sequences: \begin{equation} 0 \to R(-1) \to B_s \to R_s(1) \to 0 \label{ses1} \end{equation}
\begin{equation} 0 \to R_s(-1) \to B_s \to R(1) \to 0 \label{ses2} \end{equation} Thus $B_s$ is filtered by $R$ and $R_s$, though in no particular order, and the grading shifts which
appear depend on the chosen order. This implies that every Bott-Samelson bimodule has a filtration whose subquotients are standard bimodules (with shifts). Because each bimodule $R_w$ is
indecomposable, any direct summand of a Bott-Samelson bimodule (and hence any Soergel bimodule) also has such a filtration. However, the order in which standard modules appear in such a
filtration need not respect the Bruhat order.

Now assume that $\hg$ is a Soergel realization. It is a deeper fact that any Soergel bimodule has a filtration in which all successive subquotients are standard modules, occurring in an
order refining the Bruhat order (resp. the reversed Bruhat order). Such filtrations are called \emph{standard filtrations}, and the
graded multiplicities of the standard modules appearing do not depend on the choice of filtration. One defines the \emph{character} of a Soergel
bimodule $B$ as the element $\ch(B)$ of the Hecke algebra counting these graded multiplicities. (For the precise definition of the character see \cite[\S 5]{Soe3}).

Soergel went on to prove that the space of homomorphisms between any two Soergel bimodules is free as a left or right $R$-module, with graded rank given by evaluating the standard pairing
on the characters of each bimodule. Soergel used this formula to classify, in a non-constructive way, all the indecomposable Soergel bimodules. From the above discussion it is clear that
if $\un{w}$ is a rex, then $B_{\un{w}}$ has $R_w$ in its standard filtration with multiplicity $1$, and all other standard modules appearing are isomorphic to $R_y$
for $y<w$. The following theorem is due to Soergel \cite[Satz 6.14]{Soe3}:

\begin{thm} Let $\hg$ be a Soergel realization of $W$. For all $w \in W$ there exists a unique (up to isomorphism) bimodule $B_w$ which occurs as a direct summand of $B_{\un{w}}$ for any reduced expression $\un{w}$ for $w$. The bimodule $B_w$ is uniquely determined as a summand of $B_{\un{w}}$ by the fact that it is indecomposable and (some shift of) $R_w$ occurs in its standard filtration. The set $\{ B_w \; | \; w \in W \}$ constitutes a complete set of non-isomorphic indecomposable Soergel bimodules, up to isomorphism and grading shift.\label{SoergelThm1} \end{thm}

The statement that $B_{\un{w}}$ contains a unique summand such that (some shift of) $R_w$ occurs in its standard filtration follows immediately from the indecomposability of $R_w$. The fact that such summands for different rexes are isomorphic is not difficult to prove (using the Krull-Schmidt property and an idempotent lifting argument). The difficulty in the theorem is to show that any summand of any Bott-Samelson bimodule is isomorphic to one of the bimodules $B_x$ up to a shift.

The theorem implies that $B_w$ is the only summand of $B_{\un{w}}$ which is not a summand of (some shift of) $B_{\un{y}}$ for any shorter sequence $\un{y}$. In principle one can
``construct" $B_w$ by finding all the summands of $B_{\un{w}}$ which occur as shifts of summands of lower terms, removing them, and seeing what remains. This amounts to a calculation of
all idempotents in $\End_0(B_{\un{w}})$, which is a difficult and subtle question.

Because of the implicit definition of the indecomposable bimodule $B_w$, its intrinsic properties often depend on the characteristic of $\Bbbk$ and the realization $\hg$ of $(W,S)$.
For example, it may happen that $B_{\un{w}}$ admits a non-trivial decomposition in characteristic $0$, with a nontrivial summand $B_w$, while for a certain finite characteristics $B_{\un{w}}$ is indecomposable, meaning that $B_w = B_{\un{w}}$. In this paper we determine the algebra of endomorphisms of $B_{\un{w}}$, but are only able to make very basic statements about the representation theory of this algebra.

\subsection{Categorification}
\label{subsec-SoergelCatfn}

We denote the split Grothendieck group of an additive category by $[\CC]$. That is, $[\CC]$ is the abelian group generated by symbols $[M]$ for all objects $M \in \CC$ subject
to the relations $[M] = [M'] + [M'']$ whenever $M \cong M' \oplus M''$ in $\CC$. When $\CC$ is monoidal, $[\CC]$ has the structure of a ring via $[M][M'] = [M\otimes M']$. If in addition
$\CC$ is graded with grading shift functor $M \mapsto M(1)$, then $[\CC]$ has the structure of a $\Zvv$-algebra via $v[M] \define [M(1)]$.

The following is \emph{Soergel's categorification theorem} \cite[Satz 1.10 and Satz 5.15]{Soe3}:

\begin{thm} Let $\hg$ be a Soergel realization of $W$. There is a unique isomorphism of $\Zvv$-algebras:
\begin{align*}
\e: \HB &\simto [\SBim] \\
\un{H}_s &\mapsto [B_s].
\end{align*}
Given Soergel bimodules $B$ and $B'$ the graded rank
of $\Hom(B,B')$ as a free left (or right) $R$-module is given by $(\e^{-1}[B],\e^{-1}[B'])$, where $(-,-)$ denotes the standard pairing on $\HB$.
\label{SoergelThm2}
\end{thm}

The uniqueness of $\e$ is immediate, because $\{\un{H}_s\}_{s \in S}$ generates $\HB$. To see that $\e$ is a homomorphism it is enough to check the relations \eqref{rank1hecke} and
\eqref{rank2hecke}. Using the Frobenius algebra structure and the isomorphism $R \cong R^s \oplus R^s(-2)$ of $R^s$-bimodules, one can easily check that \begin{equation} B_s \ot B_s \cong
B_s(1) \oplus B_s(-1). \label{rank1BS} \end{equation} This isomorphism categorifies equation \eqref{rank1hecke}. Under certain assumptions on the realization, the categorification of
equation \eqref{rank2hecke} comes from an explicit description of $B_w$ as the image of a certain idempotent, for every $w$ contained in a standard rank 2 parabolic subgroup. More details
can be found in \cite[\S 4]{Soe3} or \cite{EDihedral}. In this paper we will give an alternate proof of the categorification of \eqref{rank2hecke} in more generality. It follows that $\e$
is a homomorphism of $\Zvv$-algebras.

Once one knows that $\e$ is a homomorphism, the statement that it is an isomorphism follows from the  classification of indecomposable Soergel bimodules and the fact that their characters are upper triangular. More precisely, if we fix a rex $\un{w}$ for every $w \in W$ then the set $\{ H_{\un{w}} \; | \; w \in W\}$ is easily seen to be upper triangular in the standard basis of the Hecke algebra with respect to the Bruhat order, and hence is a basis. On the other hand, Theorem \ref{SoergelThm1} shows that $\{ [B_{\un{w}}] \; | \; w \in W \}$ is upper triangular in the basis $\{ [B_w] \; | \; w \in W \}$ for $[\SBim]$. As $\un{H}_{\un{w}}$ is mapped to $[B_{\un{w}}]$ it follows that $\e$ is an isomorphism.

In fact, Soergel shows that the character map
\[
\ch : [\SBim] \to \HB
\]
 discussed in previous section provides an inverse to $\e$. The character map is rather subtle. In general it is not known how to describe the element $\ch(B_w)$ in the Hecke algebra. However, if $\Bbbk$ is a field of characteristic zero Soergel proposed the following conjecture, which came to be known as \emph{Soergel's conjecture}:
 
\begin{conj} If $\Bbbk$ is a field of characteristic $0$ then $\ch(B_w) = \un{H}_w$. \end{conj}

It is immediate that Soergel's conjecture implies the Kazhdan-Lusztig positivity conjectures. Earlier work of Soergel \cite{Soe1} showed that his conjecture also implies the
Kazhdan-Lusztig conjecture. Soergel's conjecture has recently been proved by the authors in \cite{EWHodge} for realizations (not necessarily reflection faithful) defined over $\RM$, and satisfying a certain positivity property.

\begin{remark} This statement of Soergel's conjecture is known to fail when $\Bbbk$ is a field of finite characteristic $p$. The image of $[B_w]$ in $\HB$ is known as the $p$-\emph{canonical basis}, and its computation is an interesting open problem (see \cite{Wpcan}).
\end{remark}

\begin{remark} Soergel's conjecture is false for certain complex realizations of the affine Weyl group of type $A$ which do not admit any real form (as in part (6) of Example
\ref{ex:realizations}, when $q$ is a root of unity). See \cite{EQAGS} for details. Hence the assumption that the underlying representation be defined over $\RM$ is essential for Soergel's
conjecture to hold. (The authors' proof of Soergel's conjecture uses positivity considerations in a crucial way!) \end{remark}

The fact that $\un{H}_s$ is self-biadjoint in the standard pairing is categorified by the fact that the functor of tensoring with $B_s$ is its own left and right adjoint. Just as in Remark \ref{rmk:traces pairings}, any Hom space between Bott-Samelson bimodules can be understood after adjunction in terms of $\Hom(R,B_{\un{w}})$. Soergel's Categorification Theorem
combined with Corollary \ref{cor:defects} implies that this space is free as a left $R$-module with a homogeneous basis in bijection with the subsequences of $\un{w}$ with end-point $e$.
Following the work of Libedinsky \cite{Lib}, we will show that this basis can be constructed in a natural way.

\begin{remark} Suppose that $(W_1, S_1)$ and $(W_2, S_2)$ are Coxeter systems and $S_1 \subset S_2$ is an inclusion realizing $W_1$ as a standard parabolic subgroup of $W_2$. Suppose also
that we have a $W_1$-equivariant inclusion of realizations $\hg_1 \subset \hg_2$. There is an obvious functor $\iota \co \BSBim_{S_1} \to \BSBim_{S_2}$ sending $B^1_s \mapsto B^2_s$.
Soergel's Categorification Theorem implies that this functor is ``fully-faithful after base change.'' That is, if $R_1$ and $R_2$ denote the polynomial functions on $\hg_1$ and $\hg_2$
respectively then \[ R_2 \otimes_{R_1} \Hom(B^1_{\un{w}},B^1_{\un{w}'}) = \Hom(B^2_{\un{w}}, B^2_{\un{w}'}). \] Said another way, homomorphisms from $B_{\un{w}}$ to $B_{\un{w}'}$ only
depend on the simple reflections which occur in $\un{w}$ and $\un{w}'$ (up to the action of polynomials). The (diagrammatic) proof that these functors are fully faithful after base change
will be clear later in this paper. As a result, one may define the category of Soergel bimodules for $(W, S)$ even when $S$ is infinite, either by working with an infinite dimensional
realization of $W$, or by taking a limit over the categories associated to finite subsets of $S$.\label{includingindexsets} \end{remark}

\subsection{Localization}

As $\Bbbk$ is a domain, so is $R$. Let $Q$ denote the fraction field of $R$, which is an (ungraded) field living over $\Bbbk$. One can extend any $R$-module to a $Q$-module in the usual
way. We have already shown in general that any tensor product of standard and Soergel bimodules is free as a left (or right) $R$-module, and that morphism spaces are torsion-free.
Therefore, extension of scalars from $R$ to $Q$ on the left (or right) will yield a free $Q$-module, and will act faithfully on morphisms.

\begin{lem}\label{localizationononeside} For any $\un{w} = s_1s_2 \dots s_d$ one has an isomorphism of right $Q$-modules
\[
B_{\un{w}} \otimes_R Q = Q \otimes_{Q^{s_1}} Q \otimes_{Q^{s_2}}  \dots \otimes_{Q^{s_d}} Q.
\]
The analogous statement can be made for standard bimodules, or tensor products of standard and Bott-Samelson bimodules.
\end{lem}

\begin{proof} We show that $B_s \otimes_R Q \cong R \otimes_{R^s} Q$ is isomorphic to $Q \ot_{Q^s} Q$, which implies the lemma by a simple induction. One has an obvious inclusion $i : R \otimes_{R^s} Q \into Q \ot_{Q^s} Q$ and it is enough to show that this is surjective. However, any $0 \ne f \in R$ divides the $s$-invariant polynomial $fs(f)$ and in $Q \ot_{Q^s} Q$ we have
\[
s(f) \otimes \frac{1}{fs(f)} = \frac{1}{f} \otimes 1.
\]
It follows that $i$ is surjective, and the lemma follows. The analogous proof for standard bimodules is even simpler. \end{proof}

In particular, base change from $R$ to $Q$ on the right will send a Bott-Samelson bimodule to a $Q$-bimodule, not just an $(R,Q)$-bimodule, and the base change functor is monoidal. Let
$\BSBim_Q$ denote the essential image of $\BSBim$ inside the category of $Q$-bimodules. By the above lemma, it is the full subcategory of $Q$-bimodules generated by the bimodules
$B_{s,Q}\define Q \ot_{Q^s} Q$. Let $\SBim_Q$ denote its Karoubi envelope. We define $\StdBim_Q$ similarly.

As $Q$-bimodules we have a splitting \begin{equation} \label{eq:split} B_{s,Q} \cong Q \oplus Q_s. \end{equation} One way to observe this is that the sequences \eqref{ses1} and
\eqref{ses2} are split. In fact, they ``split each other:" the composition $R \to B_s \to R$ is multiplication by $\a_s$, which becomes invertible in $Q$. Therefore, any Soergel bimodule
over $Q$ splits into standard bimodules. When the realization is faithful one has \[ \Hom(Q_w,Q_v) = \begin{cases} Q & \text{if $w = v$} \\ 0 & \text{otherwise,} \end{cases} \] so that
standard bimodules over $Q$ are indecomposable and pairwise non-isomorphic. It follows that $\SBim_Q$ is equivalent to the (additive closure of the) 2-groupoid of $W$ over $Q$.

We conclude $\SBim_Q$ has a drastically different behaviour to $\SBim$ --- $\SBim$ is far from being a 2-groupoid! In the following, we will refer to the faithful monoidal functor
\begin{gather*}
\SBim \to \SBim_Q
\end{gather*}
as \emph{localization}. It plays an important role in what follows, essentially because $\SBim_Q$ is such a simple category. We already know that Hom spaces between various $Q_w$ are either rank 1 or 0. Because the passage to the localization is injective, and because morphisms between various $Q_w$ are so easy,
we can study morphisms in $\SBim$ based on what their restrictions are to simple summands in $\SBim_Q$ after localization. Given an explicit morphism $f$ in $\BSBim$, we will be able to
write down explicitly its restriction to each standard summand of the source and target in $\SBim_Q$, and in doing so, will be able to conclude whether or not $f=0$.

\begin{remark} One may use the homogeneous fraction field $Q'$ instead of the full fraction field, and the results in this paper pertaining to localization will be (essentially) unchanged.
The homogeneous fraction field is graded; however, $Q' \cong Q'(2)$ so that the grading does not yield any interesting invariants (even the parity is irrelevant, as investigation reveals).
Though more difficult, all the localization results in this paper should hold if one only inverts the simple roots $\a_s$ and their $W$-conjugates. One advantage to using
$Q$ over these alternatives is that $Q$ is a field. \end{remark}

\part{Diagrammatics}

\label{pt:diagrammatics}

\section{Diagrammatics for Standard Bimodules}
\label{sec-diagstd}

\subsection{Diagrammatic definition}
\label{subsec-StdDiagDef}

We assume that the reader is familiar with diagrammatics for monoidal categories with biadjunctions. An introductions to the topic may be found in \cite{Lauda}.

Fix $(W,S)$. In the paper \cite{EWFenn}, we give a diagrammatic monoidal presentation for the 2-groupoid of $W$ over $\Bbbk$. As we have seen, this is equivalent to the category $\StdBim_0$. We recall that definition, beefing it up to add polynomials and obtain the 2-groupoid of $W$ over $R$, which is equivalent to $\StdBim$. We introduce it in a way which will aid our definition of Soergel diagrammatics in the next chapter.

\begin{defn} \label{defn: graph}(For the purposes of this paper) a \emph{graph with boundary} for $(W,S)$ is an \emph{isotopy class} of a graph with boundary, properly embedded in the planar strip $\R
\times [0,1]$. In other words, all vertices of the graph lie on the interior of the strip, and edges may terminate either at a vertex or on the boundary of the strip. We call the place
where an edge meets the boundary a \emph{boundary point}; boundary points are not vertices. Edges may also form closed loops. We also allow decorations to float in regions cut out by the graph;
these are thought of as 0-valent vertices with labels (we will always clarify in subdefinitions what labels are allowed). Note that a connected component either contains a vertex, is a
loop, is an arc between two boundary points, or is a decoration.
The number of vertices and the number of components are required to be finite. Every edge is ``colored" (i.e. labelled) by
an element of $S$. Isotopy is allowed to move boundary points within the boundary of the strip. The boundary points on $\R \times \{0\}$ (resp. $\{1\}$) give a finite sequence of colored
points, known as the \emph{bottom boundary} (resp. \emph{top}) of the graph. We will often abuse notation and refer to a representative of the isotopy class as a graph. \end{defn}

We fix a realization of $(W,S)$, yielding the polynomial ring $R$.

\begin{defn} A \emph{standard graph} is a graph with boundary for which: the only decorations are boxes labelled by homogeneous $f \in R$; for every vertex there is a pair $s,t \in S$
such that the vertex is $2m_{st}$-valent and its edges alternate in color between $s$ and $t$. The \emph{degree} of the graph is the sum of the degrees of every box. \end{defn}

We draw standard graphs with dashed edges, to distinguish them from the Soergel graphs to be defined later. Here is an example where $S=\{r,b,g\}$ (for \emph{r}ed, \emph{b}lue and \emph{g}reen) and $m_{r,g}=5$, $m_{r,b}=3$, $m_{b,g}=2$:

\igc{1}{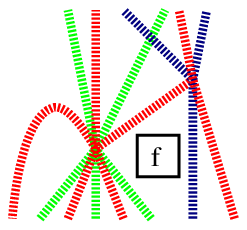}

\begin{defn} Let $\DStd$ denote the $\Bbbk$-linear additive graded monoidal category defined as follows. Objects are (direct sums of grading shifts of) sequences $\xb$ of elements of
$S$, with monoidal structure given by concatenation. The empty sequence $\emptyset=\1$ is the monoidal identity. We may draw these objects as colored points on a line, assigning one
color to each index in $S$. The space $\Hom_{\DStd}(\xb,\yb)$ will be the free $\Bbbk$-module generated by standard graphs with bottom boundary $\xb$ and top boundary $\yb$, modulo the
relations below. Hom spaces will be graded by the degree of the graphs, and all the relations below are homogeneous.

The first (unwritten) relation states that boxes add and multiply just as homogeneous polynomials in $R$ do. Thus we could have chosen our generating decorations to merely be boxes labelled by generators of $R$. The other tacit relation is that cups and caps form a collection of biadjoint pairs, and that all maps are cyclic with respect to these
biadjunctions, having the appropriate rotational symmetry. This is implied by the definition of a graph, because an isotopy class of diagram unambiguously represents a morphism:

\igc{1}{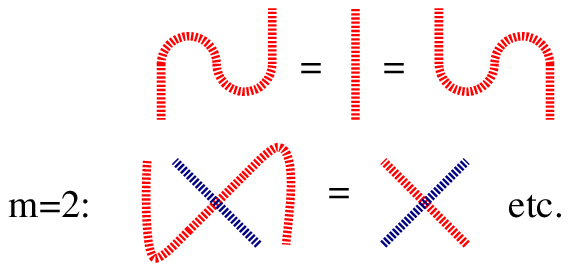}

The following relation describes how to slide boxes through edges.

\begin{equation} \label{slidepolythrustd} \ig{1}{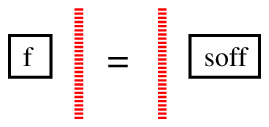} \end{equation}

The following relations hold for any $s \in S$. As usual, empty space here denotes the identity endomorphism of $\1$.

\begin{equation} \label{circlestd} \ig{1}{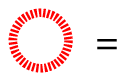} \end{equation}

\begin{equation} \label{capcupstd} \ig{1}{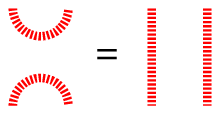} \end{equation}

These two relations imply that cups and caps form inverse isomorphisms from $ss$ to $\1$. The following relation, which states that the $2m$-valent vertex gives an isomorphism from
$\ubr{sts\ldots}{m}$ to $\ubr{tst\ldots}{m}$, holds for any $s,t \in S$ with $m_{s,t} < \infty$.

\begin{equation} \label{2misomstd} \ig{1}{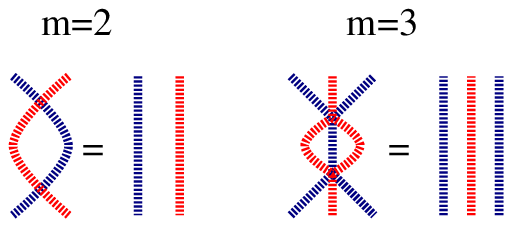} \end{equation}

The remaining relations come from triples $s,t,u \in S$ such that the corresponding parabolic subgroup is finite. By the classification of finite Coxeter groups the finite rank three parabolic subgroups can only be one of the following types:

\begin{eqnarray} (A_1 \times I_2(m)) && \ig{1}{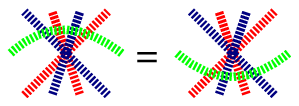} \label{A1I2mstd} \end{eqnarray}

A specific example, when $m=2$, is:

\begin{eqnarray} (A_1 \times A_1 \times A_1) && \ig{1}{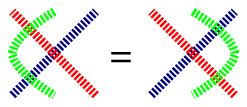} \label{A1A1A1std} \end{eqnarray}
	
\begin{eqnarray} (A_3) && \ig{1}{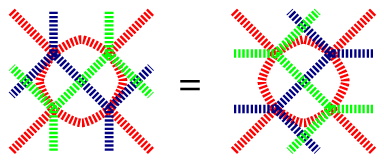} \label{A3std} \end{eqnarray}

\begin{eqnarray} (B_3) && \ig{1}{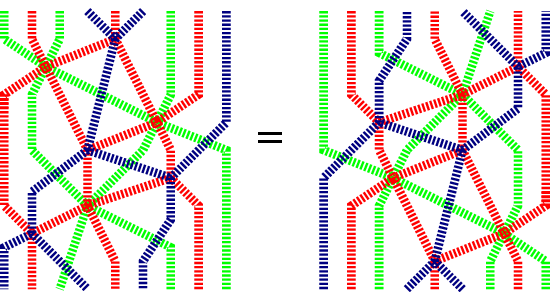} \label{B3std} \end{eqnarray}

\begin{eqnarray} (H_3) && \ig{1}{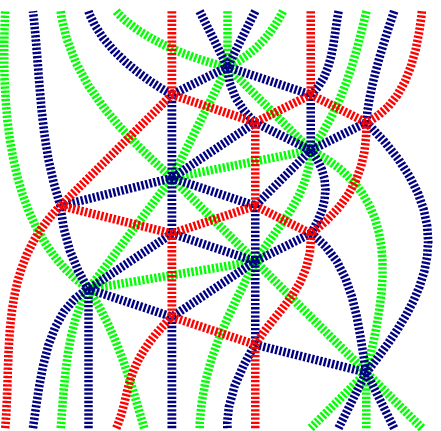} = \ig{1}{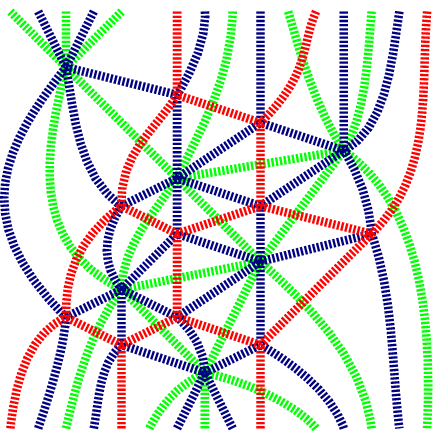} \label{H3std} \end{eqnarray}
This concludes the definition of $\DStd$. \end{defn}

\begin{defn} If we do not allow any grading shifts, and only allow $\Bbbk$-linear combinations of standard graphs without polynomials, we get a monoidal category $\DStd_0$ consisting entirely of degree $0$ maps. \end{defn}

\begin{remark} When working with $\DStd_0$, the base ring $\Bbbk$ is entirely arbitrary. When defining $\DStd$, as for $\StdBim$, one needs a realization of $W$ to obtain the ring $R$ but
no additional restrictions on $R$ are necessary. The main theorem \ref{FStdEquiv} below holds more generally for any (graded) ring $R$ with a faithful $W$-action.\end{remark}

Now we wish to show that $\DStd \cong \StdBim$, and that $\DStd_0 \cong \StdBim_0$.

\begin{defn} We define a functor $\FStd \co \DStd \to \StdBim$. Given an expression ${\un{x}} = st\dots u$, $\FStd$ sends the corresponding object to $R_s \ot R_t \ot \cdots \ot R_u$. The box containing $f$ is sent to multiplication by $f$. The cups, caps, and $2m$-valent vertices are all sent to the isomorphisms of $R$-bimodules which send 1-tensors to 1-tensors. This restricts to a functor
$\FStd_0 \co \DStd_0 \to \StdBim_0$. \end{defn}

\begin{claim} The functors $\FStd$ and $\FStd_0$ are well-defined. \end{claim}

\begin{proof} For $\FStd_0$, we need only check the relations minus (\ref{slidepolythrustd}). However, all relations hold immediately because both sides are bona fide maps in
$\StdBim_0$, and both sides send the 1-tensor to the 1-tensor. In addition, it is clear that equation (\ref{slidepolythrustd}) holds for standard bimodules, so $\FStd$ is well-defined.
\end{proof}

$\FStd$ is obviously essentially surjective. Suppose that the realization is faithful. Whenever a Hom space in $\StdBim$ is non-zero, it is generated by the isomorphism which sends the
1-tensor to the 1-tensor. This is clearly in the image of $\FStd$, so $\FStd$ is full.

\begin{thm} The functors $\FStd_0$ and $\FStd$ are equivalences, so long as the realization is faithful. \label{FStdEquiv} \end{thm}

This theorem is the main result of \cite{EWFenn}. Essentially, one needs to show that the space of morphisms from ${\un{x}}$ to ${\un{y}}$ is one-dimensional if $x=y$ (it is easy to see
that it is $0$-dimensional otherwise). Using biadjunction and various isomorphisms in both categories, this reduces to the fact that standard diagrams with empty boundary are polynomial
multiples of the empty diagram. The proof is actually topological in nature. Using Fenn's theory of diagrams \cite{Fenn} we are able to relate standard diagrams with empty boundary to
elements of the second homotopy group of a space related to the Coxeter complex.

We will discuss one important feature of the story in the next section.

\subsection{Rex moves and rex graphs}
\label{subsec-RexMoves}

The terminology of this section is ad hoc, and non-standard. This is due to ignorance, not malice.

For each $w \in W$, let $\G_w$ denote the set of all reduced expressions for $w$. This can be given the structure of a connected graph, the \emph{rex graph},  where two reduced expressions are connected by an edge if they differ by a single application of a braid relation \eqref{braidreln}. The edge itself can be labelled by the pair $\{s,t\}$ corresponding to the braid
relation; in other words, the edges are labelled by finitary rank 2 subsets of $S$. We shall only distinguish here between two different labellings: \emph{distant edges} for which
$m_{s,t}=2$, and those for which $m_{s,t}>2$. Let $\oG_w$ denote the graph obtained from $\G_w$ by contracting the distant edges.

We can associate a vertex ${\un{x}}$ of $\G_w$ with the standard bimodule $R_{\un{x}}$, isomorphic to $R_w$. We think of a path as a \emph{rex move}, a sequence of braid relations traversing the rex graph to a new reduced expression. To an
edge labelled $\{s,t\}$ from ${\un{x}}$ to ${\un{y}}$, we can associate the $2m_{s,t}$-valent vertex, which is a morphism from $R_{\un{x}}$ to $R_{\un{y}}$. Therefore we can associate a morphism in $\DStd_0$
to each path in $\G_w$, and we also call this morphism a rex move. We can never construct a morphism with caps or cups in this fashion, since those involve non-reduced expressions. In
fact, $\DStd$ essentially encodes a study of the graph of all expressions, reduced and non-reduced, for $w$; this point is made clear in \cite{EWFenn}, and may become more obvious after
the remainder of this section.

The relations of $\DStd$ which do not involve cups and caps, namely \eqref{2misomstd} through \eqref{H3std}, all come from loops (i.e. paths from a vertex to itself) in $\G_w$. That is,
each side of the relation comes from a path from $\un{x}$ to $\un{x}'$ in $\G_w$, and the relation states that the corresponding morphisms are equal. Equivalently, once \eqref{2misomstd}
is known, the relation states that the loop $\un{x} \to \un{x}' \to \un{x}$ is equal to the trivial loop at $\un{x}$.

For each edge there is a \emph{boring} loop which follows the edge and then follows it in reverse, which corresponds to \eqref{2misomstd}. In addition, each finitary rank 3 subset of $S$
gives rise to a kind of loop, which first appears in the longest elements of these parabolic subgroups.

\begin{ex} We give two examples of type $A_1 \times I_2(m)$. The labeling of indices in $S$ should be obvious. Type $A_1 \times A_1 \times A_1$: \igc{1}{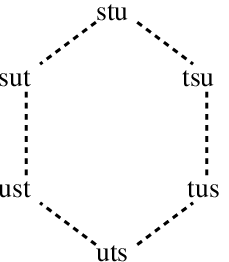} Type $A_1 \times
A_2$: \igc{1}{} \end{ex}

\begin{ex} Here is $\G_{w_0}$ for type $A_3$. The orientations on the arrows will be explained at the end of this section. \igc{1}{121321graph} \end{ex}

\begin{ex} Here is $\oG_{w_0}$ for type $B_3$. This can be deduced from \eqref{B3std}, and we let the avid reader do the same for $H_3$. The red and blue edges correspond to different
parabolic subsets. \igc{1}{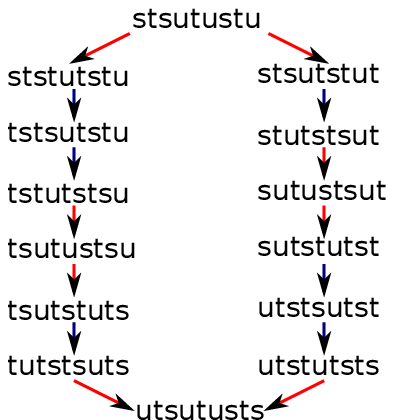} \end{ex}

There is one additional kind of loop, which arises when two braid relations can be applied to disjoint parts of a rex. This is called a \emph{disjoint square}. In $\DStd$, this corresponds
to the fact that distant pictures commute in a diagrammatic category. Disjoint squares need not involve disjoint colors.

\begin{ex} A disjoint square.  \igc{1}{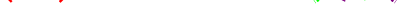} \end{ex}

It is known \cite[Chapter 2, \S 5]{Ronan} that these loops generate the set of all loops in $\G_w$, in a sense which the reader can intuit. This corresponds topologically to the fact that
$\pi_2$ of the truncated dual Coxeter complex is generated by the finitary rank 3 subsets. In other words, the relations \eqref{2misomstd} through \eqref{H3std} are sufficient to imply
that any two paths from $\un{x}$ to $\un{x}'$ will induce the same morphism in $\DStd$. The theory of reduced expressions is enough to say that diagrams in $\DStd$ without cups and caps form a one-dimensional space. The proof in \cite{EWFenn} deals with non-reduced expressions as well.

Let us prepare the reader for the following chapter. Note that the only loops in $\oG_w$ (in addition to the boring loops and disjoint squares) come from \emph{connected} finitary rank 3
subsets of $S$, which are of the type $A_3$, $B_3$, or $H_3$. The equality of the two sides in \eqref{A3std} is often called the \emph{Zamolodzhikov tetrahedron equation}, and so we refer to the
$B3$ and $H3$ equations as Zamolodzhikov equations as well.

The story will be complicated much further when we start working with Bott-Samelson bimodules instead of standard bimodules. For two rexes ${\un{x}}$ and ${\un{y}}$ connected by an edge
labelled by the subset $J=\{s,t\}$, there will be a corresponding map from $B_{\un{x}}$ to $B_{\un{y}}$ which projects from $\ubr{B_s \ot B_t \ot \cdots}{m} \to B_{w_J}$ and then includes
$B_{w_J} \to \ubr{B_t \ot B_s \ot \cdots}{m}$. To any path in $\G_w$ we can still associate a morphism in $\BSBim$, but two paths will not always agree. There is no \eqref{2misomstd}; the
boring loop is projection to a summand, not the identity map. It is true, at least, that following the boring loop twice is the same as following it once. After localizing, however, rex
moves in $\BSBim$ restricted to the unique ``top" summand $Q_w$ will give exactly the rex moves in $\DStd$. In particular, this implies that two paths give rex moves in $\BSBim$ which
agree modulo ``lower terms", where lower terms are maps which, when localized, give the zero map on $Q_w$. For any ${\un{x}},{\un{y}} \in \G_w$, the space of maps from $B_{\un{x}}$ to
$B_{\un{y}}$ modulo lower terms is free of rank 1 over $R$, as we shall see.

\section{Diagrammatics for Soergel Bimodules} \label{sec-diagsbim}

In this chapter we define a diagrammatic category $\DC$ by generators and relations. We provide a functor $\FC$ from this category to Bott-Samelson bimodules. The main result of this paper
is that $\FC$ is an equivalence, and the proof will be given in later chapters.

Fix a realization $\hg$ of $(W,S)$ over $\Bbbk$, as in section \ref{subsec-realization}, with $S$ finite (though see Remark \ref{SCnotfinite}). As before, we let $R = \Bbbk[\hg]$ denote
the coordinate ring of the reflection representation, $\a_s$ the simple root associated to $s \in S$, $\pa_s$ the Demazure operator. We let $Q$ denote the field of fractions of $R$.

For reasons of simplicity, we will assume the realization is balanced in the first pass. Later, in section \ref{nonbalanced}, we treat the unbalanced case. For reasons to become clear in
that section, we must always assume our realization is even-balanced.

\subsection{Generators and relations}

\begin{defn} A \emph{Soergel graph} is a type of graph with boundary (see Definition \ref{defn: graph}). The only decorations are boxes labelled by homogeneous $f \in R$. The vertices in this graph are of 3 types (see Figure \ref{thegenerators}):

\begin{itemize} \item Univalent vertices (dots). These have degree $+1$. \item Trivalent vertices, where all three adjoining edges have the same color. These have degree $-1$. \item
$2m$-valent vertices, where the adjoining edges alternate in color between two elements $s \ne t \in S$, and $m_{st}=m<\infty$. These have degree $0$. \end{itemize}

The \emph{degree} of a Soergel graph is the sum of the degree of each vertex, and the degree of each box. \end{defn}

\begin{figure} \label{thegenerators} \caption{The vertices in a Soergel graph} $\ig{1}{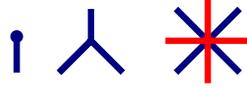}$ \end{figure}

When there is no ambiguity we refer to a Soergel graph merely as a ``graph," even though it is an isotopy class of embedded graph. We may also wish to discuss Soergel graphs on the planar disc, on an annulus, etc.

%
%

A \emph{boundary dot} is a connected component of a graph consisting of a dot connected to the boundary by a single edge. A Soergel graph that contains no dots or trivalent vertices is a standard graph, as in the previous chapter. However, we draw the edges as solid lines, not dashed lines.

\begin{defn} Let $\DC$ (or if there is ambiguity, $\DC_{S}$) denote the $\Bbbk$-linear monoidal category defined as follows. Objects are sequences $\un{w}$, sometimes denoted $B_{\un{w}}$, with monoidal structure given by concatenation. The space $\Hom_{\DC}(\un{w},\un{y})$ is the free $\Bbbk$-module generated by Soergel graphs with
bottom boundary $\un{w}$ and top boundary $\un{y}$, modulo the relations below. Hom spaces will be graded by the degree of the graphs, and all the relations below are homogeneous.

\emph{The polynomial relations:}
\begin{gather}
\label{alphais}
\begin{array}{c}
\tikz[scale=0.9]{
\draw[dashed,gray] (0,0) circle (1cm);
\draw[color=red] (0,0.5) -- (0,-0.5);
\node[draw,fill,circle,inner sep=0mm,minimum size=1mm,color=red] at (0,0.5) {};
\node[draw,fill,circle,inner sep=0mm,minimum size=1mm,color=red] at (0,-0.5) {};
} \end{array}
=
\begin{array}{c}
\tikz[scale=0.9]{
\draw[dashed,gray] (0,0) circle (1cm);
\node at (0,0) {$\alpha_s$};
} \end{array}, \\
\label{dotforcegeneral} 
\begin{array}{c}
\tikz[scale=0.9]{
\draw[dashed,gray] (0,0) circle (1cm);
\draw[color=red] (0,-1) to (0,1);
\node at (-0.5,0) {$f$};
} \end{array}
=
\begin{array}{c}
\tikz[scale=0.9]{
\draw[dashed,gray] (0,0) circle (1cm);
\draw[color=red] (0,-1) to (0,1);
\node at (0.5,0) {${\color{red}s}(f)$};
} \end{array}
+
\begin{array}{c}
\tikz[scale=0.9]{
\draw[dashed,gray] (0,0) circle (1cm);
\draw[color=red] (0,-1) to (0,-0.5);
\draw[color=red] (0,1) to (0,0.5);
\node[color=red,draw,fill,circle,inner sep=0mm,minimum size=1mm] at (0,-0.5) {};
\node[color=red,draw,fill,circle,inner sep=0mm,minimum size=1mm] at
(0,0.5) {};
\node at (0,0) {$\partial_{\color{red}s} f$};
} \end{array}.
\end{gather}

\emph{The one color relations:}
\begin{equation} \label{assoc1} \ig{1}{assoc1} \end{equation}	
\begin{equation} \label{unit} \ig{1}{unit} \end{equation}
\begin{equation} \label{needle} \ig{1}{needle} \end{equation}

\emph{The two color relations:} In order to simplify this presentation greatly, we will assume that our realization is balanced. For discussion of the general case, see Section \ref{nonbalanced}.

The color scheme depends slightly on the parity of $m=m_{st}<\infty$. We give one example of each relation for each
parity; the reader can guess the general form.

\begin{equation} \label{assoc2} \ig{1}{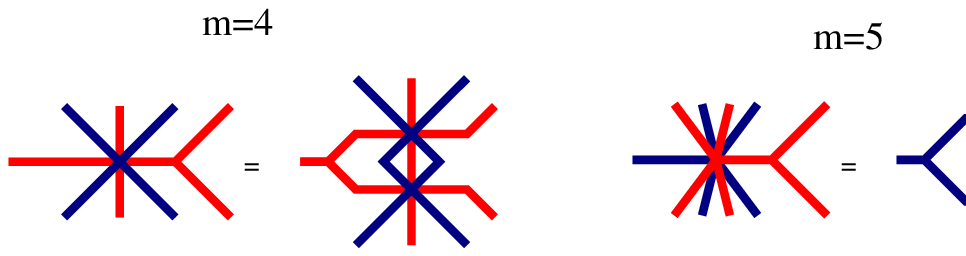} \end{equation}
\begin{equation} \label{dot2m} \ig{1}{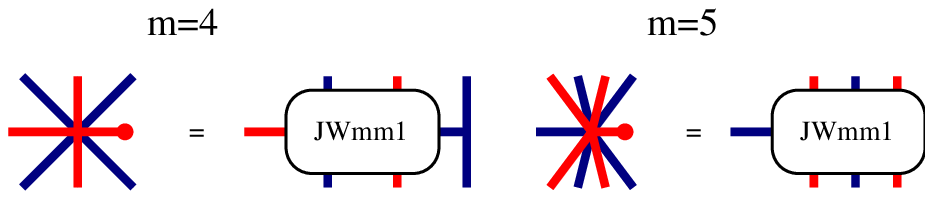} \end{equation}

In equation \eqref{dot2m} above, the \emph{Jones-Wenzl morphism} $JW_{m-1}$ is a $\Bbbk$-linear combination of graphs constructed only out of dots and trivalent vertices. It will be
defined and discussed in the next section.

\emph{The three color relations:} It will be clear from the graphs which colors represent which indices. These relations are identical to those defined for $\DStd$ earlier, with the exception of
$H_3$.

For a triplet of colors forming a sub-Coxeter system of type $A_1 \times I_2(m)$ for $m<\infty$, we have

\begin{eqnarray} (A_1 \times I_2(m)) && \ig{1}{A1I2m} \label{A1I2m} \end{eqnarray}

A specific example, when $m=2$, is the case $A_1 \times A_1 \times A_1$:

\begin{eqnarray} (A_1 \times A_1 \times A_1) && \ig{1}{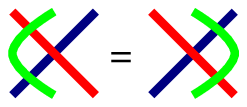} \label{A1A1A1} \end{eqnarray}

The last three relations are for types $A_3$, $B_3$, and $H_3$ respectively, and are known as the \emph{Zamolodzhikov relations}. Unfortunately, $H_3$ is not complete.

\begin{eqnarray} (A_3) && \ig{1}{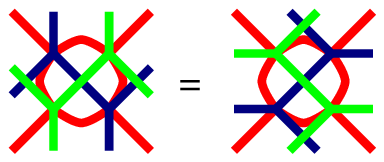} \label{A3} \end{eqnarray}

\begin{eqnarray} (B_3) && \ig{1}{B3} \label{B3} \end{eqnarray}

\begin{eqnarray} (H_3) && 
\ig{1}{H3left}
- 
\ig{1}{H3right}
= 
\text{lower terms}
\label{H3} \end{eqnarray}

This concludes the definition of $\DC$. \end{defn}

\begin{remark} \label{incomplete} What it means to be a ``lower term" in \eqref{H3} was discussed in Section \ref{subsec-RexMoves}. Exactly what the lower terms are is currently unknown.
There is a correct answer, which is whatever makes the functor $\FC$ well-defined. We discuss how this could be addressed computationally below. \end{remark}


Note that (graded) Hom spaces are enriched in graded $R$-bimodules, since one can put a polynomial in a box and place it in the leftmost or rightmost region. We will see much later that all
hom spaces are free when considered as right or left $R$-modules, however this is far from clear at this stage. Because of \eqref{dotforcegeneral}, any diagram is equal to a linear combination of diagrams where polynomials only appear in the left-hand region. It is also an easy consequence of \eqref{dotforcegeneral} that, for $f \in R^s$, we have the \emph{polynomial sliding relation}:
\begin{equation} \label{dotslidegeneral} \ig{1}{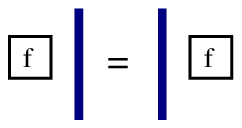} \end{equation}

Fix $k \ge 0$. Relations \eqref{unit} and \eqref{assoc1} imply that any two non-empty trees of a single color connecting $k$ points on the boundary are equal.

Let us assume Demazure Surjectivity, so that there exists some $\d$ with $\pa_s(\d)=1$. It is not difficult to use the one color relations to prove the decomposition \begin{equation}
\label{BsBsDecomp} B_s \ot B_s \cong B_s(1) \oplus B_s(-1). \end{equation} One simply needs to find two inclusion maps $i_1,i_2 \co B_s \to B_s \ot B_s$ of degree $-1,+1$ respectively, and
two projection maps $p_1,p_2$ of degree $+1,-1$ respectively, such that the usual relations are satisfied: $p_1 i_1 = p_2 i_2 = \1_{B_s}$, $p_1 i_2 = p_2 i_1 = 0$, and $\1_{B_s \ot B_s} =
i_1 p_1 + i_2 p_2$. One can choose these maps as follows. \igc{1}{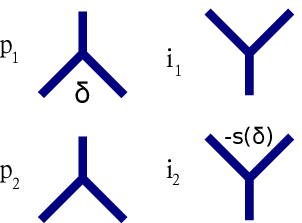}
This splitting is not canonical. In fact, the space of maps in degree $-1$ is one-dimensional, so that $i_1$ and $p_2$ are canonical. However, there may be many choices of $\d$.

\begin{remark} In fact, the choice of splitting is even more general than a choice of $\d$. For instance one may have
\begin{equation} \ig{1}{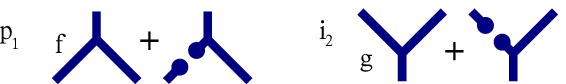} \label{iiidemp2} \end{equation}
so long as $f + g = -\a_s$. Using \eqref{dotforcegeneral}, the example above has $f = \d - \a_s$ and $g= -\d$. However, this more general splitting exists even when $\hg$ does not satisfy Demazure Surjectivity: for instance, one could let $f = -\a_s$ and $g=0$. \end{remark}

\begin{remark} Using \eqref{dotforcegeneral}, one can replace two facing dots (a ``broken edge") with a sum of two diagrams having a complete edge, one with $\d$ on the left and the other
with $-s(\d)$ on the right. We call this procedure \emph{fusing} two dots. However, in the absence of Demazure Surjectivity, one can only fuse two dots up to a scalar, and the double
leaves ``basis" we define in the next chapter will no longer span all morphisms. Thus Demazure Surjectivity is still very important for proper behavior of the diagrammatic category. \end{remark}

See Remark \ref{twocolorrelncatfd} for the categorification of \eqref{rank2hecke}.

\begin{remark} \label{SCnotfinite} This remark is a diagrammatic analogue of Remark \ref{includingindexsets}. Given an inclusion of Coxeter systems $S_1 \subset S_2$, there is natural
functor $\iota \colon \DC_1 \to \DC_2$ which sends a Soergel graph for $S_1$ to itself, viewed as a Soergel graph for $S_2$ (reinterpreting boxes via the map $R_1 \to R_2$). As in the
bimodule case, this functor is not full for the trivial reason that $R_2$ is bigger than $R_1$, and thus $\DC_{S_2}$ has more boxes. However, it will follow from the main theorem of this
paper that $\iota$ is fully faithful after base change on the left from $R_{1}$ to $R_{2}$.

Suppose that a color $s$ does not appear on the boundary of a graph. The fact that $\iota$ is full after base change implies that we may manipulate the graph using our relations so that it
is in the span of graphs with polynomials where the color $s$ does not appear at all. This ``color removal" operation can be performed simultaneously for any number of colors. The proofs
used in type $A$ in \cite{EKh} actually provided a direct graphical algorithm for removing extraneous colors from a graph (in certain cases). Such an algorithm in general type would be
interesting, and remains an open problem. Nonetheless, we prove indirectly that extraneous colors can be removed by constructing a basis without them.

One can use these inclusions to define $\DC$ for Coxeter systems where $S$ is infinite, as a limit over the finite subsets of $S$. Any diagram contains finitely many colors and thus lies
inside a finite subcategory. The ring $R$ will no longer be Noetherian, but aside from that, Hom spaces will have all the nice finiteness properties (as free $R$-modules) that they have
when $S$ is finite. \end{remark}

\begin{remark} \label{generatorsinorder} As a monoidal category, $\DC$ can be defined with generators and relations. The generating morphisms are the vertices above, and cups and caps of
each color. However, relation \eqref{unit} implies that cups and caps can be constructed out of dots and trivalent vertices. Therefore, a full set of generators for $\DC$ is as listed in the introduction. 
When constructing a functor to $\BSBim$, we will also need to check the unwritten isotopy relations. This is standard in diagrammatic categories; see, for instance, \cite{Lauda}. \end{remark}

\subsection{Jones-Wenzl morphisms} \label{subsec-JW}

Presenting the Bott-Samelson category in rank $2$ is the topic of \cite{EDihedral}. The calculations and proofs are too long to duplicate or fully discuss in this paper. If the reader is
willing to accept the Jones-Wenzl morphisms as black boxes\footnote{no pun intended!}, satisfying the properties stated in this section and the next, then the reader need not consult
\cite{EDihedral}. We assume the reader is familiar with quantum numbers and Temperley-Lieb algebras; background can be found in \cite{EDihedral}.

Fix $m<\infty$. Consider the Temperley-Lieb algebra $TL$ over $\ZM[z]$ with $m-1$ strands, where the circle is evaluates to $-z$. In any $\Bbbk[z]$-algebra we will use quantum
numbers to indicate the images of certain polynomials in $z$, where $[2]$ is the image of $z$. For instance, $[3]$ is the image of $z^2-1$. When we say that $q$ is specialized to a
primitive $2m$-th root of unity in a $\ZM[z]$-algebra, we mean that $m$ satisfies the minimal polynomial of $[2]$ at this root of unity, which equates to the statements that $[m]=0$, $[k]\ne 0$ for $0< k < m$, and $[m-1]=1$.

After inverting some quantum numbers, Temperley-Lieb algebras contain elements known as \emph{Jones-Wenzl projectors}, and the Jones-Wenzl projector on $m-1$ strands $JW_{m-1}$ is known to
be negligible and rotation-invariant when $q$ is specialized to a primitive $2m$-th root of unity. The Jones-Wenzl projector $JW_m$ can be defined whenever certain quantum binomial
coefficients are invertible. Using diagrammatics for the Temperley-Lieb algebra, the Jones-Wenzl projector can be described as a linear combination of crossingless matchings.

\begin{ex} $\ig{1}{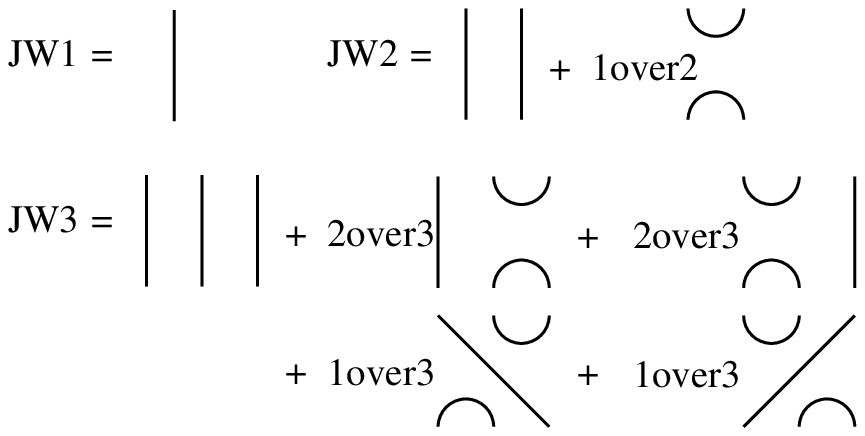}$ \end{ex}

Any crossingless matching will divide the planar strip into $2(m-1)$ regions, which can be colored (say, red and blue) so that colors always alternate across a strand. This gives rise to
the \emph{two-colored Temperley-Lieb algebra}, which is an endomorphism algebra appearing inside a 2-category with 2 objects, red and blue. In this algebra, the value of a blue circle
surrounded by red and a red circle surrounded by blue can be different scalars; so we define this algebra over the ring $\ZM[x,y]$. Jones-Wenzl projectors still exist, but each comes in
two flavors with different coefficients, depending on whether blue or red appears on the far right. The other flavor is obtained by switching the colors and switching $x$ and $y$.

\begin{ex} $\ig{1}{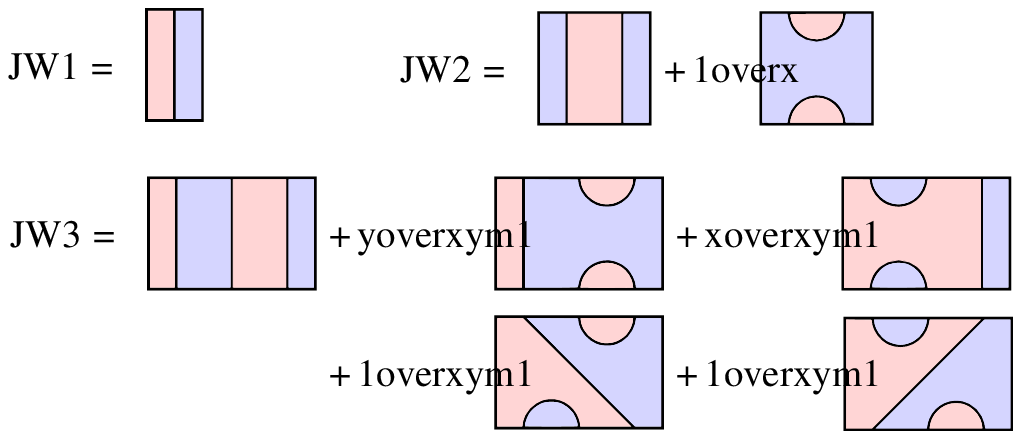}$ \end{ex}

Choose two indices $s,t \in S$. Given a 2-colored crossingless matching, we obtain a Soergel graph on the planar disc as follows: deformation retract each region into a tree composed out
of trivalent and univalent vertices, and color these trees appropriately. The result will always be a Soergel graph of degree $+2$. In order for this map to be well-defined, not just for
crossingless matchings but also for diagrams with embedded circles, one must specialize $\ZM[x,y]$ to $\Bbbk$ under the map sending $x \mapsto a_{st}$ and $y \mapsto a_{ts}$. The reader
should convince themselves that this makes sense, using \eqref{dotforcegeneral}, \eqref{alphais}, and \eqref{needle} until the scalar $\pa_s(\a_t) = a_{st}$ appears. Associated to the
Jones-Wenzl projector we have a linear combination of Soergel graphs, which we call the \emph{Jones-Wenzl morphism}. It comes in two color-flavors, as before.

\begin{ex} $\ig{1}{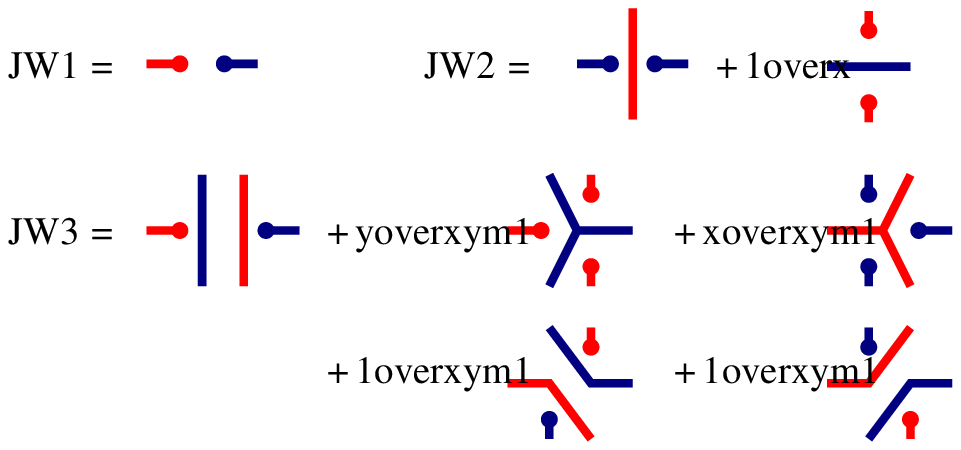}$ \end{ex}

The coefficients in these Jones-Wenzl projectors are rational functions in the 2-colored quantum numbers, as in Definition \ref{def:2qnum}. When $[m]_x=[m]_y=0$, one can show that
$JW_{m-1}$ is well-defined within $\ZM[x,y]$ (i.e. no denominators are necessary). If either $m$ is odd or if $m$ is even and $[m-1]=1$, $JW_{m-1}$ will be rotation-invariant (under
color-preserving rotations). If $m$ is even and $[m-1]=1$, it is also the case that rotating the left-blue-aligned Jones-Wenzl morphism by one strand yields the left-red-aligned
Jones-Wenzl morphism. If $m$ is odd, this holds if and only if $[m-1]_x = [m-1]_y = 1$; this is the feature behind the definition of a balanced realization. Using the examples above, the
reader should convince themselves of these rotational facts when $m=4$ (since $a_{st} a_{ts} = 2$) or when $m=3$ and $a_{st}=a_{ts}= \pm 1$. More discussion of the non-balanced case is
found in Section \ref{nonbalanced}.

The Jones-Wenzl Soergel graph above is not technically a morphism in $\DC$, because it lives on the planar disc, but it can be plugged into another diagram to produce a planar strip graph.
This is done in relation \eqref{dot2m}, where the rotation-invariant $JW_{m-1}$ is used. An implication of \eqref{dot2m} and \eqref{assoc2} is the following relation.
\begin{equation} \ig{1}{twocoloridemp} \label{twocoloridemp} \end{equation}
Note that here and henceforth, we will only use $JW_{m-1}$ when $m_{s,t}=m$, which we abbreviate as $JW$.

On the RHS of \eqref{twocoloridemp}, we took the Jones-Wenzl in the two-colored Temperley-Lieb algebra, transformed it into a degree $2$ Soergel graph on the disc, and then transformed it
again into a degree $0$ endomorphism of a color-alternating Bott-Samelson in $\DC$, by placing trivalent vertices on each side. This general procedure can be applied to any two-colored
crossingless matching, giving a map from the two-colored Temperley-Lieb algebra to a certain endomorphism ring in $\DC$. This map is a homomorphism. because $JW_{m-1}$ is idempotent in
$TL$, the RHS of \eqref{twocoloridemp} is also idempotent. The $2m$-valent vertex is thus ``half an idempotent."

The defining properties of the Jones-Wenzl projectors in $TL$ correspond to the following two properties in Soergel graphs. The first is that the coefficient of a certain graph is $1$, as
can be seen from the examples above. This is the graph which, when trivalent vertices are attached to make a degree $0$ morphism, becomes the identity map. The second property is called
\emph{death by pitchfork}:

\begin{equation} \ig{1}{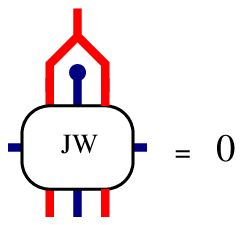} \label{JWpitch} \end{equation}

Now let us note some properties of any $(m-1,m-1)$ crossingless matching, transformed into a Soergel graph. When viewed as a degree $+2$ graph on the disc, there will be a boundary dot on
at least one of any $m-1$ adjacent strands. When viewed as a degree $+0$ endomorphism of a length $m$ color-alternating object $B_s \ot B_t \ot \ldots$ in $\DC$, it can be given a
\emph{negative-positive decomposition}. That is, we say the initial \emph{width} of the map is $m$, the number of strands on the boundary. Reading from bottom to top, we may first apply
all negative maps (i.e. bottom boundary dots, merging trivalents, and caps) which are the maps which decrease the width. Then we can apply all positive maps (i.e. cups, splitting
trivalents, top boundary dots), increasing the width back up to $m$. The morphism will factor through a length $k$ color-alternating object $B_s \ot B_t \ot \ldots$ for some $k \le m$
which is the minimal width reached. Whenever the crossingless matching is not the identity map, we have $k < m$. These properties will be used in Chapter \ref{sec-LLLproof}.

\begin{remark} \label{twocolorrelncatfd} Consider the usual Temperley-Lieb category over a $\ZM[z]$-algebra where $[m]=0$ and $[k] \ne 0$ for $0 < k < m$. Then the Temperley-Lieb algebra
with $k$ strands is isomorphic to (an integral form of) the endomorphism ring of the $U_q(\mathfrak{sl}_2)$-representation $V_1^{\ot k}$ (at generic $q$), for any $0 \le k<m$. The direct
sum decomposition \[V_1^{\ot m-1} \cong V_d^{\oplus c_{m-1,d}} \] implies the existence of certain idempotents (and isomorphisms between their images within an isotypic component) which
produce this splitting. A similar statement can be made for the two-colored Temperley-Lieb category. We can transform these maps into degree $0$ morphisms in $\DC$, to obtain a direct sum
decomposition for any color-alternating $B_s \ot B_t \ot \ldots$ of length $\le m$. This can be used to categorify relation \eqref{rank2hecke}. See \cite{EDihedral} for more details.

However, we cannot assume in general that $[k] \ne 0$ for all $k < m$. For instance, this is false when we take a realization of a finite dihedral group, and view it as a non-faithful
realization of a larger dihedral group. Nonetheless, we will soon categorify the relation of Example \ref{rank2heckeEx}. \end{remark}

\subsection{The functor to bimodules} \label{subsec-functor}

We now fix a realization where Demazure Surjectivity holds. This ensures that $\Delta_s \in R \ot_{R^s} R$ is well-defined (see section \ref{subsec-polys}).

\begin{defn} Let $\FC$ be the $\Bbbk$-linear monoidal functor from $\DC$ to $\BSBim$ defined as follows. The object ${\un{x}}$ is sent to $B_{{\un{x}}}$. The images of the dots and trivalent vertices were given in the introduction, and correspond to the four structure maps of a Frobenius extension. In the introduction, one of the dots was sent to $1 \mapsto \frac{1}{2} (\a_s \ot 1 + 1 \ot \a_s)$, but this should be $1 \mapsto \Delta_s$ in general; whenever $\frac{1}{2}$ exists, these two expressions are equal. The bimodule image of the $2m$-valent vertex is given explicitly in \cite{EDihedral}, though in a convoluted and not particularly useful form. \end{defn}

\begin{claim} This definition gives a well-defined functor. \end{claim}

\begin{proof} Previous papers have done most of the work for this claim. The polynomial relations are obvious. Any other relation involving only a subset of colors can be checked in the
category $\DC_{S'}$, where $S'$ is the appropriate subset. Most of the relations (including the isotopy relations) involve at most 2 colors, and the dihedral case was checked in
\cite{EDihedral}. It remains to check the relations arising from rank 3 parabolics. The case of $A_3$ and $A_1 \times I_2(m)$ for $m=2,3$ was done already in \cite{EKh}, where general type
$A$ was completed. The check for $A_1 \times I_2(m)$ for other $m < \infty$ exactly parallels the proof for $m=2,3$ in \cite{EKh}, and is essentially trivial.

The case of $B_3$ was checked by computer. The case of $H_3$ can be checked by computer, once the appropriate relation is found. The computer check used a localization technique, which we
will discuss in the section after next. \end{proof}

\subsection{Localization}

\begin{defn} Let $\DC_Q$ denote the localization of $\DC$ at $Q$, which is to say that we allow boxes labelled by $f \in Q$ in the leftmost region, and require that they multiply as
in $Q$. Let $\FC_Q$ denote the functor $\DC_Q \to \BSBim_Q$ which extends $\FC$ under base change. \end{defn}

Because of arguments akin to Lemma \ref{localizationononeside}, this is the same as the category which allows boxes labelled by $f \in Q$ in any region, and allows
$f \in Q^s$ to slide across a line colored $s$.

The bimodule $B_s$ is equipped with short exact sequences \eqref{ses1} and \eqref{ses2} which split after localization to $Q$. Each short exact sequence ``splits" the other, in the sense
that following one map from \eqref{ses1} and then one map from \eqref{ses2} gives an endomorphism of $Q$ or of $Q_s$ which is multiplication by $\a_s$, and is therefore invertible. Two of
the maps we have already seen: they are the bottom and top boundary dot. We seek a calculus which mixes Soergel bimodules and standard bimodules, and which contains the other two maps. We
might draw them as follows.

\begin{equation} \ig{1}{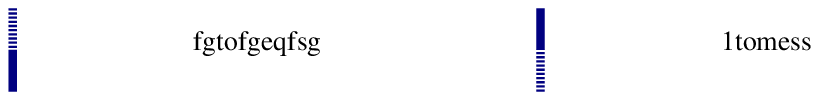} \label{bivalentgenerators} \end{equation}

\begin{defn} Let $\Kar(\DC_Q)$ (temporarily abusive notation) denote the \emph{non-cyclic(!)} biadjoint monoidal category generated on top of $\DC_Q$ as follows. Add new objects $Q_s$ for
each $s \in S$, called \emph{reflection indices}, whose identity morphisms we draw as dashed lines of the same color as $s$. Thus an object of $\Kar(\DC_Q)$ will have objects which are
sequences of normal indices and reflection indices. We typically disambiguate by writing $B_{\un{x},Q}$ for a sequence of normal indices and $Q_{\un{x}}$ for a sequence of reflection
indices. Allow an additional kind of vertex of degree $+1$, a bivalent vertex with one solid $s$ edge and one dashed $s$ edge. We also allow dashed cups and caps. Finally, impose these new
relations:

\begin{equation} \ig{1}{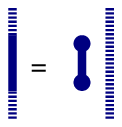} \label{newidemprel1} \end{equation}
\begin{equation} \ig{1}{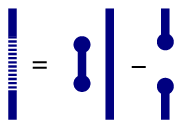} \label{newidemprel2} \end{equation}
\begin{equation} \ig{1}{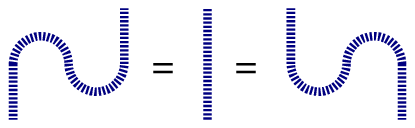} \label{biadjointdashed} \end{equation}
\begin{equation} \ig{1}{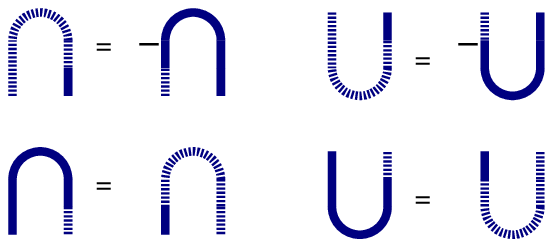} \label{rotatebivalent} \end{equation}
	
\end{defn}

The reader can check that these relations are consistent with \eqref{bivalentgenerators}. Note that we can define this category over $R$ instead of $Q$ if desired, or even in the absence
of Demazure Surjectivity. The remainder of this section will investigate $\Kar(\DC_Q)$ further, and prove the following theorem.

\begin{thm} Assume Demazure Surjectivity. This diagrammatic category is equivalent to the Karoubi envelope of $\DC_Q$. It is also equivalent to $\DStd_Q$, and to $\StdBim_Q$. \label{localdiagthm} \end{thm}

Unfortunately, $\Kar(\DC_Q)$ is not cyclic! It is worse than that the rotation of one bivalent vertex is not the other. The bivalent vertex itself is not invariant under 360 degree
twisting, being off by a sign. Any morphism with an even number of bivalent vertices will be cyclic (i.e. invariant under 360 degree twists), but this still does not imply that twisting by
180 degrees is equal to rotation by 180 degrees, because bivalent vertices do not behave this way consistently. An example can be seen below in \eqref{twist2mwbivalent}. At least different
representatives of the same isotopy class only differ by a sign.

\begin{remark} \label{geometrynocyclic} If one is familiar with the geometric underpinning of Soergel bimodules, then it should not be terribly off-putting that these maps are not cyclic.
One expects cyclicity whenever one analyzes convolution between perverse sheaves, because the procedure of taking the biadjoint of a sheaf is a functor. Soergel bimodules are the
equivariant (derived) global sections of the semisimple perverse sheaves on the flag variety, and taking global sections is a well-behaved functor on this semisimple category, so that
cyclicity happens to be preserved. However, taking global sections of non-semisimple sheaves tends to forget structure, and break the compatibility with the biadjunction functor. If one
more appropriately models non-semisimple perverse sheaves as complexes of semisimple sheaves, then these complexes will be biadjoint in a cyclic way.

For instance, the standard bimodule $R_s$ corresponds to two different, mutually biadjoint perverse sheaves: the shriek and the star extension of the constant sheaf on $\mathbb{P}^1$ minus
a point. These two different perverse sheaves have two different resolutions in terms of semisimple perverse sheaves, and these descend to \eqref{ses1} and \eqref{ses2}. The two complexes
of Soergel bimodules which are quasi-isomorphic to $R_s$ are biadjoint in a cyclic way. The bimodule $R_s$ itself is self-biadjoint. However, the compatibility between these two
biadjunction structures is broken. \end{remark}

We call an isotopy class of graphs as above \emph{mixed graphs}. A mixed graph only represents a morphism in $\Kar(\DC_Q)$ up to sign. A mixed graph has some solid edges and some dashed
edges. A mixed graph without any dashed edges is a Soergel graph or \emph{solid graph}, and a mixed graph without any solid edges is a standard graph or \emph{dashed graph} (we will expand
what counts as a dashed graph soon). (Isotopy classes of) solid graphs and dashed graphs do unambiguously represent a morphism, without any sign issues, because they have no bivalent
vertices.

Now let us do some calculations. Using \eqref{newidemprel1} and \eqref{newidemprel2}, one can easily produce the following equalities after multiplication by $\a_s$. Since we're working
over $Q$, we may divide by $\a_s$.

\begin{equation} \ig{1}{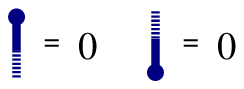} \label{dotbi} \end{equation}

\begin{remark} If defining this diagrammatic category over $R$ instead of $Q$, one should add \eqref{dotbi} to the list of relations. \end{remark}

The new bivalent vertices give rise to an idempotent $\ig{1}{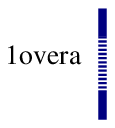}$, which is complementary to the idempotent $\ig{1}{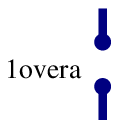}$ which already existed in $\DC_Q$. Therefore,
$B_{s,Q} \cong Q_s(1) \oplus Q(1)$. (We include the gradings for those who wish to use the homogeneous fraction field rather than the full fraction field. Remember that $Q \cong Q(2)$ so that the grading lives in $\Z/2\Z$.)  By convention we tend to include $\frac{1}{\a_s}$ in the projection map, rather than the inclusion map.

\begin{equation} \ig{1}{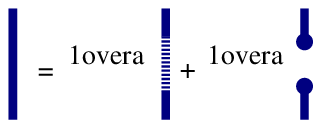} \label{localdecomp} \end{equation}
We can take a line and either ``break" it or ``dash" it. How very violent!

\begin{claim} The reflection indices behave like standard modules with respect to polynomials. In other words, we have \eqref{slidepolythrustd}. \end{claim}

\begin{proof} Place a polynomial $f$ on the left side of the diagrams in \eqref{newidemprel1}. Use the polynomial forcing rules to force $f$ through the solid line in the middle. Any term
where the line breaks is zero by \eqref{dotbi}. The remaining term has $s(f)$ on the right hand side instead. Dividing by $\a_s$ gives the desired equality. \end{proof}

\begin{claim} The dashed cups and caps are redundant, being equal to the following maps. The only relations necessary are \eqref{newidemprel1} and \eqref{newidemprel2}.
\begin{equation} \ig{1}{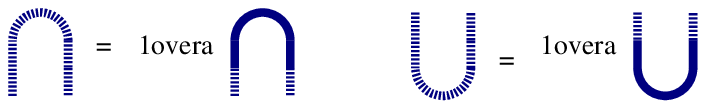} \label{dashedcupscaps} \end{equation} \end{claim}

\begin{proof} Equation \eqref{dashedcupscaps} follows from \eqref{newidemprel1} by adding a cap or cup and using \eqref{rotatebivalent}. Conversely, if we only use \eqref{dashedcupscaps}
as a definition of the dashed cups and caps, it is a simple calculation to check \eqref{rotatebivalent} and \eqref{biadjointdashed}. We give a sample computation here.
\igc{1}{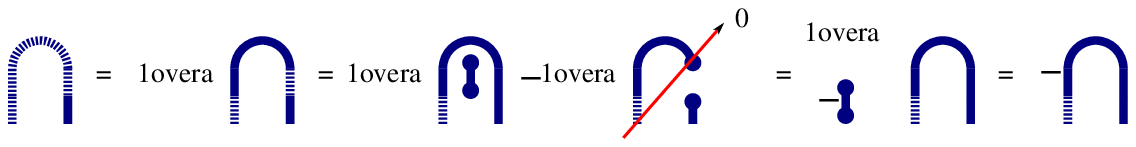} \qedhere
\end{proof}

Fix a category $\CC$ specified with generators and relations, an object $M \in \CC$, and an idempotent $e \in \End(M)$. Let $\CC(M,e)$ denote the partial idempotent completion which
formally adds the image of $e$ as a new object. Let us call this new object $X$. It is easy to give a presentation of $\CC(M,e)$ by generators and relations. One adds a new object $X$ and
two new generators, a map $\iota \colon X \to M$ and a map $\pi \colon M \to X$. One adds two new relations, which state that $\iota \pi = e \in \End(M)$ and $\pi \iota = \1 \in \End(X)$.
This presentation clearly gives a category $\CC(M,e)$ equipped with a fully faithful map $\CC \to \CC(M,e)$ which has all the desired properties. If the image of $e$ is already an object
in $\CC$ then $X$ will be isomorphic to it. Similarly, if we wish to adjoin a set of new summands $\{X_\b\}$, we need only add inclusions $\{\iota_\b\}$ and projections $\{\pi_\b\}$ with
the relations above for each $\b$. Similar statements can be made about monoidal categories with monoidal presentations, graded categories etc.

Now it is clear that $\Kar(\DC_Q)$ is the category obtained from $\DC_Q$ by adjoining the complements of $Q$ in $B_{s,Q}$, for each $s$. The bivalent vertices are the new maps $\iota$
and $\pi$, and the two relations on $\iota$ and $\pi$ correspond to \eqref{newidemprel1} and \eqref{newidemprel2}. To show that $\Kar(\DC_Q)$ is the Karoubi envelope of $\DC_Q$, it is
enough to show that $\Kar(\DC_Q)$ is idempotent complete. Because each $B_{s,Q}$ decomposes into $Q_s$ and the monoidal identity, it is clear that any object in $\DC_Q$ is isomorphic to a
direct sum of sequences consisting solely of reflection indices. Thus we need only show that any sequence consisting of reflection indices represents an indecomposable object. This will be
implied once we show that this diagrammatic category is equivalent to $\DStd_Q$.

We do some more computations in preparation for the proof of this equivalence. Suppose that one takes a $2m$-valent vertex and places $m$ consecutive bivalent vertices on it (say,
on bottom).

\igc{1}{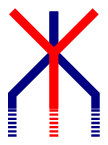}

If one places a dot on top of this diagram, we get $0$. This is because we can use \eqref{dot2m} to replace the $2m$-valent vertex with the Jones-Wenzl morphism, and the Jones-Wenzl
morphism must have a dot on one of the bottom exits. This dot will then hit the bivalent vertex, yielding zero by \eqref{dotbi}. However, each line on top decomposes into two idempotents
as in \eqref{localdecomp}, and only the dashed idempotent survives. Thus if there are ever $m$ consecutive bivalent vertices on a $2m$-valent vertex, we may as well assume that all $2m$
are present.

If we place $2m$ bivalent vertices around a $2m$-valent vertex, we get a morphism of degree $2m$. In order to get a morphism of degree $0$ we should divide by a polynomial. In other words,
half the bivalent vertices should be inclusions and half projections, and one half should be paired with $\frac{1}{\a_s}$. So consider the following diagram with $m$ strands.

\igc{1}{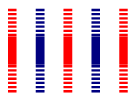}

This is a reduced expression for $w_0$, the longest element of $W_{s,t}$. Using \eqref{newidemprel1}, we get a sequence of vertical dashed lines, with $\a_s$ to the left of each line
colored $s$. If we pull all these polynomials to the far left region using \eqref{slidepolythrustd}, we get $\rho_{s,t}$, the product of all $m$ positive roots corresponding to reflections
in $\langle s, t \rangle$ (for a definition of positive roots in a dihedral group, see \cite{EDihedral}). Note that $w_0$ sends the set of positive roots to the set of negative roots, so
that $w_0(\rho)=(-1)^m \rho$.

A simple calculation using \eqref{rotatebivalent} shows that

\begin{equation} \ig{1}{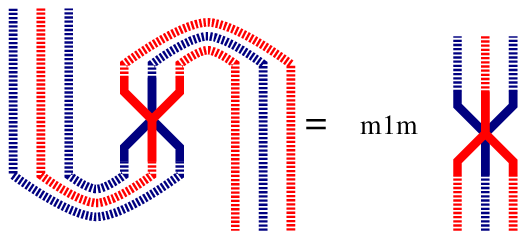} \label{twist2mwbivalent} \end{equation}

Therefore

\begin{equation} \ig{1}{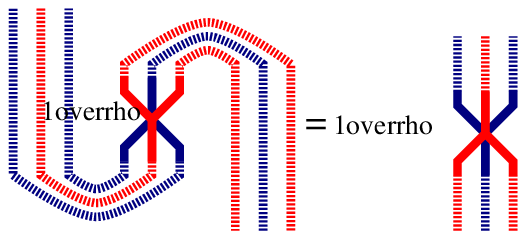} \label{2mstdiscyclic} \end{equation}

Thus the RHS of \eqref{2mstdiscyclic} is a degree $0$ map which \emph{is} cyclic, and is a perfect candidate for the $2m$-valent vertex in $\DStd_Q$.

\begin{defn} Let $\StdF$ denote the functor from $\DStd_Q$ to $\Kar(\DC_Q)$ defined herein, called the \emph{standardization functor}. On objects, it sends a sequence of indices to the
corresponding sequence of reflection indices. On morphisms, it sends caps and cups to dashed caps and cups, and it sends the $2m$-valent vertex to the morphism in \eqref{2mstdiscyclic}.
\end{defn}

We draw the image of the $2m$-valent vertex in $\DStd_Q$ as a dashed $2m$-valent vertex. From the above, it should be easy to check that

\begin{equation} \ig{1}{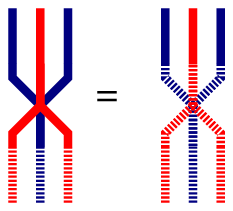} \label{slidebivalentthru2m} \end{equation}

\begin{prop} The functor $\StdF$ is well-defined, and is an equivalence of categories. \label{StdFanequiv} \end{prop}

This proposition implies all of Theorem \ref{localdiagthm} except for the connection to $\StdBim$. We have already showed that $\StdF$ is essentially surjective, because $\Kar(\DC_Q)$ is
additively generated by reflection indices. We need to check that $\StdF$ is well-defined, full and faithful.

\begin{proof}[Proof that $\StdF$ is well-defined.] We have already checked the isotopy relations in $\StdBim$, because of \eqref{biadjointdashed} and \eqref{2mstdiscyclic}. We have also checked
polynomial-sliding. Relation \eqref{capcupstd} follows as below (we used $\d = \frac{\a_s}{2}$, but any $\d$ will work).

\igc{1}{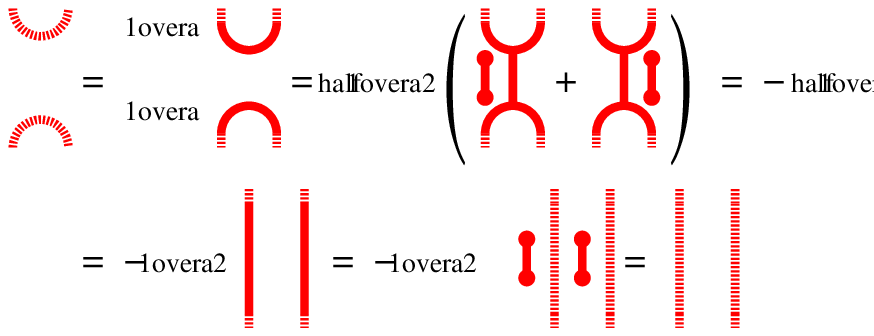}

The proof of \eqref{circlestd} is easy, and we leave it as an exercise.

The proofs of \eqref{2misomstd} and the three color relations all follow from the same method. Take relations \eqref{twocoloridemp} and the three color relations in $\DC$, and place a
bivalent vertex below every strand on bottom. Doing so will kill any diagram with a bottom boundary dot, including all the non-identity diagrams in the Jones-Wenzl projector in
\eqref{twocoloridemp}, and all the lower terms in the $H_3$ relation \eqref{H3}. We ignore all those terms. For any diagram composed entirely out of $2m$-valent vertices for various $m$,
we can use \eqref{slidebivalentthru2m} to pull the ``dashed-ness" from bottom to top, until the entire diagram is dashed except with bivalent vertices at the top. The result is precisely
the corresponding relation in $\DStd$, with bivalent vertices on top. Bivalent vertices are invertible, so this checks the relation. \end{proof}

Now we can apply any relations in $\DStd$ to dashed diagrams in $\Kar(\DC_Q)$.

We see that surrounding a $2m$-valent vertex with bivalent vertices yields (up to polynomial) the dashed version of the map. However, there can be no dashed version of the dot or the
trivalent vertex, since there are no maps between standard bimodules when they do not express the same element of $W$. Unsurprisingly, surrounding a dot or a trivalent vertex with bivalent
vertices is zero. For the dot this is \eqref{dotbi}. It is not too hard to show that

\igc{1}{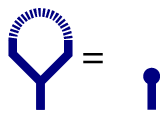}

Therefore

\begin{equation} \ig{1}{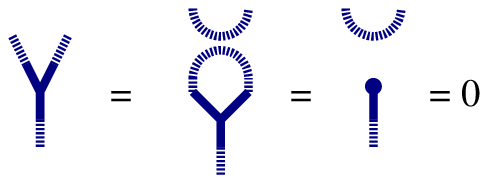} \label{dashedtriiszero} \end{equation}

\begin{lemma} \label{reduceinKarDCQ} Every morphism in $\Kar(\DC_Q)$ is a linear combination of graphs for which: \begin{itemize} \item The only dots appearing are boundary dots. \item There is a single polynomial in the leftmost region, and no other connected components without boundary. \item There are no trivalent
vertices. \item There are no solid $2m$-valent vertices, only dashed $2m$-valent vertices. \end{itemize} \end{lemma}

\begin{proof} Consider a diagram in $\Kar(\DC_Q)$. One can use \eqref{localdecomp} to either break or dash any edge. The reduction goes as follows: \begin{enumerate} \item Counting regions
separated by a dashed line as a single \emph{area}, we can assume there is a single area. This can be done by taking a solid line which separates two areas and either breaking it or
dashing it. \item We can assume that every $2m$-valent vertex appears only in its dashed version. If there are any solid $2m$-valent vertices, break or dash one of its edges. If broken,
one can use \eqref{dot2m} to remove the $2m$-valent vertex. If dashed, now break or dash the next edge. Eventually we can assume every edge is dashed, any thus can replace the $2m$-valent
vertex with its dashed version (up to multiplication by a polynomial in some region). \item We can assume there are no trivalent vertices. If there are any trivalent vertices, break or
dash one of its edges. If broken, one can use \eqref{unit} to remove the trivalent vertex. If dashed, now break or dash the next edge. One of the three edges must be broken, or the result
is zero by \eqref{dashedtriiszero}. \item We can assume there are no dots except for boundary dots. Any dot not connected to the boundary must be connected to another dot (and so becomes a
box), or to bivalent vertex (and so becomes zero). \item Any remaining solid line can only run into the boundary or into a bivalent vertex (it can not form a circle, for this would create
a second area). Any solid line meeting two bivalent vertices can be removed by \eqref{newidemprel1}, yielding a continuous dashed line (up to a polynomial). Thus any connected component
without boundary is either a box or is a purely dashed diagram. \item We can assume all boxes are in the leftmost region. This is because there is a single area, and we can slide
polynomials through dashed lines using \eqref{slidepolythrustd}. Thus the polynomials merge into a single box. \item Any remaining closed component can be removed. This uses the fact that
closed diagrams in $\DStd$ reduce to the empty diagram. \end{enumerate} \end{proof}

\begin{cor} $\StdF$ is full. \end{cor}

\begin{proof} Apply the reduction of the lemma to a map between reflection indices. There can be no boundary dots. A solid edge can not connect to anything except a
bivalent vertex any more. Thus every solid edge can be removed with \eqref{newidemprel1} (at the cost of adding a polynomial). Any map with only dashed edges is clearly in the image of
$\StdF$. \end{proof}

We can define a functor $\Kar(\DC_Q) \to \StdBim$ extending $\FC$. This functor acts on bivalent vertices as in \eqref{bivalentgenerators}. Clearly this functor intertwines $\StdF$ and
$\FC_\std$. Since $\FC_\std$ is an equivalence, this would imply that $\StdF$ is faithful, and is therefore an equivalence. This concludes the proof of Proposition \ref{StdFanequiv} and Theorem \ref{localdiagthm}

In the next section we construct a quasi-inverse for $\StdF$, giving a proof that $\StdF$ is faithful without needing the functor $\FC$ to bimodules.

\subsection{Computation using localization}
\label{subsec-compute}

Each object $B_{\un{x}}$ in $\DC$ splits in $\Kar(\DC_Q)$ as the sum of $Q_\eb$ over all subsequences $\eb$ of ${\un{x}}$. Here, $Q_\eb$ denotes the object which is the tensor product of $Q_{{\un{x}}_k}$
when $\eb_k=1$ and $\1$ when $\eb_k=0$. Given a graph expressing a morphism from a sequence $B_{\un{x}}$ of length $d$ to a sequence $B_{{\un{x}}'}$ of length $d'$, we can localize to obtain a $2^{d'} \times 2^d$ matrix of maps
between reflection sequences. Since Hom spaces between standard bimodules are always either rank 1 or 0, this matrix is actually populated with polynomials in $Q$, and is fairly sparse
(because many Hom spaces are zero a priori). Computing any term in this matrix consists of applying the appropriate projection and inclusion maps to the top and bottom of the graph, and
using the diagrammatics of $\Kar(\DC_Q)$ to reduce the graph to a dashed graph with the desired polynomial on the left. That is, given $({\un{x}},\eb)$ and $({\un{x}}',\eb')$, the coefficient of the
map $Q_\eb \to Q_{\eb'}$ of a map $\phi$ is given by reducing the following diagram $\phi_\eb^{\eb'}$.

\igc{2}{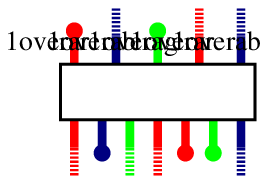}

There is no sign issue in this convention, even though a mixed graph represents a morphism only up to sign. For any Soergel graph $\phi$ we start with, we can choose a representative of
the isotopy class, and add idempotents as above to obtain a specific representative of the mixed graph. Different representatives of $\phi$ will give different mixed graphs, but they
differ only by isotopy of the solid part of the graph, and therefore have the same sign when viewed as morphisms in $\Kar(\DC_Q)$.

Note that by \eqref{slidebivalentthru2m}, the coefficient associated to $\eb=(1,1,\ldots,1)$ and $\eb' = (1,1,\ldots,1)$ of a $2m$-valent vertex is precisely $1$. Therefore, for any rex
move, this ``highest" coefficient will be $1$.

One can check if two maps in $\DC$ are equal by computing these two matrices and comparing the terms. More combinatorially, one can compute once and for all the ``basic" matrices
attached to the generating morphisms. For instance, the dot gives a $2 \times 1$ matrix, the trivalent vertex a $4 \times 2$ matrix, and the $2m$-valent vertex a $2^m \times 2^m$ matrix.
Computing a more general map consists merely of multiplying these basic matrices and annoying bookkeeping. Computers excel at such tasks.

We have not yet proven it, but the passage $\DC \to \DC_Q$ is faithful. Not knowing this, the equality of two matrices only implies that the original maps are equal modulo $R$-torsion.
However, location is injective in the bimodule world, because Hom spaces are free left $R$-modules. In other words, all $R$-torsion is in the kernel of the functor $\FC$. Therefore we can
calculate whether two Soergel graphs have the same image under $\FC$ by localizing $\DC \to \Kar(\DC_Q)$ and computing the matrices above. This is a powerful tool.

For instance, we want to know whether both sides of the $B_3$ relation \eqref{B3} correspond to the same map between Bott-Samelson bimodules. One need only compute by hand the matrices
associated to the 4-valent, 6-valent, and 8-valent vertices, and then plug in two appropriate formulae into a computer. Checking that two sparse matrices of size $2^8 \times 2^8$ are equal
is trivial. If we knew what the $H_3$ relation \eqref{H3} should be, then checking it would require calculating the 10-valent vertex (which is easy), and then computation of a pair of
$2^{15} \times 2^{15}$ matrices (which is quite time-consuming). Unfortunately, backsolving for the coefficients in the H3 relation would require doing linear algebra with a large number
$N$ of unknowns (at least $70$). There is an equation for each nonzero term in a sparse $2^{15} \times 2^{15}$ matrix, and the coefficients come from $N$ different such matrices which need
to be computed. Neither the author's computers nor their brains appear to be up to the task.

Finally, we provide the alternative proof of faithfulness.

\begin{prop} $\StdF$ is faithful. \label{stdfisfaithful2}\end{prop}

\begin{proof} In fact, the techniques we have developed allow us to construct a quasi-inverse for $\StdF$. We construct a functor $\GC$ from $\Kar(\DC_Q)$ to $\DStd_Q$ as follows. Let
$\GC$ send reflection sequences to themselves, and normal sequences to the corresponding formal direct sum of reflection sequences. On morphisms, $\GC$ will send standard graphs to
themselves. The generating morphisms which are not standard graphs are sent as follows: a bivalent vertex is sent to the appropriate inclusion or projection; a dot or trivalent vertex or
$2m$-valent vertex is sent to the appropriate matrix of standard diagrams. It is trivial to check that all the relations hold, so this functor is well-defined, and is obviously a
quasi-inverse to $\StdF$. \end{proof}

Remember that diagrams and their linear combinations are only useful for displaying morphisms between tensor products of generators, not for \emph{direct sums} of those. To talk about a
morphism between direct sums, we need to use matrices of diagrams. Thus we do not expect there to be an actual diagram in $\DStd_Q$ corresponding to a solid graph.

This proof that $\Kar(\DC_Q) \cong \DStd_Q$ was entirely diagrammatic, and no mention of bimodules was required.

\subsection{Unbalanced realizations}
\label{nonbalanced}

Now we discuss the diagrammatic alterations which must be made to accommodate the case of non-balanced realizations. This was discussed for dihedral groups in the appendix to
\cite{EDihedral}, where two separate diagrammatic conventions are proposed to deal with the new bookkeeping required. We follow the second suggested convention from that appendix.

Fix a dihedral parabolic subgroup with $m=m_{st} < \infty$. Relation \eqref{dot2m} implies that there is a close connection between $2m$-valent vertices and two-colored Jones-Wenzl
projectors. However, this begs the question: which Jones-Wenzl projector? After all, the left-blue-aligned Jones-Wenzl projector is not equal to the rotation of the left-red-aligned one. A
careful examination of \eqref{dot2m} and \eqref{twocoloridemp} shows that different choices of Jones-Wenzl projector must be made, depending on the location of the dot and the orientation
of the $2m$-valent vertex.

If some rescaling of \eqref{dot2m} is to hold, for any positioning of the dot and the $2m$-valent vertex, then the Jones-Wenzl projector must satisfy the death by pitchfork property
\eqref{JWpitch}. In particular, $JW_{m-1}$ must be negligible, and must have some rotational eigenvalue. This is not possible when the realization is even-unbalanced (see \cite{EDihedral}
for more details). While one may be able to design a diagrammatic calculus for even-unbalanced realizations, we will not attempt to do so. When $m$ is even and $[m-1]=1$, all the two-color
relations above hold as stated.

The $2m$-valent vertex is supposed to correspond, under the functor $\FC$, to some non-zero morphism between bimodules, living in a one-dimensional space of morphisms. Let $\un{w}_{(s)}$
denote the reduced expression $\ldots tsts$ of length $m$ ending in $s$, and let $\un{w}_{(t)}$ denote the reduced expression $\ldots stst$ of length $m$ ending in $t$. There is a unique
bimodule map $f_{(s)} \co BS(\un{w}_{(t)}) \to BS(\un{w}_{(s)})$ which sends the 1-tensor to the 1-tensor, and a unique map $f_{(t)} \co BS(\un{w}_{(s)}) \to BS(\un{w}_{(t)})$ which does
the same. However, when the realization is not balanced, these maps are not rotations of each other (by one strand), and this is the underlying issue. However, these maps are individually
invariant under color-preserving rotations (just like $JW_{m-1}$, even in the odd-unbalanced case), so we may draw each unambiguously as some kind of $2m$-valent vertex. We label the
vertices $(s)$ or $(t)$ to distinguish the two. (We only draw the case when $m$ is odd, but nothing prevents drawing the even case too.)

\igc{1}{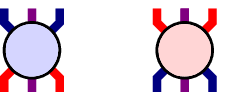}

The purple strand is meant to encode the appropriate sequence of alternating red and blue strands. To reiterate, when the blue-centered $2m$-valent vertex is oriented such that its
upper-right strand is blue (as in the picture above), then it corresponds under $\FC$ to a morphism which preserves the 1-tensor. When the blue-centered $2m$-valent vertex is oriented such
that its upper-right strand is red, one differs from this map by an invertible scalar $\l$. In \cite{EDihedral}, this scalar $\l = [m-1]_y$ is discussed at additional length. Therefore,
one has the rotational relation: \begin{equation} {
\labellist
\small\hair 2pt
 \pinlabel {$\l$} [ ] at 65 24
 \pinlabel {$\l^{-1}$} [ ] at 194 23
\endlabellist
\centering
\ig{1}{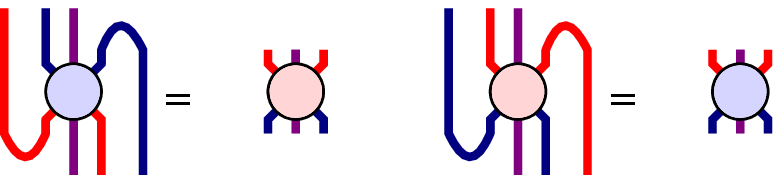}
} \end{equation}

When we place a dot on a $2m$-valent vertex, we obtain a relation similar to \eqref{dot2m}. \begin{equation} {
\labellist
\small\hair 2pt
 \pinlabel {$JW_{m-1}$} [ ] at 92 93
 \pinlabel {$JW_{m-1}$} [ ] at 236 93
 \pinlabel {$JW_{m-1}$} [ ] at 92 20
 \pinlabel {$JW_{m-1}$} [ ] at 236 20
\endlabellist
\centering
\ig{1}{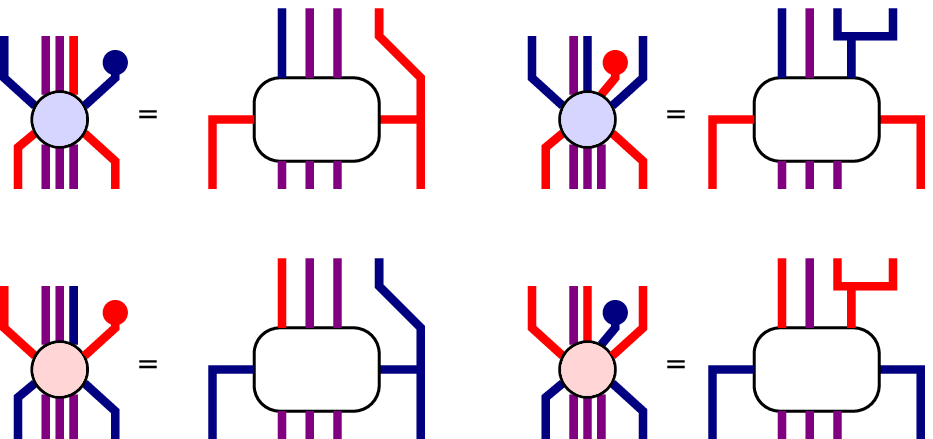}
} \end{equation} In each case, the version of the
Jones-Wenzl projector used is the one whose coefficient of the identity is $1$, when color-aligned as above. By rotating these pictures, one knows how to place a dot on any strand in
either $2m$-valent vertex. To check that these relations make sense, observe that both sides send the 1-tensor to the 1-tensor.

To give the correct versions of \eqref{assoc2} and \eqref{twocoloridemp}, one should label the $2m$-valent vertices in such a way that both sides preserve the 1-tensor. \begin{equation}
\ig{1}{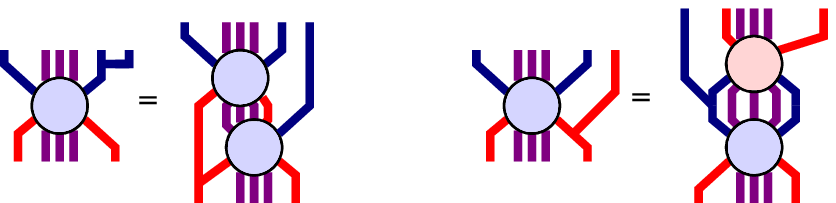} \end{equation} \begin{equation} \ig{1}{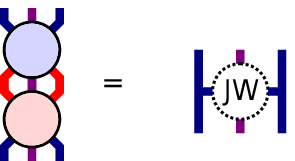} \end{equation} There are additional, color-switched versions of each of these relations.

Finally, to give the correct version of the three-color relations, one should again ensure that both sides preserve the 1-tensor. Here is the $A_3$ relation. \begin{equation}
\ig{1}{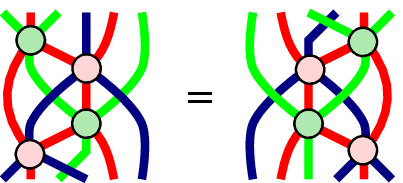} \end{equation}

Note that, when working with these diagrams, there is a scalar ambiguity that appears when defining a rex move, determined by the choice of central color on each $2m$-valent vertex. Our
convention is that one will always choose a coloring so that the 1-tensor is preserved by the rex move.

For the remainder of this paper, we will work solely with the balanced case, and thus will not need the extra confusion of labeled $2m$-valent vertices. In the rest of the paper, it is
roughly the case that diagrams are only important ``up to scalar." That is, we will be asking whether certain morphisms span or are linearly independent, and these concepts are not
affected by rescaling. It should not be hard to convince oneself that the remainder of this paper applies equally well to the odd-unbalanced case.

\part{Light leaves morphisms and proofs}

\label{pt:proofs}

\section{Libedinsky's Light Leaves} \label{sec-LLL}

In this section, we investigate Libedinsky's ``light leaves'' maps \cite{Lib} from a diagrammatic perspective, and prove (modulo the arguments in the next chapter) that ``double leaves"
form a cellular basis for $\DC$.


\subsection{Diagrammatics for light leaves}
\label{subsec-LibedDiag}

Fix a rex ${\un{w}}$. We want to find a basis of diagrams for the space of morphisms $B_{\un{x}} \to B_{\un{w}}$ modulo ``lower terms." Here, lower terms are morphisms which induce the
zero map to the unique standard summand $Q_w \sumset B_{\un{w}}$ after localization. This basis will be parametrized by subsequences $\eb$ of $\un{x}$ which express $w$. Libedinsky
associates a morphism $\LL_{{\un{w}},\eb}$ to $\eb$, although this choice is not canonical. In fact, there are many valid choices for what each $\LL_{{\un{w}},\eb}$ can be, and selecting
one morphism for each $\eb$ will give a basis modulo lower terms. Thus we will give a general rubric for selecting $\LL_{{\un{w}},\eb}$, which does not specify a single map but a set of
maps, any of which will suffice. See Remark \ref{LLprecision} below on how to be more specific.

Recall that a rex move is a morphism constructed from $2m$-valent vertices which corresponds to a path in the reduced expression graph of some element $w \in W$. Rex moves have degree $0$.
When the realization is unbalanced, there is an ambiguity when defining the rex move associated to a path, which amounts to an invertible scalar. This scalar will be irrelevant for our
discussion below: a rescaled basis is a basis still. However, we use the convention that rex moves always preserve the 1-tensor.

\begin{constr} For every $({\un{x}},\eb)$ expressing $w$ and every $k \le \ell({\un{x}})$, we let $({\un{x}}_{\le k},\eb_{\le k})$ be the first $k$ terms, expressing an element $w_k$, and let ${\un{x}}_{> k}$
denote the remaining terms. When ${\un{x}}$ is the empty set and $\eb$ its unique subsequence, the map $\LL_{{\un{w}},\eb}$ is the empty diagram. Suppose that, by induction, we have already chosen  a map $\LL_{k-1} \define \LL_{{\un{x}}_{\le k-1},\eb_{\le k-1}} \co B_{{\un{x}}_{\le k-1}} \to B_{{\un{w}}_{k-1}}$ for some rex ${\un{w}}_{k-1}$ of $w_{k-1}$. Suppose that the next index ${\un{x}}_k$ is $s$. By
placing a vertical line $\1_{s}$ next to $\LL_{{\un{x}}_{\le k-1},\eb_{\le k-1}}$ we get a map from $B_{{\un{x}}_{\le k}} \to B_{{\un{w}}_{k-1}} \ot B_s$. We will now choose a map $\phi_k \co
B_{{\un{w}}_{k-1}} \ot B_s \to B_{{\un{w}}_k}$ for some rex ${\un{w}}_k$ of $w_k$. The composition will be $\LL_k \define \LL_{{\un{x}}_{\le k},\eb_{\le k}}$.

It follows that $\LL_{{\un{w}},\eb}$ is a composition $\phi_{\ell({\un{w}})} \circ \phi_{\ell({\un{w}})-1} \ot \1 \circ \cdots \circ \phi_1 \ot \1$, where at the $k$-th step $\1$ denotes the
identity map of $B_{{\un{w}}_{> k}}$. The composition of the first $k$ terms is $\LL_k \ot \1_{B_{{\un{w}}_{>k}}}$.

\begin{equation} \label{LLinduction} \ig{1}{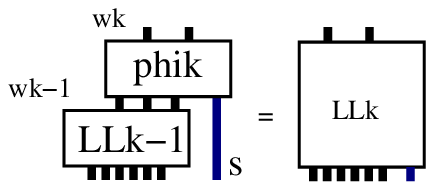} \end{equation}

There are four possibilities for the map $\phi_k$, depending on $\eb_k$. To obtain $\phi_k$ follow these three steps (see also figure \ref{fourmaps}):

\begin{itemize} \item If $\eb_k$ is either U1 or U0, do nothing. If $\eb_k$ is either D0 or D1, then $s$ is in the right descent set of $w_{k-1}$. Apply $\b \ot \1_{B_s}$, where $\b$ is
some rex move from ${\un{w}}_{k-1}$ to a rex with $s$ on the right. Now the top has $B_s \ot B_s$ on the far right.

\item If $\eb_k$ is U1 do nothing. If $\eb_k$ is U0, apply a dot to the rightmost $B_s$. If $\eb_k$ is D1, apply a cap to the final $B_s \ot B_s$. If $\eb_k$ is D0, apply a merging
trivalent vertex to the final $B_s \ot B_s$.

\item We have now reached some reduced expression for $w_k$. Now apply some rex move $\a$ to get to the desired rex ${\un{w}}_k$. \end{itemize} \end{constr}

\begin{figure}  \caption{The Four Maps} $\ig{1}{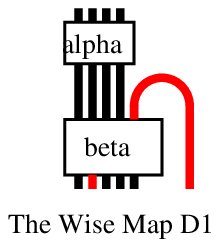} \ig{1}{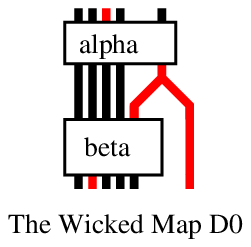} \ig{1}{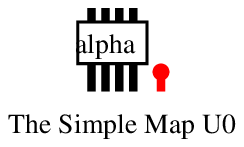} \ig{1}{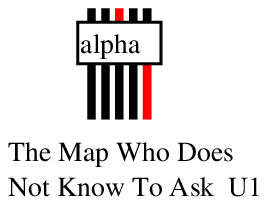}$ \label{fourmaps} \end{figure}

As expected, the degree of the morphism $\LL_{{\un{w}},\eb}$ is $+1$ for each U0 and $-1$ for each D0 and hence agrees with defect of $\eb$.
Note also that the width (i.e. number of strands) always weakly decreases from bottom to top in an $\LL$ map.

\begin{ex} Here is a possible map $\LL_{\un{w},\eb}$ for $\un{w}=rbrgbrr$ with $m_{br}=3$ and $m_{bg}=m_{gr}=2$, and for $\eb = 1111010$. \igc{1}{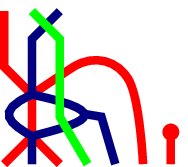} \end{ex}

\begin{remark} \label{LLprecision} Clearly there are many choices in this construction. The first important choice is which rex ${\un{w}}$ is the final target. If one is to compare
$\LL_{{\un{x}},\eb}$ with $\LL_{{\un{x}},\fb}$ for two subsequences $\eb,\fb$ both expressing $w$, then they should have the same target ${\un{w}}$, and this is generally assumed. However, at each step
in the inductive construction one needs to choose a rex ${\un{w}}_k$, and there is no particular need to be consistent with this choice. For instance, the intermediate map $LL_k$ in the
construction of $\LL_{{\un{x}},\eb}$ need not equal the map we constructed for $\LL_{{\un{x}}_{\le k},\eb_{\le k}}$. There is even no need for the intermediate rex ${\un{w}}_k$ to agree with the final
target for the chosen map $\LL_{{\un{x}}_{\le k},\eb_{\le k}}$. When $\eb_k$ is D0 or D1, one has a free hand to choose which rex with $s$ on the right will be factored through, and which rex
move $\b$ will take us there. The rex move $\a$ is also freely chosen, and by no means does it have to be the shortest or easiest way to traverse the reduced expression graph $\G_w$.

To be absolutely precise, i.e. to make the above construction into an algorithm, one should fix once and for all the following data: \begin{enumerate} \item For each $w \in W$, a rex
${\un{w}}$. \item For each $w \in W$ and each index $s \in \SC$ in the \emph{right descent set} of $w$, a rex ${\un{w}}_s$ which ends in $s$. \item For any two rexes ${\un{w}}_1$ and ${\un{w}}_2$ in for $w$, a rex move from ${\un{w}}_1$ to ${\un{w}}_2$. \end{enumerate} This is not the only way to make the algorithm precise.

The more precise one is, the more annoying certain statements get. Flexibility will be more useful. However, at some point (in the final chapter) we will have to show how one set of
choices ``spans" all the other possible choices. \end{remark}

\begin{remark} \label{LLof111} When $\eb=(1,1,\ldots,1)$, it is a sequence of all U1 precisely when ${\un{w}}$ is a reduced expression. If so, the set of possible $\LL_{{\un{w}},\eb}$ is precisely
the set of rex moves. The most convenient choice is for $\LL_{{\un{w}},\eb}$ to be the identity map.

Now suppose that $({\un{w}},\eb)$ has length $k$, and the first $k-1$ terms are all U1, a rex for $w_{k-1}$. In the construction of $\LL_{{\un{w}},\eb}$ we may choose $\LL_{k-1}$ to be the identity
map, meaning that $\LL_{{\un{w}},\eb} = \phi_k$. Thus for any sequence $({\un{w}},\eb)$ the map $\phi_k$ is a valid choice for $\LL_{{\un{w}}',\eb'}$, where ${\un{w}}' = {\un{w}}_{k-1} s$ and $\eb' =
(1,1,1,\ldots,1,\eb_k)$. This observation gives us a different inductive way to look at light leaves: \igc{1}{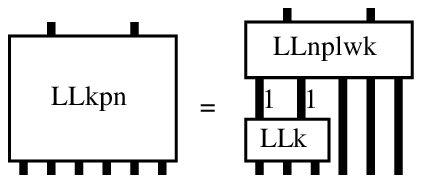} One should think of the top of an $\LL$ map as being a rex
labelled with all U1's, for the purpose of further $\LL$ maps. \end{remark}

We will write $\LL_{{\un{x}},w}$ to indicate a set consisting of one fixed choice of $\LL_{{\un{x}},\eb}$ for each subexpression $\eb$ of $\un{x}$ expressing $w$. When we speak of the ``span" of  $\LL_{{\un{x}},w}$, we mean all morphisms obtained as linear combinations of $\LL$ maps, with polynomials appearing in the left-most region. Note that $\LL$ maps themselves never have polynomials in any region.

\begin{remark} Not every diagram is in the light leaves basis. Here is a diagram which is not in the light leaves basis, and a description of it as an $R$-linear combination of $\LL$ maps.
\igc{1}{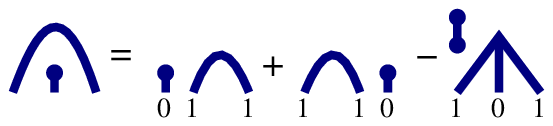}

Here is a diagram which is not in the span of $\LL$ maps at all, because it factors through ``lower terms."
\igc{1}{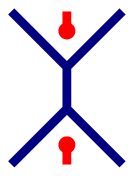}
However, when viewed under adjunction as a map to $\emptyset$, it is an $\LL$ map.
\igc{1}{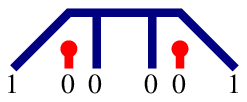}
Note that $\emptyset$ has no lower terms, so that $\LL_{{\un{x}},e}$ should genuinely form a basis for
$\Hom_{\DC}(B_{\un{x}},\1)$.
\end{remark}

\subsection{Localizing light leaves}

Let us fix a light leaves map $\LL_{{\un{x}},\eb}$, with $\eb$ expressing $w$. Let ${\un{w}}$ be the target of the map. Now consider what happens after the passage $\DC \to \Kar(\DC_Q)$. The target $B_{\un{w}}$ splits up into standard summands, with a unique summand isomorphic to $Q_w$. For each subsequence $\fb$ of ${\un{x}}$ which also expresses $w$ we have a summand $Q_\fb \sumset B_{\un{w}}$ isomorphic to $Q_w$, and this is sent by $\LL_{{\un{w}},\eb}$ into the unique summand $Q_w \sumset B_{\un{w}}$ with some coefficient $p^\eb_\fb \in Q$. For the conventions used to calculate this coefficient, see section \ref{subsec-compute}. A priori this coefficient depends on the choices made in the construction of $\LL_{{\un{w}},\eb}$.

\begin{prop}[Path Dominance Upper-triangularity] If $p^\eb_\fb \ne 0$ then $\fb \le \eb$ in the path dominance order. Moreover, $p^\eb_\eb$ is non-zero, and is a product of roots
independent of the choice of $\LL_{{\un{x}},\eb}$. \end{prop}

\begin{proof} Let us use the same notation as the previous section, so that $w_k$ is the element expressed by $({\un{w}}_{\le k},\eb_{\le k})$. We write $v_k$ for the corresponding element with $\eb$ replaced by $\fb$.

Remember that $\LL_{{\un{w}},\eb}$ is defined inductively, beginning with $\LL_k = \LL_{{\un{w}}_{\le k},\eb_{\le k}} \ot \1$. The target of $\LL_k$ is ${\un{w}}_k$, a rex for $w_k$, which only has
standard summands corresponding to elements $u \le w_k$. Therefore $\LL_k$ will clearly act as zero on $Q_{\fb_{\le k}}$ unless $v_k \le w_k$. Thus $\LL_{{\un{w}},\eb}$ will act as zero on
$Q_\fb$ unless $v_k \le w_k$ for every $k$, which is exactly the condition for $\fb \le \eb$ in the path dominance order. The upper-triangularity now follows.

We now turn to an explicit formula for $p^\eb_\eb$. For each $k \le \ell({\un{x}})$ define a root $\a_k$ as follows:
\begin{gather*}
  \a_k \define \begin{cases} w_{k-1}(\a_{{\un{x}}_k}) & \text{if $\eb_k$ is U0,} \\
-w_{k-1}(\a_{{\un{x}}_k}) & \text{if $\eb_k$ is D1,} \\
1 & \text{otherwise.}
\end{cases}
\end{gather*}
Note that $w_{k-1}(\a_{{\un{x}}_k})$ is the coefficient one obtains if one places $\a_{{\un{x}}_k}$ in the region just before ${\un{w}}_k$, as an element of $Q_\eb$,
and pulls it to the far left region. We claim that $p^\eb_\eb = \prod_{k=1}^{l({\un{x}})} \a_k$.

Consider what happens at the $k$-th step, when one includes from either $\1$ or $Q_s$ into $B_s$, and then follows $\phi_k$, the inductive part of an $\LL$ map (see figure
\ref{calculatingp}). First let us take care of all the cases when $\eb_k$ is 0. If $\eb_k$ is U0 then $\phi_k$ is a dot and the inclusion from $R$ is a dot, leaving us with a factor of $\a_s$ in that spot,
which we drag left to obtain $\a_k$. If $\eb_k$ is D0 then $\phi_k$ is a trivalent vertex; the dot from the inclusion pulls into the trivalent vertex, leaving nothing behind. What remains
is an $\LL$ map with all 0's removed (with a coefficient), so it is enough to check the formula when $\eb$ consists only of 1's.

Remember that bivalent vertices ``pull through" rex moves, as in relation \eqref{slidebivalentthru2m}. All that remains of the $\LL$ map is rex moves and caps. If $\eb_k$ is U1 then one
has a bivalent vertex as the inclusion map, which will eventually pull through braids and run into either a D1 or the top of the diagram. If it runs into the top, it will precisely cancel
the projection map on top, leaving no coefficient. If $\eb_k$ is D1 then the bivalent vertex meets a bivalent vertex from an earlier U1, and using relation \eqref{dashedcupscaps} we obtain
a factor of $-\a_s$, which pulls left to become $\a_k$, as in figure \ref{calculatingp}.

\begin{figure} \label{calculatingp} \caption{Calculating $p^\eb_\eb$} $\ig{1}{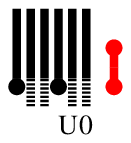} \qquad \ig{1}{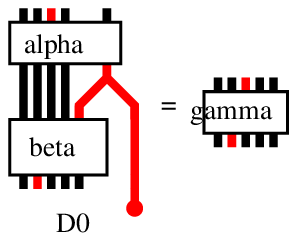} \qquad \ig{1}{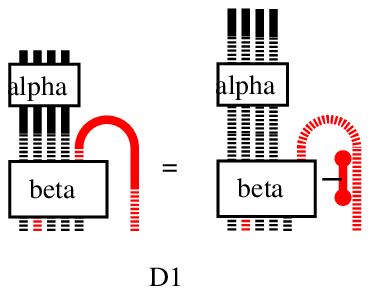}$ \end{figure}

What we obtain is precisely the standard diagram representing the only map from $Q_{\eb}$ to $Q_{{\un{w}}}$, with a polynomial on the left equal to $\prod_{k=1}^{l({\un{w}})} \a_k$. \end{proof}

\begin{remark} Do not believe that just because $p^\eb_\eb$ did not depend on the choice of rex moves in the construction of $\LL_{{\un{w}},\eb}$, that no coefficients do. When $\fb < \eb$,
$p^\eb_\fb$ does depend on the rex moves chosen. Here is an example, where ${\un{w}}=(s,t,s,s,t,s)$ in type $A_2$, $\eb=(1,1,1,1,1,1)$ and $\fb=(0,0,0,0,0,0)$. \igc{1}{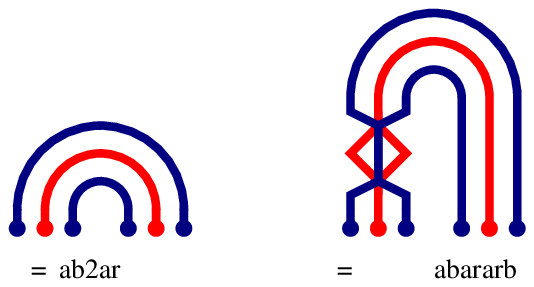}
However, it is a priori clear that all coefficients $p_{\fb}^\eb$ are either products of $W$-conjugates of simple roots, or are zero. 
\end{remark}

\begin{cor} Fix an expression $\un{x}$ and let $\LL_{{\un{x}},w}$ be a set consisting of one light leaves map $\LL_{{\un{x}},\eb} \co B_{\un{x}} \to B_{\un{w}}$ for each subexpression $\eb$ expressing $w$. Now consider the maps $B_{\un{x},Q} \to Q_{w}$ in $\Kar(\DC_Q)$ obtained by postcomposing $\LL_{{\un{x}},\eb}$ with the projection $B_{\un{w},Q} \to Q_{w}$. These maps form a basis for $\Hom(B_{\un{w},Q}, Q_{w})$. Moreover, the original set $\LL_{{\un{x}},w}$ is linearly independent over $R$ as a subset of $\Hom(B_{\un{x}}, B_{\un{w}})$. \label{LLalocalbasis}\end{cor}

\begin{proof} The morphism space $B_{\un{w},Q} \to Q_{\un{w}}$ is the direct sum, for each $\fb$ expressing $w$, of the morphism space $Q_\fb \to Q_{\un{w}}$. The basis result now follows
from upper-triangularity, and the invertibility of the diagonal in $Q$. Linear independence follows immediately. \end{proof}

This corollary is essentially Libedinsky's theorem \cite{Lib}, and we have now presented a diagrammatic proof for it, entirely within the context of $\DC$ and $\DStd$ (i.e. without ever
using bimodules). As soon as one knows that the dimensions of $\Hom$ spaces is given by the standard pairing (as is the case for Soergel bimodules by \ref{SoergelThm2}) it follows by counting dimensions that the set $LL_{\un{w}, e}$ gives a  basis of $\Hom(B_{\un{w}}, R)$.

\subsection{Double leaves}

The previous section was essentially a discussion of maps $B_{\un{w}} \to Q_w$ for some $w$. Let us use this to discuss maps $B_{\un{w}} \to B_{\un{y}}$ for arbitrary expressions ${\un{w}}$ and ${\un{y}}$.

Consider a light leaves map $\LL_{{\un{x}},\eb} \co B_{\un{x}} \to B_{\un{w}}$ where ${\un{w}}$ is a rex for $w$. Flipping this diagram upside-down, we get a map $\oLL_{{\un{x}},\eb} \co
B_{\un{w}} \to B_{\un{x}}$. By the results of the previous section, $\oLL_{{\un{w}},\eb}$ behaves nicely after localization, giving a nonzero map $Q_{{\un{w}}} \to Q_\fb$ only when $\fb
\le \eb$. The coefficients appearing are not actually $p^\eb_\fb$, because the polynomials $\frac{1}{\a_s}$ should appear only in the projections maps from $B_{s,Q}$ to $Q$ or $Q_s$, not
in the inclusion maps. In the new formula for $p^\eb_\eb$, U0 and U1 will not contribute, and D0 and D1 will contribute $\frac{1}{\a_k}$. We leave the reader to check the details.
Regardless, the result is still invertible in $Q$ and the analogue of Corollary \ref{LLalocalbasis} holds.

Now let ${\un{x}}$ and ${\un{y}}$ be arbitrary sequences with subsequences $\eb$ and $\fb$ respectively, such that $({\un{x}},\eb)$ and $({\un{y}},\fb)$ both express $w$. Choose a rex ${\un{w}}$ for $w$, and
construct maps $\LL_{{\un{w}},\eb} \co B_{\un{w}} \to B_{\un{w}}$ and $\oLL_{{\un{y}},\fb} \co B_{\un{w}} \to B_{\un{y}}$. We define the corresponding \emph{double leaves map} to be the composition 
\[
\LLL_{w,\fb,\eb} \define \oLL_{{\un{y}},\fb} \circ \LL_{{\un{x}},\eb}.
\] 
After localization, we obtain a coefficient $p^{\fb,\eb}_{\fb',\eb'}$ given by the inclusion of each standard summand $Q_{\eb'}$ of $B_{\un{x}}$ and projection to each standard summand $Q_{\fb'}$ of ${\un{y}}$. We know several facts about these coefficients:

\begin{itemize} \item $p^{\fb,\eb}_{\fb',\eb'}=0$ unless $({\un{x}},\eb')$ and $({\un{y}},\fb')$ express the same element $v$.
\item $p^{\fb,\eb}_{\fb',\eb'}=0$ unless both $\eb'
\le \eb$ and $\fb' \le \fb$. In particular, this implies that the commonly expressed element $v$ must satisfy $v \le w$. We refer to this latter phenomenon as \emph{Bruhat
upper-triangularity}, a special kind of path dominance upper-triangularity.
\item When $v=w$, $\eb'=\eb$ and $\fb'=\fb$, the coefficient is nonzero and is a product of roots, obeying a
simple formula independent of the choice of $\LL$ maps. \end{itemize}

\begin{prop} Let $\LLL_{{\un{w}},{\un{y}}}$ contain one map $\LLL_{w,\fb,\eb}$ for each $w \in W$ and each pair of subsequences $({\un{x}},\eb)$ and $({\un{y}},\fb)$ expressing $w$. Then after localization, $\LLL_{{\un{w}},{\un{y}}}$ forms a basis of maps $B_{\un{w}} \to B_{\un{y}}$. In particular, the set $\LLL_{{\un{w}},{\un{y}}}$ is linearly independent. \label{LLLalocalbasis} \end{prop}

\begin{proof} After localization, the space of maps $B_{\un{w}} \to B_{\un{y}}$ is a direct sum of $\Hom(Q_\eb,Q_\fb)$ for each pair of subsequences. These terms have a partial order, and the
$\LLL$ maps satisfy upper-triangularity with respect to this partial order, with an invertible diagonal. \end{proof}

\begin{remark} In the definition of the double leaves basis we could have taken $\LL$ maps and rotated them 180 degrees, instead of flipping them. This would avoid some of the annoyances of the final
chapter, but would introduce its own annoyances. Most notably, rotation takes a map $B_{\un{w}} \to B_{\un{y}}$ and returns a map $B_{\o({\un{y}})} \to B_{\o({\un{w}})}$, where $\o$ denotes reversing the order of a
sequence. To define $\oLL$ using a rotated map, the target of the original map must be $B_{\o({\un{w}})}$, which is actually a rex for $w^{-1}$, and this requires additional bookkeeping.
Rotation will be more obviously useful in the final chapter. Note that both approaches are equally valid, although coming up with a change of basis matrix between them would be a combinatorial nightmare. Also, flipping a diagram vertically is more natural in terms of the cellular structure.\end{remark}

Let us now state some of our main results, which all assume Demazure Surjectivity:

\begin{thm} \label{homspacetheorem} The set $\LLL_{{\un{x}},{\un{y}}}$ forms a free $R$-basis for $\Hom(B_{\un{x}},B_{\un{y}})$ in $\DC$. \end{thm}

\begin{prop} \label{homspaceprop} The set $\LL_{{\un{x}},e}$ forms a free $R$-basis for $\Hom(B_{\un{x}},\1)$ in $\DC$. \end{prop}

\begin{cor} Hom spaces in $\DC$ are free graded $R$-modules. \end{cor}

\begin{remark} Proposition \ref{homspaceprop} is a special case of Theorem \ref{homspacetheorem}, because when ${\un{y}}=\emptyset$ we must have $w=e$, and $\oLL_{{\un{y}}}$ is the empty
diagram.\end{remark}

The final chapter of this paper contains an elementary diagrammatic (though unpleasant) proof of these results.  Of course it is enough to show that $\LLL$ spans, as we have already shown that this set is linearly independent over $R$. The reader is now equipped to read that chapter, which does not use anything beyond this
section.

We assume these three results for the rest of this chapter.

Because Theorem \ref{homspacetheorem} did not depend on the particular $\LL$ maps chosen, this implies that two different choices of $\LLL$ are pairwise dependent. It does \emph{not} imply
that two different choices of $\LL_{{\un{x}},w}$ are pairwise dependent, because the dependence relations may require lower terms.

\begin{ex} Let $s,t$ be the indices in type $A_2$. There are two choices for the map $\LL_{sts,111}$, as below. The difference between them is not an $\LL$ map, but it is an $\LLL$ map.
\igc{1}{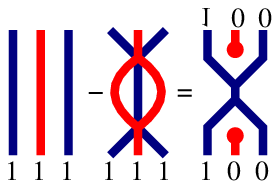} \end{ex}

\subsection{Cellularity}

We assume Theorem \ref{homspacetheorem}, so that Hom spaces are free left $R$-modules and localization is injective on Hom spaces.

\begin{claim} For any $w \in W$, let $\XC_w$ be the set of maps in $\DC$ which, after localization, induce the zero map on every $Q_w$ summand of their source object. Then $\XC_w$ is a
2-sided ideal in $\DC$. \end{claim}

\begin{proof} Left to the reader. \end{proof}

\begin{claim} The light leaves map $\LLL_{w,\eb,\fb}$ will be in $\XC_v$ for any $v \nleq w$. \end{claim}

\begin{proof} This follows from Bruhat upper-triangularity. \end{proof}

For any $I \subset W$, let $\LLL_I$ denote the span of all $\LLL$ maps which factor through $w \in I$. Now let $I \subset W$ be an ideal with respect to the Bruhat order. In other words,
if $w \in I$ and $v \le w$ then $v \in I$.

\begin{claim} When $I$ is an ideal, $\LLL_I$ is a 2-sided ideal in $\DC$. It is equal to $\XC_{W \setminus I}$, the intersection of the ideals $\XC_w$ for each $w \notin I$. \end{claim}

\begin{proof} We need only show that $\LLL_I$ is equal to $\XC_{W \setminus I}$, for the latter is clearly a 2-sided ideal. The inclusion $\LLL_I \subset \XC_{W \setminus I}$ follows from
the previous claim.

Now suppose that $\phi \in \Hom({\un{w}},{\un{y}})$ is in $\XC_{W \setminus I}$, and write $\phi$ as a linear combination of $\LLL$ maps. Unless $\phi$ is zero, some $\LLL_{w,\fb,\eb}$ has a
non-zero coefficient. Choose $w$, $\eb$, and $\fb$ successively such that each is maximal in the Bruhat/path dominance order relative to the constraint that, with the previous choices,
there is a nonzero coefficient for $\LLL_{w,\fb,\eb}$ in $\phi$. Because of upper triangularity, this is the only coefficient which can possibly contribute to a map from $Q_\eb$ to
$Q_{\fb}$, and it does contribute in a non-zero way. Therefore $\phi$ induces a nonzero map on $Q_\eb$, implying that $w \in I$. Since this is true for each maximal choice of $w$, we see
that $\phi \in \LLL_I$. \end{proof}

It is clear that if ${\un{w}}$ and $\un{w}'$ are rexes for $w \in W$, then any two rex moves $B_{\un{w}} \to B_{\un{w}'}$ are equal modulo $\LLL_{< w}$. Both rex moves are in $\LLL_{\le w}$, and induce the identity map on the unique $Q_w$ summand after localization.

Now let ${\un{y}}={\un{x}}$. Recall that $\LLL_{w,\fb,\eb} = \oLL_{{\un{x}},\fb} \circ \LL_{{\un{x}},\eb}$, factoring through ${\un{w}}$ in the middle. Suppose we compose them in the opposite order, to get an
endomorphism of $B_{\un{w}}$. We have an $\LLL$ basis for endomorphisms of $B_{\un{w}}$ as well, and there is a \emph{unique} light leaves morphism which induces a nonzero map on
the unique standard summand $Q_w$. This map can be any rex move, and though we often assume for convenience that it is the identity, this assumption is not necessary. Let $\psi(\eb,\fb)$
denote the coefficient of this light leaves morphism inside the composition $\LL_{{\un{w}},\eb} \circ \oLL_{{\un{w}},\fb}$.

\begin{claim} Let ${\un{x}},{\un{y}},{\un{z}}$ be arbitrary. Fix $w \in W$, and choose subsequences $\eb$ of ${\un{x}}$, $\fb$ and $\gb$ of ${\un{y}}$, and $\hb$ of ${\un{z}}$ which all express $w$. Then the
composition $\LLL_{\hb,\gb} \LLL_{\fb,\eb} \co B_{\un{x}} \to B_{\un{z}}$ is equal to $\psi(\gb,\fb) \LLL_{\hb,\eb}$ modulo $\LLL_{< w}$. \end{claim}

\begin{proof} $\LLL_{\hb,\gb} \LLL_{\fb,\eb}$ is a composition of four maps, the inner ones being $\LL_{{\un{y}},\gb} \oLL_{{\un{y}},\fb}$. This composition is equal to $\psi(\gb,\fb)$ times the
identity of $B_{{\un{w}}}$, plus maps in $\LLL_{< w}$. Therefore the overall composition is equal to $\psi(\gb,\fb) \oLL_{{\un{z}},\hb} \LL_{{\un{w}},\eb}$ modulo $\LLL_{< w}$, as desired. \end{proof}

\begin{claim} Let $a \co B_{\un{y}} \to B_{\un{z}}$ be an arbitrary morphism, and $\LLL_{w,\eb,\fb}$ be a light leaves map $B_{\un{x}} \to B_{\un{y}}$ factoring through the rex $B_{\un{w}}$. Then $a \LLL_{w,\eb,\fb}
= \sum_{\fb'} r_a(\fb,\fb') \LLL_{w,\eb,\fb'}$ modulo $\LLL_{< w}$. The sum runs over subexpressions $\fb'$ of ${\un{z}}$ expressing $w$. The coefficients $r_a(\fb,\fb')$ do not depend on $\eb$. \end{claim}

\begin{proof} Write $a \LLL_{\eb,\fb} = a \oLL_{{\un{y}},\fb} \LL_{{\un{w}},\eb}$, and consider $a \oLL_{{\un{y}},\fb}$ as a map from $B_{\un{w}} \to B_{\un{z}}$. The space of maps $B_{\un{w}} \to B_{\un{z}}$ modulo
$\LLL_{< w}$ is spanned by $\{ \oLL_{{\un{z}},\fb'}\}$ over all $\fb'$, so that $a \oLL_{{\un{y}},\fb} = \sum_{\fb'} r_a(\fb,\fb') \oLL_{{\un{z}},\fb'}$ modulo lower terms. Composing with $\LL_{{\un{x}},\eb}$
once more, we get the desired result. \end{proof}

\begin{defn} Let $\iota \co \DC \to \DC^{\op}$ denote the antiinvolution which preserves objects and flips diagrams upside-down. \end{defn}

Note that this reverses vertical composition, but not horizontal composition ($\iota$ is monoidal and contravariant). Clearly $\iota(\LLL_{\fb,\eb}) = \LLL_{\eb,\fb}$.

For the definition and basic properties of cellular categories, see \cite{West}.

\begin{prop} The category $\DC$ is cellular, with cellular basis $\LLL$ (for any appropriate choice of $\LLL$ maps) and antiinvolution $\iota$. \label{Discellular} \end{prop}

This proposition follows from the previous claim. The cells correspond to $w \in W$ with the Bruhat order, and the set $M({\un{x}},w)$ consists of all subsequences of ${\un{x}}$ expressing $w$. We do not know of any interesting interactions between the cellular structure and the monoidal structure.

\subsection{The diagrammatic character}

 Recall that for any ideal $I \subset W$ in the Bruhat order we have a (cellular) ideal $\LLL_I$ in $\DC$. For any coideal $J$ (i.e. $W \setminus J$ is an ideal) we consider the quotient category
\[
\DC^J \define \DC/\LLL_{W \setminus J}.
\]
A basis for morphisms in $\DC^J$ is given by double leaves maps which do not factor through $W \setminus J$. If $w$ is a minimal element in $J$ then the images of $B_{\un{w}}$ in $\DC^J$ for any rex $\un{w}$ are canonically isomorphic. (The difference of any two morphisms $B_{\un{w}} \rightrightarrows B_{\un{w}'}$ corresponding to rex moves $\un{w} \rightrightarrows \un{w'}$ lies in $\LLL_{W \setminus J}$.) Similarly for any rex $\un{w}$ we have
\begin{equation} \label{endquot}
\End_{\DC^J}(B_{\un{w}}) = R.
\end{equation}

Given any $w \in W$ we set
\[
\DC^{\ge w} \define \DC^{ \{ y \;|\; y \ge w \} }
\]
and the above remarks show that for any rex $\un{w}$ the object $B_{\un{w}}$ does not depend on the choice of reduced expression up to canonical isomorphism. Given any expression $\un{x}$, it follows from Theorem \ref{homspacetheorem} that $\Hom_{\DC^{\ge w}}(B_{\un{x}}, B_{\un{w}})$ is a free left $R$-module with basis the set of light leaves maps $\{ LL_{\un{x}, \eb} \}$ where $\eb$ is a subexpression of $\un{x}$ expressing $w$. It follows by Lemma \ref{lem:defects} that we have the identity
\begin{equation} \label{bsch}
\un{H}_{\un{x}} = \sum  \grk \Hom_{\DC^{\ge w}}(B_{\un{x}}, B_{\un{w}}) H_w
\end{equation}
where $\grk$ denotes the graded rank of the free $R$-module $\Hom_{\DC^{\ge w}}(B_{\un{x}}, B_{\un{w}})$.

We would like to extend this ``character'' map to the Karoubi envelope $\Kar(\DC)$. The problem is that for an arbitrary $\Bbbk$ and $B \in \Kar(\DC)$, the $R$-module $\Hom_{\DC^{\ge w}}(B_{\un{x}}, B_{\un{w}})$ is projective (as the summand of a free $R$-module), but is not necessarily free. Thus it is not a priori clear what $\grk$ should mean.

For this reason we assume that $\Bbbk$ is a local ring. By Nakayama's lemma and its graded version, direct summands of free graded $R$-modules are graded free.

\begin{defn}
We define the \emph{diagrammatic character} by
\begin{align*}
\ch : [\Kar(\DC)] & \to \HB \\
B & \mapsto \sum_{w \in W} \grk \Hom_{\Kar(\DC^{\ge w})}(B, B_{\un{w}})H_w.
\end{align*} 
\end{defn}

The diagrammatic character is obviously a homomorphism of abelian groups, and it is easy to check that
$\ch(v[B]) = \ch(B(1)) = v\ch(B)$. 
Hence $\ch$ is a homomorphism of $\Zvv$-modules. It is immediate from \eqref{bsch} that
\begin{equation} \label{eq:chhom}
\ch(B_{\un{x}}B_{\un{y}}) = H_{\un{x}\un{y}} = H_{\un{x}}H_{\un{y}} = \ch(B_{\un{x}})\ch(B_{\un{y}}).
\end{equation}
and so $\ch$ is a homomorphism on the $\Zvv$-submodule of $[\DC]$ generated by the isomorphism classes of Bott-Samelson bimodules.  In the next section will see that $\ch$ is an isomorphism of algebras if $\Bbbk$ is a complete local ring.

\subsection{Soergel's theorem}
\label{subsec-SThmDiag}

We present here another proof of Soergel's Categorification Theorem (Theorems \ref{SoergelThm1} and \ref{SoergelThm2}). This proof applies directly to $\DC$, but implies the corresponding theorem for Soergel bimodules via Theorem \ref{thm-equiv}. The proof is quite formal, relying only on general facts about Krull-Schmidt categories and Theorem \ref{homspacetheorem} showing that double leaves give a basis for Hom spaces between Soergel bimodules. We find our proof conceptually simpler than Soergel's original proof, although the complexity of the diagrammatic arguments in the final chapter does temper this somewhat.

Recall that an object $M$ in an additive category is indecomposable if $M \ne 0$ and $M \cong M' \oplus M''$ implies that one of $M'$ or $M''$ is zero. Recall that a Krull-Schmidt category is an additive category in which every object is isomorphic to a finite direct sum of indecomposable objects, and an object is indecomposable if and only if its endomorphism ring is local. Now assume that $\Bbbk$ is a complete local ring. It is known that any $\Bbbk$-linear idempotent complete additive category such that all Hom spaces are finitely generated is Krull-Schmidt. (This follows from the fact that any finitely generated $\Bbbk$ algebra is either local or admits a non-trivial idempotent.) Theorem \ref{homspacetheorem} shows that this condition is met for degree zero morphisms in $\DC$. We conclude:

\begin{lemma}
  If $\Bbbk$ is a complete local ring then the category $\Kar(\DC)$ is Krull-Schmidt.
\end{lemma}

The following is a diagrammatic variant of Soergel's theorem, classifying the indecomposable Soergel bimodules:

\begin{thm} \label{SoergelReproof} 
Assume that $\Bbbk$ is a complete local ring. Then for all $w \in W$ there exists a unique summand $B_w$ of $B_{\un{w}}$ which is not isomorphic to the shift of a summand of $B_{\un{v}}$ for any rex $\un{v}$ for $v < w$. The object $B_w$ does not depend on the reduced expression $\un{w}$ up to isomorphism. Moreover any indecomposable object in $\Kar(\DC)$ is isomorphic to a shift of $B_w$ for some $w \in W$. Hence one has a bijection:
\begin{align*}
W & \simto
\left \{ \begin{array}{c}
\text{indecomposable objects in $\Kar(\DC)$} \\
\text{up to shifts and isomorphism}
\end{array} \right \} \\
w & \mapsto B_w
\end{align*}
\end{thm}

\begin{proof}
  Fix a rex $\un{w}$ for $w$ and write the identity on $B_{\un{w}}$ as a sum of mutually orthogonal indecomposable idempotents:
\[
\1_{B_{\un{w}}} = e_1 + \dots + e_n.
\]
After localizing, each $e_i$ acts as an idempotent on $B_{\un{w}} \cong Q_w \oplus \bigoplus_{\eb \ne (1, \dots, 1)} Q_{\eb}$. Because $Q_w$ is indecomposable there exists a unique idempotent (say $e_1$) such that the restriction of $e_i$ to $Q_w$ is non-zero. (If we write each $e_i$ in terms of double leaves then $e_1$ is characterised as the unique idempotent with a non-zero coefficient of $\LLL_{\eb, w, \fb}$ where $\eb = \fb = (1, \dots, 1)$.) We define $B_w$ to be the image of $e$ in $\Kar(\DC)$. Hence for all $\un{w}$ we have constructed an indecomposable object $B_w$ in $\Kar(\DC)$.

It remains to show that any indecomposable object in $\Kar(\DC)$ is isomorphic to a shift of one of the objects $B_w$. So let $B$ be an arbitrary indecomposable object in $\Kar(\BC)$. That is, $B$ consists of a Bott-Samelson bimodule $B_{\un{x}}$ and an indecomposable idempotent $e \in \End(B_{\un{x}})$. For any $w \in W$ the ring $\End_{\DC^{\ge w}}(B_{\un{x}})$ is a quotient of $\End(B_{\un{x}})$. Fix $w$ maximal in the Bruhat order such that the image of $e$ in $\End_{D^{\ge w}}(B_{\un{x}})$ is non-zero. Equivalently, if we write $e$ in terms of double leaves
\[
e = \sum \lambda_{\eb,y,\fb} \LLL_{\eb, y, \fb}
\]
then $w$ is maximal such that some coefficient $\lambda_{\eb,w,\fb}$ is $\ne 0$. Hence in $\DC^{\ge w}$ we can write
\[
e = \sum \g_{\eb, \fb} (\oLL_{{\un{x}},\fb} \circ \LL_{\un{x},\eb})
\]
for some (homogenous) coefficients $\g_{\eb, \fb} \in R$, where the sum is over subexpressions $\eb, \fb$ of $\un{x}$ expressing $w$. Now assume that for all such subexpressions $\eb$ and $\fb$ with $d(\eb) + d(\fb) = 0$ ($d$ denotes the defect) we have
\[
 \LL_{\un{x},\eb} \circ e \circ \oLL_{{\un{x}},\fb} \in \mathfrak{m} \subset \Bbbk = \End^0_{\DC^{\ge w}}(B_{\un{w}})
\]
where $\mathfrak{m}$ denotes the maximal ideal of $\Bbbk$. Then by expanding $e^3 = e$ we conclude that each $\g_{\eb, \fb} \in R$ belongs to the ideal generated by $R^+$ and $\mathfrak{m}$ for all $\eb, \fb$. However both
\[ (R^+\End(B_{\un{x}}))^0 \quad \text{and} \quad \mathfrak{m} \End^0(B_{\un{x}}) \]
are contained in the Jacobson radical of $\End(B_{\un{x}})^0$. We obtain a contradiction, because no non-zero idempotent can be contained in the Jacobson radical.

We conclude that there exists subsequences $\eb'$ and $\fb'$ of $\un{x}$ such that $d(\eb') + d(\fb') = 0$ and such that
\[
\LL_{\un{x},\eb} \circ e \circ \oLL_{{\un{x}},\fb'} \in \Bbbk^{\times} \subset \Bbbk = \End^0_{\DC^{\ge w}}(B_{\un{w}}).
\]

Now let us return to $\Kar(\DC)$. Consider the composition
\[
B_{\un{w}} \stackrel{\oLL_{{\un{x}},\fb'}}{\longto} B \stackrel{\LL_{\un{x},\eb}}{\longto} B_{\un{w}}
\]
and recall the summand $B_w \subset B_{\un{w}}$ constructed earlier in the proof. These maps induce maps
\[
B_w \stackrel{i}{\to} B \stackrel{p}{\to} B_w
\]
such that the image of $p \circ i$ is invertible in $\End_{\DC^{\ge w}}(B_w)$. We conclude that $p \circ i$ does not belong to the maximal ideal of $\End(B_w)$ and hence is invertible. It follows that a shift of $B_w$ is isomorphic to a summand of $B$. However $B$ was assumed indecomposable, and hence $B \cong B_w(m)$ for some $m \in \ZM$. \end{proof}

\begin{cor}Assume that $\Bbbk$ is a complete local ring. The diagrammatic character
\[ \ch : [\Kar(\DC)] \to \HB \]
is an isomorphism of $\Zvv$-algebras.
\end{cor}

\begin{proof}
It is immediate from Theorem \ref{SoergelReproof} that:
\begin{enumerate}
\item $[\Kar(\DC)]$ is spanned by the classes $[B_{\un{x}}]$ for all expressions $\un{x}$;
\item the classes $\{ [B_x] \; | \; x \in W \}$ give a $\Zvv$-basis for $[\Kar(\DC)]$.
\end{enumerate}
Combining (1) with \eqref{eq:chhom} we conclude that $\ch$ is a homomorphism. Using the definition of the diagrammatic character and the construction of the objects $B_x$ we have
\[
\ch(B_x) = \sum_{y \le x} g_{y,x} H_y
\]
for some $g_{y,x} \in \Zvv$ with $g_{x,x} = 1$. Hence the set $\{ \ch(B_x) \; | \; x \in W \}$ is a basis for $\HB$, being upper triangular in the standard basis. By (2), $\ch$ maps a basis of $[\Kar(\DC)]$ to a basis of $\HB$ and hence is an isomorphism.
\end{proof}

One can check directly that $\ch(B_s) = \un{H}_s$ for all $s \in S$. Hence:

\begin{cor} \label{cor:homto[B]}
Assume that $\Bbbk$ is a complete local ring. The map $\un{H}_s \mapsto [B_s]$ defines a homomorphism $\HB \to [\Kar(\DC)]$.\end{cor}

\subsection{The equivalence to bimodules}

In this section we assume that $\hg$ is a Soergel realization over a field $\Bbbk$. In this section we prove that our diagrammatic category is equivalent to Soergel bimodules.

By Corollary \ref{cor:homto[B]} the map $\HB \to [\DC] : \un{H}_s \mapsto [B_s]$ is a homomorphism. By taking the graded ranks of Hom spaces in $\DC$ we obtain a pairing on $[\DC]$ which induces a semi-linear pairing $( - , - )_{\DC}$ on $\HB$ by pull-back. It is obvious from the diagrammatic description of $\DC$ that $B_s$ is self-biadjoint in $\DC$, so $\un{H}_s$ is self-biadjoint in this pairing. It follows that $( -, - )_{\DC}$ is determined by the trace $\e_\DC : \HB \to \Zvv : h \mapsto ( h, 1 )_{\DC}$. Because the degree of any light leaves map is given by the defect of the corresponding subexpression, it follows from Proposition \ref{homspaceprop} and Corollary \ref{cor:defects} that $\e_{\DC}$ agrees with the standard trace on objects of the form $\un{H}_{\un{w}}$ for all expressions $\un{w}$. As these elements generate $\HB$ we conclude that $\e$ and $\e_{\DC}$, and hence $( - , - )_{\DC}$ and the standard form $( -, - )$, agree.

In section \ref{subsec-functor} we constructed a monoidal functor $\FC : \DC \to \BSBim$. It induces a monoidal functor on the idempotent completions $\FC : \Kar(\DC) \to \SBim$.

\begin{thm} Under the above assumptions $\FC : \Kar(\DC) \to \SBim$ is an equivalence of monoidal categories. \label{thm-equiv}
\end{thm}

\begin{proof} Because idempotent completion preserves equivalences, it is enough to show that $\FC : \DC \to \BSBim$ is an equivalence. Clearly this functor is essentially surjective, so it is enough to show that it is fully-faithful. It is a theorem due to Libedinsky \cite{Lib} that images under $\FC$ of the dots, trivalent vertices and $2m_{st}$-valent vertices generate all moprhisms between Bott-Samelson bimodules. It follows that $\FC$ is full. Now by the above discussion and the fact that our realization is a Soergel realization, the graded dimensions of the homomorphism spaces in $\DC$ and $\BSBim$ conincide. We conclude that $\FC$ induces an isomorphism on Hom spaces, being a surjection between graded vector spaces of the same (finite) dimension in each graded component. Hence $\FC$ is fully-faithful.\end{proof}

\begin{remark} One can avoid the appeal to Libedinsky's result as follows. By adjunction, it is enough to prove that the images of dots, trivalent vertices and $2m_{st}$-valent vertices generate $\Hom(B_{\un{w}}, R)$ for any expression $\un{w}$. Now a fixed choice of light leaves maps $\LL_{\un{w},e}$ is obviously mapped to a composition of such maps, and is mapped to a linearly independent subset of $\Hom(B_{\un{w}}, R)$ by the same localisation argument as in the proof of Corollary \ref{LLalocalbasis}, this time carrried out in the localized category of Soergel bimodules. By comparing the degrees of $\LL_{{\un{w}},\eb}$ and the graded rank of $\Hom(B_{\un{w}}, R)$ one concludes that $\LL_{\un{w},e}$ spans $\Hom(B_{\un{w}}, R)$ as a graded $R$-module. Hence the result. (This is basically an adaption of Libedinsky's argument.)
\end{remark}

\section{Double Leaves Span}
\label{sec-LLLproof}

This chapter contains a diagrammatic proof that light leaves form a spanning set for Hom spaces. It is somewhat involved, and a key role is played by the recursive combinatorial structure of light leaves maps.

\subsection{Negative-positive decompositions}
\label{subsec-negpos}

For the rest of this chapter we will be interested in embedded graphs, not isotopy classes thereof. We will abusively use the term \emph{Soergel graph} to refer to a graph embedded without
horizontal tangent lines, so that no two vertices share the same $y$-coordinate. This kind of graph can be written as the product of the generators in the introduction, tensored with
identity maps. Recall that those generators were the bottom boundary dot, the top boundary dot, the trivalent split, the trivalent merge, the $2m$-valent vertex viewed as a map with $m$ inputs and $m$
outputs, and polynomials. In all our arguments, polynomials will be treated separately from other parts of a graph.

Given such a graph, it has \emph{height} $k$ if it is a product of $k$ generators, ignoring the polynomials. At a given $y$-coordinate without a vertex, we say the diagram has \emph{width}
$k$ if the object given by the horizontal line at that coordinate is $B_{\un{x}}$ for a sequence of length $k$ (i.e. if the horizontal line passes through $k$ strands). The \emph{maxwidth} of a
diagram is the maximal width attained.

We classify the generators as being \emph{positive}, \emph{neutral}, or \emph{negative}, depending on whether they increase, preserve, or decrease the width when reading from bottom to
top. Thus a top boundary dot and a splitting trivalent are positive, a $2m$-valent vertex and a box are neutral, and a bottom boundary dot and a merging trivalent are negative. Note that
light leaves $\LL_{{\un{x}},\eb}$ are constructed purely out of non-positive maps. The \emph{negative height} of a map is the number of negative generators used, and similarly for the
\emph{positive height}.

\igc{1}{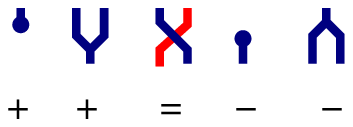}

The central generator is supposed to represent any $2m$-valent vertex.

A graph will be called \emph{negative-positive} if it consists of a composition of negative and neutral (\emph{non-positive}) maps followed by a composition of positive and neutral
(\emph{non-negative}) maps. In other words, the maximal width is attained on the outside; the map shrinks in width towards the middle and then expands again. We say the map is
\emph{strictly} negative-positive if the width shrinks non-trivially. Because polynomials are neutral, they can appear anywhere. Given a morphism in $\DC$, an expression for it as a
$\Bbbk$-linear combination of negative-positive graphs is called a \emph{negative-positive decomposition}.

Any map of the form $\LL_{{\un{x}},\eb}$ is constructed out of non-positive generators. Thus every double leaves map $\LLL$ is a negative-positive map, and the main theorem implies that every morphism in $\DC$ has a negative-positive decomposition.

We state some lemmas about negative-positive decompositions:

\begin{claim} Consider the Jones-Wenzl morphism as a degree $0$ map (i.e. the RHS of \eqref{twocoloridemp}). With the exception of the identity map, every other diagram is strictly
negative-positive. Moreover, every other diagram attains a width $\le m-2$. \end{claim}

\begin{proof} This is obvious. Perhaps it is most evident using the diagrammatic cellular structure on the Temperley-Lieb algebra. \end{proof}

\begin{claim} Any strictly non-positive map from $B_{\un{w}}$, for ${\un{w}}$ a reduced expression, is in the span of maps with a bottom boundary dot. \label{posrexhasdot} \end{claim}

\begin{proof} No rex can be the source of a merging trivalent $ss \to s$. Thus the map must consist of some rex move with polynomials followed by a bottom dot. We know,
using \eqref{dot2m}, that a dot on top of a $2m$-valent vertex yields a sum of diagrams, each of which has a bottom boundary dot. Thus we can ``pull" the dot successively through all the
$2m$-valent vertices in the rex move (ignoring any polynomials) until we have a bottom boundary dot. \end{proof}

\begin{claim} If $\phi = fgh$ where $f$ is non-negative, $h$ is non-positive, and $g$ has a (strictly) negative-positive decomposition, then $\phi$ has a (strictly) negative-positive
decomposition. \end{claim}

\begin{proof} This is obvious. \end{proof}

\begin{lemma} \label{rexmovesmlt} Let ${\un{w}}$ and ${\un{w}}'$ be two rexes for the same element $w \in W$, and let $\beta$ and $\beta'$ be two rex moves from ${\un{w}}$ to ${\un{w}}'$. Then $\beta -
\beta'$ has a strictly negative-positive decomposition. In particular, $\beta -
\beta'$ is in the span of diagrams having both a bottom and a top boundary dot.\end{lemma}

\begin{proof} We already know (see section \ref{subsec-RexMoves}) that two rex moves can be connected by a series of transformations. These transformations correspond to the Zamolodchikov
relations \eqref{A3std}, \eqref{B3std}, \eqref{H3std}, \eqref{A1I2mstd}, and the relation \eqref{2misomstd}. The difference between two rex moves which differ by a single transformation in
$\DC$ is given by the analog of each of these relations: \eqref{A3}, \eqref{B3}, \eqref{H3}, \eqref{A1I2m}, and \eqref{twocoloridemp}. Applying the transformation for $A_3$ or $B_3$ or
$A_1 \times I_2(m)$ will yield no difference between the rex moves. Applying the transformation for $H_3$ or \eqref{twocoloridemp} will have a difference with a strictly negative-positive
decomposition. More precisely, this transformation is applied somewhere within the rex move, but using the previous claim, the overall difference will still have a strictly
negative-positive decomposition. We can write $\beta - \beta' = (\beta - \beta_1) + (\beta_1 - \beta_2) + \ldots +(\beta_k -\beta')$ where each successive difference corresponds to a
single transformation. The result follows. \end{proof}

For $w \in W$ and an arbitrary sequence ${\un{x}}$ we write $w \le {\un{x}}$ if there exists a subsequence $\eb$ of ${\un{x}}$ expressing $w$.

\begin{lemma} \label{rexesonlyplz} Let ${\un{x}}$ be a sequence. The identity of $B_{\un{x}}$ has a negative-positive decomposition where each term factors through some $B_{\un{w}}$ for some reduced
expression ${\un{w}}$ for $w \le {\un{x}}$. \end{lemma}

\begin{proof} Let us use induction on the length of ${\un{x}}$. Whenever ${\un{x}}$ is a reduced expression the statement is trivial. Suppose that ${\un{x}}$ contains a repeated index $ss$. One can apply
the relation \eqref{iiidemp2} give a negative-positive decomposition where each term factors through ${\un{x}}'$, the sequence identical to ${\un{x}}$ except with one copy of $s$ removed. Applying
the inductive hypothesis to ${\un{x}}'$ and nesting negative-positive decompositions, we have the result for ${\un{x}}$.

Suppose that ${\un{x}}$ is not a reduced expression. There is some finite sequence of braid relations $\ubr{sts\ldots}{m} = \ubr{tst\ldots}{m}$ which, when applied to ${\un{x}}$, yield a new
sequence which has a repeated index. Let us induct on the number of such relations which need to be applied before a repeated index is reached. For each relation applied, we can use
\eqref{twocoloridemp} to replace the identity $\1_{sts\ldots}$ inside $\1_{\un{x}}$ with the doubled $2m$-valent vertex (the LHS of \eqref{twocoloridemp}) plus a linear combination of strictly
negative-positive maps. The doubled $2m$-valent vertex yields a neutral map factoring through some sequence ${\un{x}}'$ which is closer to having a repeated index. Therefore, by induction,
$\1_{\un{x}}$ has a negative-positive decomposition as desired. \end{proof}

\subsection{Modulo lower terms}
\label{subsec-modulolower}

Let us fix an element $w \in W$ with a reduced expression ${\un{w}}$. Let $I_{\un{w}}$ denote the right ideal (bottom ideal, if you think diagrammatically) consisting of linear cominations of diagrams with arbitrary source $\un{x}$, fixed target $\un{w}$ and which are strictly positive on top. We have already shown that this is the same as the ideal generated by all the top boundary dots. The elements of this ideal are the so-called \emph{lower terms} (with respect to $\un{w}$).
In the absence of localization, taking the quotient by $I_{\un{w}}$ is the best way to talk about maps to $Q_w$. After all, $Q_w$ is the unique standard summand which is joint kernel of the top
boundary dots. When ${\un{w}} = \emptyset$, the ideal $I_{\un{w}}$ is zero.

Let ${\un{w}}'$ be the reduced expression obtained after placing a $2m$-valent vertex above ${\un{w}}$. The $2m$-valent vertex, as a map from $B_{\un{w}} \to B_{{\un{w}}'}$, clearly sends $I_{\un{w}} \to
I_{{\un{w}}'}$, because we can ``pull" top-dots through $2m$-valent vertices. The color-reversed $2m$-valent vertex gives a map $B_{{\un{w}}'} \to B_{\un{w}}$ sending $I_{{\un{w}}'} \to I_{\un{w}}$. The
composition of these two maps, minus the identity, has a strictly negative-positive decomposition, and thus consists of lower terms. Therefore the doubled $2m$-valent vertex acts as the
identity modulo $I_{\un{w}}$. We see that, for any ${\un{x}}$, the spaces $\Hom(B_{\un{x}},B_{\un{w}})/I_{\un{w}}$ and $\Hom(B_{\un{x}},B_{{\un{w}}'})/I_{{\un{w}}'}$ are isomorphic. In fact, for any two rexes, these spaces of
morphisms modulo lower terms are all \emph{canonically} isomorphic. After all, for two arbitrary rexes ${\un{w}}$ and ${\un{w}}'$, we can use any rex move to give an isomorphism 
$\Hom(B_{\un{x}},B_{\un{w}})/I_{\un{w}} \to \Hom(B_{\un{x}},B_{{\un{w}}'})/I_{{\un{w}}'}$, and any two rex moves are equal modulo lower terms by Lemma \ref{rexmovesmlt}.

We are interested in the span of the maps $\LL_{{\un{x}},\eb}$ with target ${\un{w}}$, modulo lower terms. By the previous paragraph, we do not care which rex ${\un{w}}$ we chose, or what
rex move is applied at the very end of the construction of $\LL_{{\un{x}},\eb}$. However, other choices of rex moves in the construction of $\LL_{{\un{x}},\eb}$ may still be important.

Our first step towards showing that $\LLL$ forms a basis for all Hom spaces is showing that $\LL$ forms a basis for maps to rexes, modulo lower terms.

\begin{prop} \label{LLabasismlt} Let ${\un{x}}$ be arbitrary and ${\un{w}}$ be a rex for $w \in W$. Choose a map $\LL_{{\un{x}},\eb}$ for each $\eb$ expressing $w$. These maps form a basis for
$\Hom(B_{\un{x}},B_{\un{w}})/I_{\un{w}}$, under the action of $R$ on the left. \end{prop}

We briefly defer the proof of this proposition. The discussion above implies that knowing this Proposition for a single rex ${\un{w}}$ of $w$ will imply the result for every
rex of $w$.

We already know linear independence using localization arguments, so it is enough to show that they span. We will prove this by induction, but the induction will not be easy. After all,
the base case where ${\un{x}}={\un{w}}=\emptyset$ is already a highly non-trivial statement: that all diagrams without boundary reduce to the empty diagram (with polynomials). We do not know how to
show this statement directly (say, with graph theory) except in type $A$ or dihedral type; the equivalent statement for $\DStd$ was shown only using nontrivial arguments involving homotopy groups.

Identical statements can be made about maps from ${\un{w}}$ using the vertical flip map which is an antiinvolution. The proposition as stated above is equivalent to one
saying that maps from ${\un{w}}$ modulo terms with bottom dots are spanned by $\oLL_{{\un{x}},\eb}$.

\subsection{Reduction to working modulo lower terms}
\label{subsec-reduction1}

\begin{proof}[Proof of Theorem \ref{homspacetheorem} assuming Proposition \ref{LLabasismlt}] Fix sequences ${\un{x}}$ and ${\un{y}}$ and a morphism $\phi \co {\un{x}} \to {\un{y}}$. We want to write $\phi$ as
an $R$-linear combination of diagrams of the form $\oLL_{{\un{y}},\fb} \LL_{{\un{x}},\eb}$ where the subsequences express a common element $w \in W$.

The identity map of ${\un{x}}$ has a negative-positive decomposition where every term factors through a reduced expression, by Lemma \ref{rexesonlyplz}. We only need to work with one diagram at
a time, so without loss of generality we assume that $\phi=fg$ factors as $g \co {\un{x}} \to {\un{w}}$ and $f \co {\un{w}} \to {\un{y}}$ for some reduced expression ${\un{w}}$. We will prove the statement by
induction on ${\un{w}}$. That is, we assume that any morphism $\phi=fg$ which factors through a rex $\vb$ for $v<w$ is in the span of $\LLL$. The base case follows from the same arguments.

Write $g = g_w + g_l$ and $f = f_w + f_l$. Here $g_w$ is an $R$-linear combination of $\LL_{{\un{x}},\eb}$, and $g_l \in I_{{\un{w}}}$; similarly, $f_w$ is in the span of $\oLL_{{\un{y}},\fb}$ and $f_l
\in \overline{I}_{{\un{w}}}$. This decomposition is guaranteed by Proposition \ref{LLabasismlt}. The composition $f_w g_w$ is clearly in the span of $\LLL$. This is sufficient to prove the base
case where $w = e \in W$, since it is clear that $g_l=f_l=0$. Now we need to account for the lower terms. It will suffice to show that $\psi g_l$ is in the span of $\LLL$ for
any $\psi$ (the argument with $f_l$ is the same, upside-down).

Consider $\psi g_l$. Now $g_l \in I_{{\un{w}}}$ so it is generated by top boundary dots, and we can separate $g_l$ into terms each generated by a single top boundary dot. Thus the composition
looks as follows:

\igc{1}{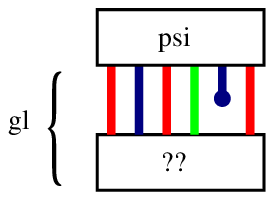}

Let $\zb$ denote the sequence consisting of ${\un{w}}$ with the dotted index removed. This morphism factors through $\zb$, and the identity of $\zb$ has a decomposition which factors through
reduced expressions for elements $v$ found as subsequences of $\zb$. Subsequences of $\zb$ are strict subsequences of ${\un{w}}$, so $v < w$. Now induction implies that $\psi g_l$ is in the
span of $\LLL$, as desired. \end{proof}

\subsection{The grand induction}
\label{subsec-hereitis}

Now comes the crux of the argument, a giant induction on maxwidth to prove Proposition \ref{LLabasismlt}. This entire section represents the proof.

Fix a rex ${\un{w}}$ for $w$. Let $X_{{\un{x}}}$ denote the set of all possible maps which are valid constructions of $\LL_{{\un{x}},\eb}$ with target ${\un{w}}$, but which also have polynomials in any
region. We omit ${\un{w}}$ from the notation $X_{\un{x}}$ because, as noted, the choice of reduced expression giving the final target does not matter, because the canonical isomorphisms between Hom spaces modulo lower terms
preserve the sets $X_{\un{x}}$.

Fix $M \ge 0$ and consider the following two statements:

\vspace{0.5cm}

($L_M$) For any ${\un{x}}$ with $\ell({\un{x}}) \le M$ and any ${\un{w}}$ with $\ell({\un{w}}) \le M$, choose a single map $\LL_{{\un{x}},\eb}$ for each appropriate sequence $\eb$ with target ${\un{w}}$. Then every map
$B_{\un{x}} \to B_{\un{w}}$ of maxwidth $\le M$ is in the left $R$-span of the $\LL_{{\un{x}},\eb}$ modulo $I_{\un{w}}$.

\vspace{0.5cm}

($X_M$) For any ${\un{x}}$ with $\ell({\un{x}}) \le M$ and any ${\un{w}}$ with $\ell({\un{w}}) \le M$, every map $B_{\un{x}} \to B_{\un{w}}$ of maxwidth $\le M$ is in the $\Bbbk$-span of $X_{\un{x}}$ modulo $I_{\un{w}}$.

\vspace{0.5cm}

Obviously $(L_M)$ is stronger than $(X_M)$. Our induction will use $(X_M)$ and $(L_{M-1})$ to prove $(L_M)$, and $(L_M)$ to prove $(X_{M+1})$. The base case is $M=0$. A map of maxwidth $0$ necessarily has $\un{w} = \emptyset$ and is just a polynomial, so both $(L_0)$ and $(X_0)$ hold. Now we fix ${\un{x}}$ and ${\un{w}}$. Any map $B_{\un{x}} \to B_{\un{w}}$ will have maxwidth $M \ge \ell({\un{w}})$ and $M \ge \ell({\un{x}})$, and if the map is in
$X_{\un{x}}$ then it has maxwidth precisely $M=\ell({\un{x}})$. Thus the statements $(L_M)$ and $(X_M)$ are vacuous for maps to ${\un{w}}$ when $M < \ell({\un{w}})$. When proving the inductive statement for maps to
${\un{w}}$, there will be two separate cases: $M=\ell({\un{w}})$, which we think of as the ``base case" for ${\un{w}}$ because it does not use induction; and $M>\ell({\un{w}})$.

Suppose that $M=\ell({\un{w}})$, and consider a graph $B_{\un{x}} \to B_{\un{w}}$. There can be no negative maps on top of the diagram, because of the maxwidth constraint, so the diagram ends with a
non-negative map. Unless the diagram is purely neutral, it ends with a strictly non-negative map which, by the upside-down version of Claim \ref{posrexhasdot}, implies that the diagram
lies within $I_{\un{w}}$. Hence we can assume that the diagram is neutral, and $\ell({\un{x}})=\ell({\un{w}})$. Any neutral map must consist only of polynomials and $2m$-valent vertices, so it is a rex move
with polynomials, and ${\un{x}}$ is also a rex for $w$. Such a map is an element of $X_{\un{x}}$ as desired, when $\eb=(1,1,\ldots,1)$. To show $(L_M)$ we need to show that any single rex move, with
polynomials only in the leftmost region, will span morphisms consisting of compositions of all rex moves with all polynomials modulo $I_{\un{w}}$. First we use the polynomial forcing relation \eqref{dotforcegeneral} to move the polynomials
to the leftmost region. This leaves behind terms where strands in the rex move are broken, but such terms are in $I_{\un{w}}$, because they have a strictly negative-positive decomposition. Now
we apply Lemma \ref{rexmovesmlt} to show that the difference between two rex moves is also in $I_{\un{w}}$.

Thus we have shown $(L_M)$ and $(X_M)$ for the case when $M=\ell({\un{w}})$. We assume henceforth that $\ell({\un{w}})<M$.

\begin{proof}[$(L_{M-1})$ and $(X_M) \implies (L_M)$] We need to show that each element of $X_{\un{x}}$ is in the span of our particular fixed choice of $\LL$ maps. Since the width of a map in
$X_{\un{x}}$ is precisely $\ell({\un{x}})$, induction already works unless $\ell({\un{x}})=M$, which we now assume.

To get from an element of $X_{\un{x}}$ associated to $\eb$ to our particular choice of $\LL_{{\un{x}},\eb}$, we need to force all polynomials to the left, and change our rex moves (possibly changing
the intermediate rexes ${\un{w}}_k$ as well). Let us refer to a rex move in the $\LL$ construction by which $\phi_k$ it appears in, and by whether it appears in the form $\a$ or $\b$ in Figure
\ref{fourmaps}. In fact, rex moves of the form $\b$ inside $\phi_k$ can be viewed instead as part of $\a$ inside $\phi_{k-1}$, so we shall assume that $\b$ is always trivial, and the
location of the rex move is specified by which $\phi_k$ it appears in. With this assumption, ${\un{w}}_k = {\un{w}}^s_k$ always.

Suppose that we are building a new light leaves map for some $({\un{x}},\eb)$ expressing $w$, using arbitrary choices, but we make a single error. While performing a rex move, we
accidentally insert some map with a strictly negative-positive decomposition. We call a graph of this form an ``$\LL$ map with error" or simply an \emph{error term}. Polynomial forcing
\eqref{dotforcegeneral} implies that the cost of sliding a polynomial across a rex move is adding an error term. By inspection, every region in an $\LL$ map is separated from the leftmost
region by the reduced expression ${\un{w}}_k$, or a subset thereof, so that we can assume that all polynomials are on the left, modulo error terms.

Lemma \ref{rexmovesmlt} implies that the difference between two rex moves is an error (that is, what would be an error if plugged in to the $\LL$ construction). Thus if two constructions
of $\LL$ differ only in the choice of rex moves, then the difference between the $\LL$ moves is spanned by error terms. Two constructions can also differ in their choice of ${\un{w}}_k =
{\un{w}}^s_k$, but this can be accounted for by changing rex moves as well. One can replace the identity of ${\un{w}}_k$ with a rex move $fg$ where $g \co {\un{w}}_k \to {\un{w}}'_k$ and $f \co {\un{w}}'_k \to
{\un{w}}_k$. Then, viewing $g$ as part of $\a$ in $\phi_k$ and $f$ as part of $\a$ in $\phi_{k+1}$, we can effectively replace ${\un{w}}_k$ with ${\un{w}}'_k$. (This still works even when dealing with
expressions having $s$ on the right, because we can split the rex move into the parts that affect $s$ and the parts that do not.) Hence, two different constructions of the same $\LL$
differ by error terms. Thus to show that a particular choice of $\LL$ maps spans them all, we need only show that that this choice of $\LL$ maps spans any error term.

This discussion of error terms is not strictly required for the proof we give next, but does say that the difference between $\LL$ maps is under control, and is useful for picturing how
one would, by hand, attempt to reduce a diagram into the desired form. We will use the notion of error terms in the later proofs.

Note that adding an error does not affect the maxwidth of a map. Let $F$ denote the partially constructed error term $\LL_{M-1}$, after the first $M-1$ steps have been applied. Now $F$ is
a map of width $\le M-1$, so we may apply induction and replace $F$ with an $\LL$ map \emph{of our choice}, or with a term having dots on top. By adjusting the rex moves in $\phi_{M-1}$
and $\phi_M$ we assume that whenever $F$ is an $\LL$ map, the top of $F$ is our desired ${\un{w}}_{M-1}$. Now we apply $\phi_M$ to finish the construction. We draw the four possibilities.
\igc{1}{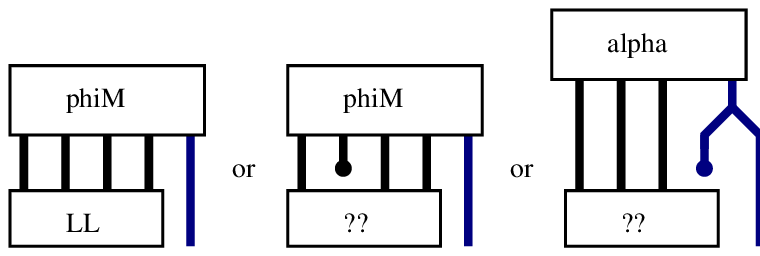}

As noted previously, any error in $\phi_M$ itself yields a map contained in $I_{\un{w}}$ because some dot pulls to the top, so we can assume $\phi_M$ is error free, and that
it has our desired rex move. If $F$ is a term which has a dot on top, then this dot will pull through $\phi_M$ to become a dot on the very top except in a single case: when $\eb_M$ is D0
or D1, and the dot is on the final strand.

We have sufficient choice of $\phi_M$ such that the first diagram is actually one of our specified $\LL$ maps! The second diagram is in $I_{\un{w}}$. For the third and fourth diagrams, we can
reapply induction to $F$ without the dot on the blue strand, thus replacing the question marks with $\LL$ maps of our choice (or with more terms of the second type). Because the top of $F$
was a reduced expression ending in $s$, removing $s$ from the end still yields a reduced expression for which $s$ is \emph{not} in the right descent set. Now the overall expression is our
desired $\LL$ map, for $e_M$ being U1 or U0 respectively (again, up to the freedom we have to alter $\a$). \end{proof}

\begin{proof}[$(L_{M-1}) \implies (X_M)$] Now let us fix ${\un{w}}$ with $\ell({\un{w}})<M$, and prove that $(L_{M-1})$ implies $(X_M)$ for ${\un{w}}$. Consider a diagram of maxwidth $\le M$. If the
maxwidth $M$ is never attained then we may use induction, so let us assume that width $M$ is attained at least once. Width $M$ may be maintained for a period of time using $2m$-valent
vertices, and eventually the width may drop again to $M-1$ and rise back to $M$ for another interval.

\begin{claim} \label{goesbackdown} Assume $(L_{M-1})$. A map $B_{\un{x}} \to B_{\un{w}}$ with $\ell({\un{x}}) \le M-1$ which only reaches width $M$ for one continuous interval is in the span of $X_{\un{x}}$ modulo $I_{\un{w}}$.
\end{claim}

\begin{claim} \label{staysup} Assume $(L_{M-1})$. A map $B_{\un{x}} \to B_{\un{w}}$ with $\ell({\un{x}}) = M$ which only reaches width $M$ for one continuous interval starting at the
bottom is in the span of $X_{\un{x}}$ modulo $I_{\un{w}}$. \end{claim}

Suppose that we can show these two claims. Then we may simplify the diagram from top to bottom as follows. Consider the first interval where width $M$ appears. If that extends all the way
to the bottom of the map, we can use Claim \ref{staysup} to conclude that $X_M$ holds. If the interval ends and we return to width $M-1$, we can apply Claim \ref{goesbackdown} to replace that region with
an $\LL$ map (by $(L_{M-1})$, any $\LL$ map we choose), and thus the whole region now stays below width $M-1$. After doing this, the next interval becomes the first interval, and we repeat
the argument.

We proceed to prove Claim \ref{staysup} first. Let $\eta$ denote the map in question (having bottom ${\un{x}}$), ${\un{y}}$ the topmost sequence where width $M$ is attained, and $\un{z}$
the sequence immediately above, having $\ell(\un{z})=M-1$. The map $B_{\un{z}} \to B_{\un{w}}$ has maxwidth $M-1$ so by induction we can assume it is any construction of
$\LL_{{\un{z}},\eb}$ we choose. We can ignore polynomials on the left of the diagram. We begin by treating the case without $2m$-valent vertices on bottom, where ${\un{x}}={\un{y}}$. There
are two cases, to be treated differently.

\igc{1}{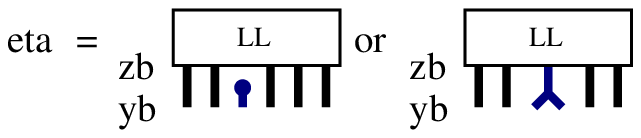}

\begin{claim} \label{addnegtri} Suppose a negative (merging) trivalent is added below a light leaves map. The result is light leaves. \end{claim}

\begin{proof} Suppose that the new trivalent vertex is attached to the $k$-th strand, with ${\un{x}}_k=s$. There are four choices for $\eb_k$: U0, U1, D0, D1. The composition will be a light
leaves with $(s\ s)$ instead of ${\un{x}}_k$ and one of four sequences instead of $\eb_k$: (U1 D1), (U1 D0), (D0 D0) or (D0 D1) respectively. This is pictured below:
\[ \ige{1}{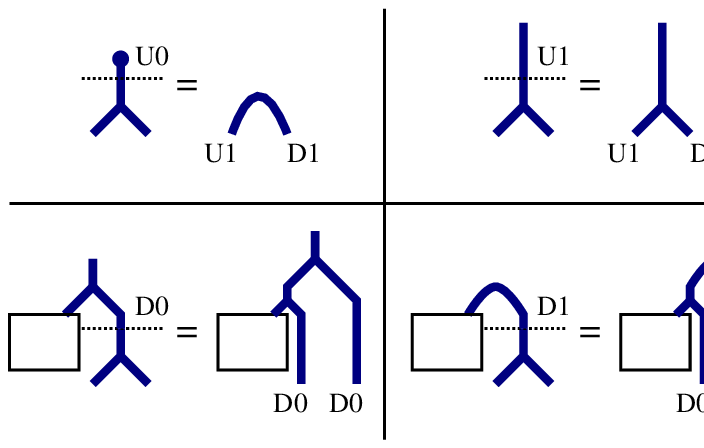} \qedhere \]
\end{proof}

\begin{claim} \label{addnegdot} Suppose that we add a negative (bottom) dot after the $k$-th strand of an $\LL_{\zb,\eb}$ of our choice, yielding a map from $B_{\un{y}}$. The result is in
$X_{\un{y}}$. \end{claim}

\begin{proof} Suppose that the new boundary dot is colored $s$. If $s$ is not in the right descent set of $w_k$, then the result is an $\LL$ map, with the new strand being U0. If $s$ is in
the right descent set, by choosing our $\LL_{\zb,\eb}$ appropriately, we can assume that $s$ occurs on the right of ${\un{w}}_k$. Now the statement follows from the equality below, because
both terms on the RHS are in $X_{\un{y}}$. We use the decomposition of Remark \ref{LLof111}, and relation \eqref{dotforcegeneral}. Remember that the use of $\frac{1}{2}$ is unnecessary, as any dual basis will do.
\[ \ige{1}{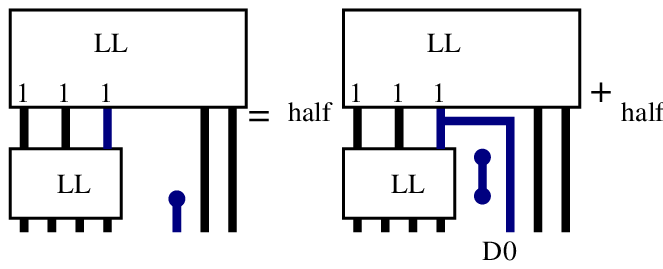} \qedhere \]
\end{proof}

\begin{remark} While this is the only use in this chapter of the assumption of Demazure Surjectivity, it is significant. Without this assumption, double leaves will not form a
basis in the simplest counterexample: maps from $B_s \ot B_s \to B_s$. \end{remark}

Now we want to add more $2m$-valent vertices below this $\LL$ map. First we prove a useful inductive lemma.

\begin{claim} \label{splitintwain} Assume $(L_{M-1})$, and let $\ell({\un{x}})=M$. Suppose we have a map $B_{\un{x}} \to B_{\un{w}}$ which never returns to width $M$ after it leaves it. Suppose further that
the map splits as follows for some $0<k<M$.

\igc{1}{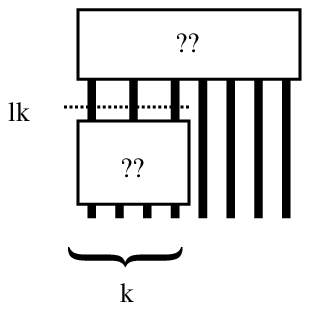}

Then the map is in the span of $X_{\un{x}}$, modulo $I_{\un{w}}$. \end{claim}

\begin{proof} Both question mark boxes have width strictly less than $M$, so by induction, we can assume that these boxes contain $\LL$ maps of our choice, or morphisms with dots on top.
The rest of the proof is exactly the same as the proof in the previous section. We induct on the width of the sequence $\vb$ with the dashed line through it. By an application of Lemma
\ref{rexesonlyplz} we can assume without loss of generality that $\vb$ is actually a reduced expression for $v \in W$. We write the lower box as $f_v + f_l$, where $f_v$ is a linear
combination of light leaves and $f_l$ is in $I_\vb$, and we write the upper box as $g_{w,1} + g_{w,0} + g_l$ where $g_l$ is in $I_{\un{w}}$, $g_{w,1}$ is a linear combination of light leaves
beginning with all U1 for $\vb$, and $g_{w,0}$ is a linear combination of light leaves which do not begin with U1's, meaning that somewhere on $\vb$ is a U0. In the product, combining
$f_v$ with $g_{w,1}$ will give valid constructions for an element of $X_{\un{x}}$. Combining anything with $g_l$ will be in $I_{\un{w}}$. For a term with either $f_l$ or $g_{w,0}$, we draw a new
dashed line to avoid the dot; this decreases the width of the dashed line by one. Induction then finishes the proof. \end{proof}

\begin{claim} \label{add2m} Assume $(L_{M-1})$, and let $\ell({\un{x}})=M$. Placing a $2m$-valent vertex below a diagram in $X_{\un{w}}$ will result in a diagram in the span of $X_{\un{x}}$ modulo $I_{\un{w}}$.
\end{claim}

\begin{proof} Any light leaves map with a neutral map below it is non-positive, so that it will never return to width $M$ after it departs. Suppose that we add the $2m$-valent vertex of
colors $s,t$ to strands ${\un{x}}_{l+1}$ through ${\un{x}}_{l+m}$ of a light leaves map $\LL_{{\un{x}},\eb}$. For any $0< k \le l$ or $l+m \le k < M$, we can split the diagram into $\LL_{\le k}$ and the
remainder as in Remark \ref{LLof111}, placing the new $2m$-valent vertex below whichever half is appropriate. Therefore, if any such $\LL_{\le k}$ maps to a shorter sequence, we can simply
use Claim \ref{splitintwain} to finish the proof. This allows us to reduce to the following special case: ${\un{x}}_{\le l}$ is a reduced expression and $\eb_{\le l}$ is all U1 (this includes
the possibility $l=0$); unless $l+m=M$, ${\un{x}}_{\le M-1}$ is a reduced expression and $\eb_{\le M-1}$ is all U1. Any time where $\eb_{l+1}$ through $\eb_{l+m}$ are all U1, the $2m$-valent
vertex can be viewed as part of the rex move in $\LL_{{\un{x}},\eb}$, so the result is clearly in $X_{\un{x}}$. Hence, we assume that $M=l+m$, and $({\un{x}},\eb)_{\le l}$ is a reduced expression for
$w_l \in W$.

Furthermore, we can alter any rex moves in $\LL_{{\un{x}},\eb}$ at will. The difference term will have an error. If this error occurs in $\phi_k$ for $k<M$ then we can use induction on the
error version of $\LL_{\le k}$, which has width $\le M-1$, to rewrite it in terms of $\LL$ maps and graphs with dots on top. These can be dealt with in the familiar manner. If the error occurs in $\phi_M$, it produces a dot on top, and is thus in $I_{\un{w}}$.

Hence, we can choose any desired reduced expression ${\un{w}}_l$ for $w_l$. Now $w_l = uv$ for $v \in W_{s,t}$ the dihedral group, and $u$ a minimal right coset representative of $W_{s,t}$. We
can choose a rex for $w_l$ of the form $\ub \vb$. Now the remaining terms of ${\un{x}}$ are all $s$ or $t$, and given any sequence of 1's and 0's, the assignment of U and D only depends on $v$
and not on $u$. Moreover, the rex moves required to send $s$ or $t$ to the right when a D appears (i.e. the rex moves of type $\b$ in Figure \ref{fourmaps}) can all be performed within
$\vb {\un{x}}_{>l}$, preserving $\ub$. Suppose that they are so in $\LL_{{\un{x}},\eb}$. Then the reduction to light leaves form does not depend on $\ub$ at all. We could assume without loss of
generality that $w_l=v \in W_{s,t}$, and the entire diagram only contains $2$ colors! Morphism spaces in $\DC$ for dihedral groups are already proven in \cite{EDihedral} to be the right
size, which implies the result; alternatively, one could do a case by case analysis. \end{proof}

\begin{remark} The case by case analysis can be quite interesting, although tedious. The following is a worthwhile exercise. Let $m=3$, and $\vb=sts$. Let ${\un{x}}=ststst$ and
$\eb=(U1,U1,U1,D0,D0,D0)$. First draw $\LL_{{\un{x}},\eb}$, and then place a $2m$-valent vertex below the D0's, and transform the diagram into light leaves format for $stssts$. This involves
repeated applications of \eqref{assoc2}. The exact same calculation works for general $m$. \end{remark}

We have now proven Claim \ref{staysup}. In order to prove Claim \ref{goesbackdown}, we need only put a positive generator below a map in $X_{\un{y}}$ for $\ell({\un{y}})=M$.

\begin{claim} \label{addpostri} Assume $(L_{M-1})$. Suppose that $\ell({\un{y}})=M$ and we add a positive (splitting) trivalent below a map in $X_{\un{y}}$ to obtain a map $B_{\un{x}} \to
B_{\un{w}}$. Then the result is in the span of $X_{\un{x}}$, modulo $I_{\un{w}}$. \end{claim}

\begin{proof} Suppose that the trivalent vertex is added to ${\un{y}}_k$ and ${\un{y}}_{k+1}$. There are 8 consistent possibilities for $\eb_k$ and $\eb_{k+1}$: $\eb_k$ can be any of U0, U1, D0, D1,
and it determines whether $\eb_{k+1}$ is U or D. In fact, if $\eb_k$ is U0 or D1, the result is manifestly in light leaves format.

\igc{1}{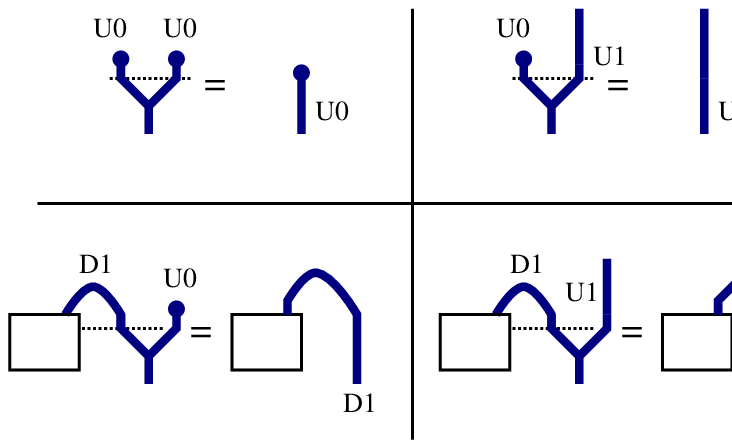}

On the other hand, if $\eb_k$ is U1 or D0, the result locally looks like the following.

\igc{1}{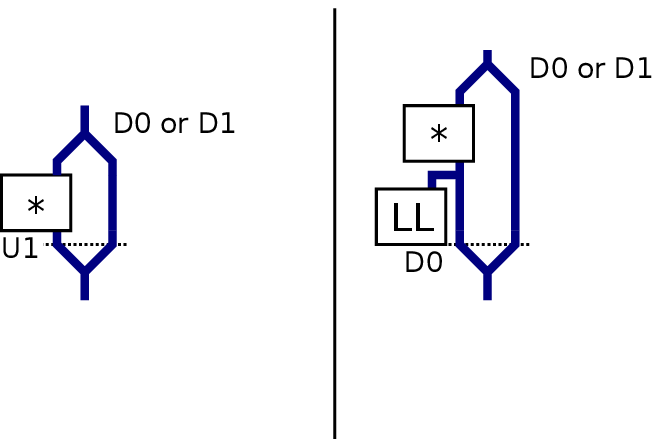}

By ``D0 or D1" we mean that the topmost trivalent in these graphs is as pictured when $\eb_{k+1}$ is D0, and has a further dot on top when $\eb_{k+1}$ is D1; this extra dot will not affect
our discussion. The dashed line indicates where width $M$ is reached. The asterisk box in these graphs is some rex move between two rexes which both have $s$ on the far right.

In fact, we believe both these diagrams to be zero. Let us pause to state a computational conjecture.

\begin{conj} \label{conj:rextri} For any rex move $\b \co \un{w} \to \un{w'}$, where both $\un{w}$ and $\un{w'}$ have $s$ on the far right, one has \begin{equation} \label{eq:slidetrithrurex} 
	{
	\vspace*{.4 cm}
	\labellist
	\small\hair 2pt
	 \pinlabel {$\b$} [ ] at 20 20
	 \pinlabel {$\b$} [ ] at 100 44
	 \pinlabel {$\un{w'}$} [ ] at 20 75
	 \pinlabel {$\un{w'}$} [ ] at 100 75
	 \pinlabel {$\un{w}$} [ ] at 20 -10
	 \pinlabel {$\un{w}$} [ ] at 100 -10
	 \pinlabel {$s$} [ ] at 47 53
	 \pinlabel {$s$} [ ] at 128 29
	\endlabellist
	\centering
	\ig{1}{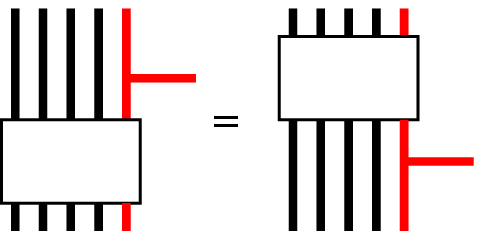}
	}.
	\vspace*{.4 cm} \end{equation} 

This would imply, using \eqref{assoc1} and \eqref{needle}, that the following is zero. \begin{equation} \label{rexmovestriis0} \ig{1}{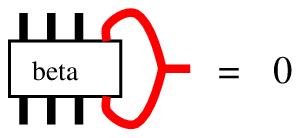} \end{equation} \end{conj}

Let us continue the proof without assuming this conjecture.

Which rex move appears in the asterisk box? If the rex move does not involve the rightmost $s$-colored strand, then clearly the result is zero by \eqref{assoc1} and \eqref{needle}. Since
two rex moves are equal modulo lower terms, we may assume that the asterisk box actually contains a strictly negative map followed by a strictly positive map. Moreover, Lemma
\ref{rexmovesmlt} guarantees that both the top and bottom of the asterisk box has a boundary dot.

We will now use these boundary dots to reduce the maxwidth of the diagram below $M$. The bottom dot (resp. top dot) appears either in one of the initial $k-1$ strands, or on the final
$s$-colored strand. If the dot appears on the final $s$-colored strand, then it will pull into the nearby trivalent vertex by \eqref{unit}. Otherwise, using rectilinear isotopy, the dot
can be pulled past this trivalent vertex. After performing these operations to the top and bottom dots, the resulting diagram never factors through a sequence of width $M$, and therefore
$(L_{M-1})$ implies that it can be expressed in the span of light leaves. \end{proof}

\begin{claim} \label{addposdot} Assume $(L_{M-1})$. Suppose that $\ell({\un{y}})=M$ and we add a positive (top) dot below a map in $X_{\un{y}}$ to obtain a map $B_{\un{x}} \to B_{\un{w}}$. Then the result is in
the span of $X_{\un{x}}$, modulo $I_{\un{w}}$. \end{claim}

\begin{proof} Suppose that the dot is attached to ${\un{y}}_l$ in $\LL_{{\un{y}},\eb}$. Let us draw a number of possibilities for what can happen to the dot.

\igc{1}{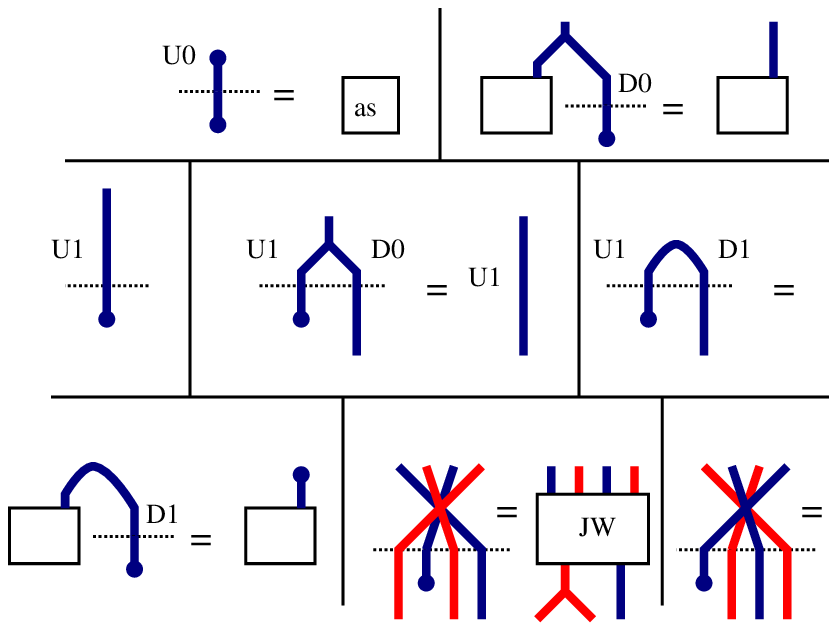}

In the first row, $\eb_l$ is U0 or D0 and is quickly taken care of. When $\eb_l$ is U1, what happens to the dot next can be a number of things. If the dot hits a rex move we end up on the
third row, which we will discuss shortly. Otherwise, we consult the second row. The dot either makes it all the way to the top (ending up in $I_{\un{w}}$) or it runs into a D0 or D1 and is
taken care of. Finally, if $\eb_l$ is D1, we end up in the first picture on the third row. The diagram on the right hand side clearly has width $\le M-1$, so it can be taken care of with
$(L_{M-1})$.

We have now proven this result except when the dot meets a rex move. We can induct on the number of $2m$-valent vertices in the rex move; the base case has just been done. Each term in the
Jones-Wenzl projector will have a dot on top. We can resolve that dot by induction to get an $\LL$ map with width $\le M-1$. After doing so, the remainder of the diagram has width $\le
M-1$, and we can apply $(L_{M-1})$ to finish the proof. \end{proof}

\begin{remark} There is an alternative proof. We can use Claim \ref{splitintwain} and the same style of argument as in Claim \ref{add2m} to reduce to the case where $({\un{y}},\eb)_{<
M-1}$ is a reduced expression with all U1's, and $\eb_M$ is D. Moreover, we can choose whatever rex moves we desire. Now the case by case analysis is easy, using some of the cases above.
\end{remark}

This concludes the proof of Claim \ref{goesbackdown}, and consequently of the fact that $(L_{M-1}) \implies (X_M)$. \end{proof}

\end{document}